\documentclass[11pt,reqno]{amsart}
 \usepackage{textcomp,multicol,enumerate,amsmath,amssymb,amsthm,latexsym,dsfont}
 \usepackage{amsfonts,mathtools,wasysym,enumerate}
 \usepackage{xcolor}
 
 \usepackage[T1]{fontenc}
 \usepackage{fourier}
 \usepackage{bbm}
 \usepackage{color}
 \usepackage{xcolor}
 \usepackage{tikz}
 \usepackage{fullpage} 
 \allowdisplaybreaks
 
 \usepackage{pgfplots}
\pgfplotsset{compat=1.15}
\usepackage{mathrsfs}
\usetikzlibrary{arrows}

 \usepackage[margin=2cm]{geometry}
 \newtheorem{theorem}{Theorem}[section]
 
 \newtheorem{lemma}[theorem]{Lemma}
  \newtheorem{claim}[theorem]{Claim}
 \newtheorem{conjecture}[theorem]{Conjecture}

 \newtheorem{question}[theorem]{Question}
 
 \usepackage{verbatim}
 \usepackage{mathrsfs}
 
 \theoremstyle{definition}

\begin{document}
\title{Component behaviour and excess of random bipartite graphs near the critical point$^{\dagger}$}
\author{Tuan Anh Do$^{\ddagger}$, Joshua Erde$^{\ddagger}$, Mihyun Kang$^{\ddagger}$, and Michael Missethan$^{\ddagger}$}

\thanks{$^{\dagger}$An extended abstract of this paper has been published in the Proceedings of the European Conference on Combinatorics, Graph Theory and Applications (Eurocomb 21), Trends in Mathematics, Vol. 14 Springer, p. 325-330, 2021.}

\thanks{$^{\ddagger}$ 
 	Institute of Discrete Mathematics, 
 	Graz University of Technology, 
 	Steyrergasse 30,
 	8010 Graz,
 	Austria,  
 	{\tt \{do,erde,kang,missethan\}@math.tugraz.at}.
 	Supported by Austrian Science Fund (FWF): I3747, W1230, P36131}

\maketitle
\begin{abstract} 
The binomial random bipartite graph $G(n,n,p)$ is the random graph formed by taking two partition classes of size $n$ and including each edge between them independently with probability $p$. It is known that this model exhibits a similar phase transition as that of the binomial random graph $G(n,p)$ as $p$ passes the critical point of $\frac{1}{n}$. We study the component structure of this model near to the critical point. We show that, as with $G(n,p)$, for an appropriate range of $p$ there is a unique `giant' component and we determine asymptotically its order and excess. We also give more precise results for the distribution of the number of components of a fixed order in this range of $p$. These results rely on new bounds for the number of bipartite graphs with a fixed number of vertices and edges, which we also derive.
\end{abstract}

\section{Introduction}
\subsection{Background and motivation}
It was shown by Erd\H{o}s and R\'{e}nyi \cite{E-R} that a `phase transition' occurs in the uniform random graph model $G(n,m)$ when $m$ is around $\frac{n}{2}$. Standard arguments on the asymptotic equivalence of the two models imply that a similar phenomenon occurs in the binomial random graph model $G(n,p)$ when $p$ is around $\frac{1}{n}$. More precisely, when $p=\frac{1-\epsilon}{n}$ for a fixed $\epsilon>0$, with high probability\footnote{With probability tending to one as $n \to \infty$.} (whp for short) every component of $G(n,p)$ has order at most $O(\log n)$; when $p=\frac{1}{n}$, whp the order of the largest component is $\Theta\left(n^\frac{2}{3}\right)$; and when $p=\frac{1+\epsilon}{n}$, whp $G(n,p)$ contains a unique `giant component' $L_1\left(G(n,p)\right)$ of order $\Omega(n)$. 

 Whilst it may seem at first that the component behaviour of the model $G(n,p)$ exhibits quite a sharp `jump' at this point, subsequent investigations, notably by Bollob\'{a}s \cite{Bollobas} and {\L}uczak \cite{Luczak}, showed that in fact, if one chooses the correct parameterisation for $p$, this change can be seen to happen quite smoothly. In particular, {\L}uczak's work implies the following result in the \emph{weakly supercritical regime}. Throughout the paper let $L_i(G)$ denote the $i$th largest component of a graph $G$ for $i \in \mathbb{N}$. We use the standard Landau notation for asymptotic orders.
\begin{theorem}[\cite{Luczak}]\label{t:giantLuczak}
Let $\epsilon=\epsilon(n)>0$ be such that $\epsilon^3 n \rightarrow \infty$ and $\epsilon=o(1)$, and let $p=\frac{1+\epsilon}{n}$. Then whp 
\[
\left|L_1\left(G(n,p)\right)\right| = (1+o(1)) 2\epsilon n \qquad \text{and} \qquad \left|L_2\left(G(n,p)\right)\right| \leq n^{\frac{2}{3}}.
\]
\end{theorem}

Furthermore, {\L}uczak's work allowed him to give a precise estimate for the excess of $L_1\left(G(n,p)\right)$ (the \emph{excess} of a connected graph is the difference between the number of edges and vertices).
The excess is in some way a broad measure of the complexity of the giant component, determining its density, which has important consequences, for example in terms of the length of the longest cycle in (see for example \cite{Luczakcycle}), or the genus of the giant component (see for example \cite{Kang}).

\begin{theorem}[\cite{Luczak}]\label{t:excessLuczak}
Let $\epsilon=\epsilon(n)>0$ be such that $\epsilon^3 n \rightarrow \infty$ and $\epsilon=o(1)$, and let $p=\frac{1+\epsilon}{n}$. Then whp
\[
\text{excess}\left(L_1\left(G(n,p)\right)\right) = (1+o(1)) \frac{2}{3}\epsilon^3 n.
\]
\end{theorem}

{\L}uczak also gave a finer picture of the distribution of the components in $G(n,p)$ in the \textit{weakly subcritical} and weakly supercritical regimes. In what follows, a \emph{tree, unicyclic, and complex} component is a component which has no, exactly one and more than one cycle, respectively.

\begin{theorem}[\cite{Luczak}]\label{t:finerLuczak}
Let $\epsilon = \epsilon(n)$ be such that $|\epsilon|^3 n  \rightarrow \infty$ and $\epsilon = o(1)$, let $p=\frac{1 + \epsilon}{n}$, let
$
\delta = \epsilon - \log(1+\epsilon)
$, and
let $\alpha=\alpha(n) > 0$ be an arbitrary function. Then the following hold in $G(n,p)$
\begin{enumerate}[(i)]
\item\label{i:treeLuczak}
With probability $1-e^{-\Omega(\alpha)}$ there are no tree components of order larger than
$$
\frac{1}{\delta} \left( \log\left( |\epsilon|^3 n\right) - \frac{5}{2} \log \log \left( |\epsilon|^3 n\right) + \alpha \right).
$$
\item\label{i:unicyclicLuczak}
With probability $1-e^{-\Omega(\alpha)}$ there are no unicyclic components of order larger than $\frac{\alpha}{\delta}$.
\item\label{i:complexLuczak}
If $\epsilon < 0$, then whp there are no complex components. 
\item\label{i:complex2Luczak}
If $\epsilon>0$, then with probability $1 - O\left(\left(\epsilon^3 n \right)^{-1}\right)$ there are no complex components of order smaller than $n^{\frac{2}{3}}$. 
\end{enumerate}
\end{theorem}
In this paper we investigate similar questions about the component structure of a different random graph model, the \emph{binomial random bipartite graph} $G(n,n,p)$, near to its critical point. The binomial random bipartite graph $G(n_1,n_2,p)$ is the random graph given by taking two partition classes $N_1$ and $N_2$ of sizes $n_1$ and $n_2$, respectively, and including each edge between $N_1$ and $N_2$ independently with probability $p$. For simplicity, we restrict our attention to the case where $n_1=n_2$. It is possible that similar techniques will work as long as the ratio $\frac{n_1}{n_2} = \Theta(1)$ is a fixed constant.

As in the case of $G(n,p)$, it is known, see for example \cite{Johansson}, that when $p=\frac{1-\epsilon}{n}$ for a fixed $\epsilon>0$, whp every component of $G(n,n,p)$ has order at most $O(\log n)$, and when $p=\frac{1+\epsilon}{n}$, whp $G(n,n,p)$ contains a unique `giant component' $L_1\left(G(n,n,p)\right)$ of order $\Omega(n)$. Hence, a phase transition occurs at $p=\frac{1}{n}$, as in $G(n,p)$.

There has been some interest in this model recently: Johannson \cite{Johansson} determined the critical point as described above in the general $G(n_1,n_2,p)$ model, Jing and Mohar \cite{Mohar} determined the genus of $G(n_1,n_2,p)$ in the dense regime, and Do, Erde and Kang \cite{Do} determined the genus of $G(n_1,n_2,p)$ in the sparse regime. 

This model can also be considered as a special case of the \emph{inhomogeneous random graphs} studied by Bollob\'{a}s, Janson and Riordan \cite{Bollobasphase}, who studied the phase transition in this much broader model. Whilst their results do not apply in the weakly supercritical regime, this regime was studied for a particular model of inhomogeneous random graphs, which again generalises the bipartite binomial random graph, namely the \emph{multi-type binomial random graph}, by Kang, Koch and Pach\'{o}n \cite{Kangmultitype}. In particular, it follows from their work that in the weakly supercritical regime there is a unique giant component, and they determine asymptotically its order.
\begin{theorem}[\cite{Kangmultitype}]\label{t:Kang}
Let $\epsilon=\epsilon(n)>0$ be such that $\epsilon^3 n \rightarrow \infty$ and $\epsilon=o(1)$, let $p=\frac{1+\epsilon}{n}$, and let $L_i = L_i\left(G(n,n,p)\right)$ for $i=1,2$. Then whp 
\begin{equation*}\label{e:giant}
|L_1\cap N_1| = (1+o(1)) 2\epsilon n \quad \text{and} \quad |L_1\cap N_2| = (1+o(1)) 2 \epsilon n.
\end{equation*} 
Furthermore,  whp  $|L_2| = o(\epsilon n)$.
\end{theorem}
In this paper we extend and strengthen the work in \cite{Johansson,Kangmultitype} on the component structure of $G(n,n,p)$ in the weakly supercritical regime.

\subsection{Main results}
 In this paper we prove the following analogues of Theorems \ref{t:giantLuczak}--\ref{t:finerLuczak} in the binomial random bipartite graph model.

Our first main result determines the existence and asymptotic order of the `giant' component in $G(n,n,p)$ near to the critical point. 
\begin{theorem}\label{t:giant}
Let $\epsilon=\epsilon(n)>0$ be such that $\epsilon^3 n \rightarrow \infty$ and $\epsilon=o(1)$, let $\epsilon'$ be defined as the unique positive solution to $(1-\epsilon')e^{\epsilon'} = (1+\epsilon)e^{-\epsilon}$, let $p=\frac{1+\epsilon}{n}$, and let $L_i = L_i\left(G(n,n,p)\right)$ for $i=1,2$. Then with probability $1-O\left(\left(\epsilon^3 n\right)^{-\frac{1}{6}}\right)$ we have
\[
\left| \left|L_1\right|- \frac{2(\epsilon + \epsilon')}{1+\epsilon}n \right| < \frac{1}{50}n^\frac{2}{3} \qquad \text{ and} \qquad
\left| 	L_2 \right|  \leq n^{\frac{2}{3}}.
\]
Furthermore, with probability $1-O\left(\left(\epsilon^3 n\right)^{-\frac{1}{6}}\right)$ we have that 
\[
|L_1 \cap N_1| = \left(1 \pm 2\sqrt{\epsilon}\right) |L_1 \cap N_2|.
\]
\end{theorem}
Note that $\epsilon'=\epsilon-\frac{2}{3}\epsilon^2+O(\epsilon^3)$. Hence, Theorem \ref{t:giant} gives a more precise bound on the order of $L_1$ than Theorem \ref{t:Kang}, as well as determining more precisely the distribution of the vertices of $L_1$ between the partition classes, and giving a better bound on the order of the second largest component. Moreover, with the help of this increased accuracy, we are able to determine asymptotically the excess of the giant component $L_1$. 
\begin{theorem}\label{t:excess}
Let $\epsilon=\epsilon(n)>0$ be such that $\epsilon^3 n \rightarrow \infty$ and $\epsilon=o(1)$, and let $p=\frac{1+\epsilon}{n}$. Then whp
\begin{equation*}\label{eq:excess}
\text{excess}\left(L_1\left(G(n,n,p)\right)\right) = (1+o(1)) \frac{4}{3}\epsilon^3 n.
\end{equation*} 
\end{theorem}

In addition, we can give a much more precise picture of the component structure of $G(n,n,p)$ near to the critical point in both the weakly subcritical and weakly supercritical regime. In what follows, let us write
\begin{equation}\label{e:definedelta}
\delta = \epsilon - \log(1+\epsilon).
\end{equation} 

Firstly, for the tree components, we show that whp there are no tree components of order significantly larger than $\frac{1}{\delta}\left(\log \left(|\epsilon|^3n\right) - \frac{5}{2}\log\log\left( |\epsilon|^3n\right)\right)$. Moreover, we show that the number of tree components of order around this tends to a Poisson distribution.

\begin{theorem}\label{t:trees} 
Let $\epsilon = \epsilon(n)$ be such that $|\epsilon|^3 n  \rightarrow \infty$ and $\epsilon = o(1)$, and let $p=\frac{1 + \epsilon}{n}$. 
\begin{enumerate}[(i)]
\item\label{i:treecomp}
Given $r_1,r_2\in \mathbb R^+$ with $r_1<r_2$ let $Y_{r_1,r_2}$ denote the number of tree components in $G(n,n,p)$ of orders between 
\[
\frac{1}{\delta}\left( \log\left( |\epsilon|^3 n\right) - \frac{5}{2} \log \log \left( |\epsilon|^3 n\right) + r_1\right) \,\,\,\,\,
\text{and} \,\,\,\,\,
\frac{1}{\delta}\left( \log\left( |\epsilon|^3 n\right) - \frac{5}{2} \log \log \left( |\epsilon|^3 n\right) + r_2\right),
\]
where $\delta$ is as in~\eqref{e:definedelta} and let  $\lambda=\lambda(r_1,r_2):= \frac{1}{\sqrt{\pi}}\left(e^{-r_1}-e^{-r_2}\right).$
Then $Y_{r_1,r_2}$ converges in distribution to $Po(\lambda)$.
\item\label{i:treecomp2}
With probability $1-e^{-\Omega(\alpha)}$, $G(n,n,p)$ contains no tree components of order larger than
\[
\frac{1}{\delta} \left( \log\left( |\epsilon|^3 n\right) - \frac{5}{2} \log \log \left( |\epsilon|^3 n\right) + \alpha \right)
\]
for any function $\alpha=\alpha(n) > 0$.
\end{enumerate}	
\end{theorem}

Secondly, for the unicyclic components, we show that whp there are no unicyclic components of order significantly larger than $\frac{1}{\delta}$, and moreover, that the number of unicyclic components of order around this again tends to a Poisson distribution.
\begin{theorem}\label{t:unicyclic}
Let  $\epsilon = \epsilon(n)$ be such that $|\epsilon|^3 n  \rightarrow \infty$ and $\epsilon = o(1)$, and let $p=\frac{1 + \epsilon}{n}$. 
\begin{enumerate}[(i)]
\item\label{i:unicomp}
Given $u_1,u_2\in \mathbb R^+$ with $u_1<u_2$ let $Z_{u_1,u_2}$ denote 
the number of unicyclic components in $G(n,n,p)$ of orders between
\[
\frac{u_1}{\delta} \,\,\,\text{    and    }\,\,\,   \frac{u_2}{\delta}, 
\]
where $\delta$ is as in~\eqref{e:definedelta} and let $\nu=\nu(u_1,u_2):=\frac{1}{2}\int_{u_1}^{u_2}\frac{\exp(-t)}{t}dt.$
Then  $Z_{u_1,u_2}$ converges in distribution to $Po(\nu)$.
\item\label{i:unicomp2}
With probability $1-e^{-\Omega(\alpha)}$, $G(n,n,p)$ contains no unicyclic components of order larger than $\frac{\alpha}{\delta}$ for any function $\alpha=\alpha(n) > 1$.
\end{enumerate}	
\end{theorem}

Finally, we show that there are whp no complex components of order at most $n^{\frac{2}{3}}$, and in fact no complex components at all in the weakly subcritical regime.

\begin{theorem}\label{t:complex}
Let $\epsilon = \epsilon(n)$ be such that $|\epsilon|^3 n  \rightarrow \infty$ and $\epsilon = o(1)$, and let $p=\frac{1 + \epsilon}{n}$. 
\begin{enumerate}[(i)]
\item\label{i:complexcomp}
If $\epsilon < 0$, then with probability $1 - O\left(\left(|\epsilon|^3 n\right)^{-1}\right)$, $G(n,n,p)$ contains  no complex components.
\item\label{i:complexcomp2}
If $\epsilon > 0$, then with probability $1 - O\left(\left(\epsilon^3 n\right)^{-1}\right)$, $G(n,n,p)$ contains no complex components of order at most $n^{\frac{2}{3}}$.
\end{enumerate}		
\end{theorem}

\subsection{Key proof ideas} 
As opposed to previous results concerning the phase transition in the binomial random bipartite graph, such as \cite{Johansson}, \cite{Kangmultitype} and \cite{Bollobasphase}, which analyse this model by comparison to branching processes, our approach is at heart based on enumerative methods, following the work of Bollob\'{a}s \cite{Bollobas} and {\L}uczak \cite{Luczak}. That is, we first derive estimates for the number of connected bipartite graphs with a fixed number of vertices and edges, and use these to bound the expectation, and higher order moments, of the number of components of various types in $G(n,n,p)$. This will allow us to describe the distribution of the \emph{small} components in $G(n,n,p)$, and in particular Theorems \ref{t:trees}--\ref{t:complex} will follow from such considerations. Furthermore, we can also bound quite precisely the number of vertices contained in \emph{large} components in $G(n,n,p)$, those of order at least $n^{\frac{2}{3}}$. It can then be shown using a standard sprinkling argument that whp there is a unique large component $L_1$ containing all these vertices. Given the order of $L_1$, we can again use these enumerative estimates to give a weak bound on its excess, which we can then bootstrap to an asymptotically tight bound via a multi-round exposure argument. This turns out to be quite a delicate argument, and in particular we make use of a correlation inequality of Harris. The main difficulties here, as opposed to the case of $G(n,p)$, come from the fact that the components of a fixed order can be split in various different ways across the partition classes, making the combinatorial expressions for the expected number of such components much harder to estimate or evaluate.

In order to derive these estimates for the number of connected bipartite graphs with a fixed number of vertices and edges we will use some standard enumerative tools, as well as the so-called \emph{core and kernel} method used by Bollob\'{a}s \cite{Bollobassparse} and {\L}uczak \cite{Luczak}. We will find that it is much easier to count the bipartite graphs whose partition classes have relatively equal sizes, which we call \emph{balanced}, and in this case we obtain effective bounds. Since these enumerative results translate directly into bounds on the moments of the number of components with a fixed number of vertices and edges in $G(n,n,p)$, we will often have to split such calculations into two parts depending on whether these components are balanced or not. In the latter case it is then necessary to obtain tighter probabilistic bounds to account for the weaker enumerative bounds.

The benefit in working directly with these enumerative results is in the increased accuracy, allowing for much finer control over the structure of $G(n,n,p)$ in the weakly supercritical regime. For this reason, these estimates may be useful in order to apply similar methods to study the structure of $G(n,n,p)$ in this regime in more detail. For example, in the case of $G(n,p)$, {\L}uczak \cite{Luczakcycle} used similar ideas to describe the distribution of cycles in $G(n,p)$ in this regime, and more recently, using some of these ideas, Dowden, Kang and Krivelevich \cite{Kang} were able to determine asymptotically the genus of $G(n,p)$ in this regime. It is possible that similar ideas could be applied to $G(n,n,p)$, for example to study the distribution of cycles, the length of the longest cycle, or the genus in this model.

\subsection{Overview of the paper}
The rest of the article is organised as follows. In Section \ref{pre}, we collect some preliminary results which are used later in the paper. In Section \ref{s:compt} we derive bounds for the expected number of components of $G(n,n,p)$ with a fixed order and excess, which form the foundation of many of the calculations in this paper. These bounds depend on good estimates for the number of bipartite graphs with a fixed number of vertices and edges, whose proofs we give in Section \ref{s:counting}. In Section \ref{s:components}, we use these estimates to study the distribution of components in $G(n,n,p)$ and prove Theorems \ref{t:trees}--\ref{t:complex}. Then, in Section \ref{s:giant}, using the previous results, we investigate the size of the largest components and prove Theorem \ref{t:giant}. Using this, we then determine the excess of the giant component and prove Theorem \ref{t:excess} in Section \ref{s:excess}. Finally, in Section \ref{s:Discuss}, we discuss possible extensions of our results, formulate a conjecture, and give some open problems.

\section{Preliminaries}\label{pre}
Unless stated otherwise, all the asymptotics in this paper are taken as $n \rightarrow \infty$. In particular, we write 
\[
f(n) \approx g(n) \text{ if } f = (1+o(1))g; \qquad f(n) \lesssim g(n) \text{ if } f \leq (1+o(1)) g;\text{ and } \qquad  f(n)\gtrsim g(n) \text{ if } f\geq (1+o(1))g.
\]
Furthermore, we write that $f(n) \gg g(n)$ if $f(n) \geq C g(n)$ for an implicit large constant $C$. We write $\mathbb{N}$ for the set of positive integers, so that in particular $0 \not\in \mathbb{N}$.

We will often need the following elementary estimates on the size of the falling factorial, which hold for any $i, n \in \mathbb{N}$ with $i\leq n$
\begin{align}\label{e:fallingfactorial}
		(n)_i:= \prod_{j=0}^{i-1} (n-j) = n^i\exp\left(-\frac{i(i-1)}{2n}-\frac{i(i-1)(2i-1)}{12n^2} + O\left(\frac{i^4}{n^3}\right)\right),
\end{align}
and also
\begin{align}\label{e:fallingfactorialbound}
		(n)_i&\leq  n^i\exp\left(-\frac{(i-1)^2}{2n}-\frac{i(i-1)(2i-1)}{12n^2} \right)\leq  n^i\exp\left(-\frac{(i-1)^2}{2n} \right).
\end{align}

The following result of Spencer
\cite{Spencer} is a useful tool for relating integrals and sums.
\begin{lemma}[{\cite[Theorem 4.3]{Spencer}}]\label{l:spencer}
Let $a < b$ be integers, let $f(x)$ be an integrable function in $[a-1,b+1]$, and let $S := \sum_{i=a}^b f(i)$ and $I := \int_a^b f(x) dx$. Let $M$ be such that $|f(x)| \leq M$ for all $x \in [a-1,b+1]$ and suppose that $[a-1,b+1]$ can be broken into at most $r$ intervals such that $f(x)$ is monotone on each. Then
\[
|S-I| \leq 6rM.
\]
\end{lemma}

Often, when calculating certain expected values, we will need an asymptotic expression for sums of the following form, whose proof we relegate to Appendix \ref{s: Appendix A}.
\begin{lemma}\label{l:difference}
Let $m\geq 0$ be constant and let $L=L(n)$ and $k=k(n)$  be such that $L+1 \leq k \leq n$, $L = \omega(1)$ and $k = o(n)$. Then
\[
S:=\sum_{d=-L}^{L}\frac{1}{(k^2-d^2)^m}\left(\frac{k-d}{k+d}\right)^{d}\exp\left(-\frac{d^2}{2n}\right) \approx \sqrt{\frac{\pi}{2}}k^{\frac{1}{2}-2m}.
\]
\end{lemma}

We will use the following Chernoff type bounds on the tail probabilities of the binomial distribution, see e.g., \cite[Appendix A]{Alon}.
\begin{lemma}\label{l:chernoff}
Let $n \in \mathbb{N}$, let $p \in [0,1]$, and let $X \sim \text{Bin}(n,p)$. Then for every positive $a$ with $a \leq \frac{np}{2}$,
\[
\mathbb{P}\left(\left|X -np \right| > a\right) < 2 \exp\left(-\frac{a^2}{4np} \right).
\]
\end{lemma}

We will also need to use the following correlation inequality, which follows from an inequality of Harris \cite{Harris}, which is itself a special case of the FKG-inequality, see for example \cite[Section, 6]{Alon}.
 
\begin{lemma}\label{l:Harris}
If $A$ is an increasing event and $B$ is a decreasing event of bipartite graphs, then in $G(n,n,p)$
\[
\mathbb{P}(A | B) \leq \mathbb{P}(A).
\]
\end{lemma}

Finally, we will also need the following lemma, which gives a useful criterion for when a sequence of random variables converges in distribution to a Poisson distribution.
 \begin{lemma}[\cite{Janson}]\label{l:Poisson}
 If $X_1,X_2,\cdots$ are random variables with finite 
moments such that $\mathbb{E}\left((X_n)_k\right)\to \lambda^k$ as $n\to \infty$ for every positive integer $k$, where $(X_n)_k$ is the $k$th factorial moment of $X_n$ and $\lambda \geq 0$ is a constant, then $X_n$ converges in distribution to $Po(\lambda)$.
\end{lemma}

\section{Component structure of $G(n,n,p)$}\label{s:compt} 
One of the main ways in which we derive information about the distribution of the components in $G(n,n,p)$ is by calculating various moments of the number of components with particular properties, and in particular the expected value.

Given $i,j \in \mathbb{N}$ and $\ell \in \mathbb{Z}$, let $X(i,j,\ell)$ denote the number of components in $G(n,n,p)$ with $i$ vertices in $N_1$, $j$ vertices in $N_2$, and $i+j+\ell$ edges. Letting $i+j=k$, we have
\begin{equation}\label{e:componentcount}
\mathbb{E}\left(X(i,j,\ell)\right) = \binom{n}{i}\binom{n}{j} C(i,j,\ell) p^{k + \ell} (1-p)^{kn -ij -k -\ell},
\end{equation}
where $C(i,j,\ell)$ is the number of connected bipartite graphs with $i$ vertices in one partition class, $j$ in the second, and $i+j + \ell$ many edges. Hence, in order to understand the quantities $\mathbb{E}\left(X(i,j,\ell)\right)$, it is important to know how the quantities $C(i,j,\ell)$ behave. 

In this section we state some bounds for $C(i,j,\ell)$, which we will prove later in Section \ref{s:counting}, and derive some consequences of these bounds, using \eqref{e:componentcount}, for the expected number of tree, unicyclic and complex components in $G(n,n,p)$.

The following estimates are useful to this end. Using the fact that $1+x = e^{x + O\left(x^2\right)}$ for any $x=o(1)$, we see that for any $i+j=k \leq n$, $c > 0$, and $\epsilon = o(1)$,
\begin{equation}\label{e:common}
   \left(1-\frac{1+\epsilon}{n}\right)^{kn-cij + O(k)}=\exp\left(-(1+\epsilon) k+\frac{(1+\epsilon)cij}{n} + O\left(\frac{k}{n}\right) \right).
\end{equation}

Throughout this section, unless stated otherwise, we let $\epsilon = \epsilon(n)$ be such that $|\epsilon|^3 n  \rightarrow \infty$ and $\epsilon = o(1)$, and let $p=\frac{1 + \epsilon}{n}$. We will also refer to $\delta$ as defined in~\eqref{e:definedelta}, i.e.,
\begin{equation*}
\delta = \epsilon - \log(1+\epsilon) \approx \frac{\epsilon^2}{2}.
\end{equation*}

\subsection{Tree components}
Let us write $\hat{C}(i,\ell)$ for the number of (not-necessarily bipartite) connected graphs with $i$ vertices and $i+\ell$ many edges. It is a classic result of Cayley that the number of trees on $i$ vertices, in other words $\hat{C}(i,-1)$, is $i^{i-2}$. The following result of Scoins \cite{Scoins} gives an analogue for bipartite trees.

\begin{theorem}[\cite{Scoins}]\label{t:scoins}
For any $i,j \in \mathbb{N}$ we have $C(i,j,-1) = i^{j-1}j^{i-1}$.
\end{theorem}

As a consequence, we can derive an asymptotic formula for the expected number of tree components in $G(n,n,p)$.

\begin{theorem} \label{t:treeexpectation}
For any $i=i(n), j = j(n) \in \mathbb{N}$ satisfying $k:=i+j\leq n$, we have 
\begin{equation}\label{e:treeexpectation}
\mathbb{E}\left(X(i,j,-1)\right) \approx \frac{n}{2\pi(ij)^{\frac{3}{2}}} e^{-\delta k} \left(\frac{i}{j}\right)^{j-i} \exp\left(-\frac{(i-j)^2}{2n}-\frac{i^3+j^3}{6n^2} +\frac{\epsilon ij}{n} + O\left(\frac{k}{n}\right) + O\left(\frac{i^4+j^4}{n^3}\right) \right).
\end{equation}
\end{theorem} 

\begin{proof}
By Theorem \ref{t:scoins} and \eqref{e:componentcount}, together with Stirling's formula, we have
\begin{align}
    \mathbb{E}\left(X(i,j,-1)\right) &= \binom{n}{i}\binom{n}{j} C(i,j,-1) p^{k -1} (1-p)^{kn -ij -k +1} \nonumber\\
     &=\frac{(n)_i}{i!}\frac{(n)_j}{j!}i^{j-1}j^{i-1}p^{k-1}(1-p)^{kn-ij-k+1} \label{e:treescount}\\
     &\approx\frac{e^k}{2\pi(ij)^{\frac{3}{2}}}\left(\frac{i}{j}\right)^{j-i}\frac{(n)_i(n)_j}{n^{k-1}} (1+\epsilon)^{k}\left(1-\frac{1+\epsilon}{n}\right)^{kn-ij-k+1}. \label{e:treesinitial}
 \end{align}
Hence, by \eqref{e:treesinitial}, \eqref{e:fallingfactorial} and \eqref{e:common}, we obtain
\begin{equation*}
\mathbb{E}\left(X(i,j,-1)\right) \approx \frac{n}{2\pi(ij)^{\frac{3}{2}}} e^{-\delta k} \left(\frac{i}{j}\right)^{j-i} \exp\left(-\frac{(i-j)^2}{2n}-\frac{i^3+j^3}{6n^2} +\frac{\epsilon ij}{n} + O\left(\frac{k}{n}\right) + O\left(\frac{i^4+j^4}{n^3}\right) \right).
\end{equation*}
\end{proof}

\subsection{Unicyclic components}
We will derive in Section \ref{s:counting} the following expression for the number of unicyclic connected bipartite graphs.
\begin{theorem}\label{t:unicyliccomponents}
For any $i,j \in \mathbb{N}$ we have 
\[
C(i,j,0) = \frac{1}{2}  i^{j-1} j^{i-1} \sum_{r=2}^{\min\{i,j\}} \frac{(i)_r(j)_r}{i^r j^r}\left(i + j-r\right),
\]
and so in particular, for any $i=i(n),j=j(n) \in \mathbb{N}$  satisfying $i,j \rightarrow \infty$ and $\frac{1}{2}\leq \frac{i}{j}\leq 2$ we have
\[
C(i,j,0)\approx \sqrt{\frac{\pi}{8}} \sqrt{i+j} i^{j-\frac{1}{2}} j^{i-\frac{1}{2}}.
\]
\end{theorem}

We note for comparison that it is known that 
 \[
 \hat{C}(i,0) \approx \sqrt{\frac{\pi}{8}}i^{i-\frac{1}{2}},
 \]
 see \cite[Corollary 5.19]{BollobasBook}. We can derive as a consequence an asymptotic formula for the expected number of unicyclic components in $G(n,n,p)$ which are appropriately balanced across the partition classes.

\begin{theorem}\label{t:unicyclicexpectation} 
For any $i=i(n), j = j(n) \in \mathbb{N}$ satisfying $i,j \rightarrow \infty$, $k:=i+j\leq n$, and $\frac{1}{2} \leq \frac{i}{j} \leq 2$, we have
\begin{equation}\label{e:unicyclicexpectation}
\mathbb{E}\left(X(i,j,0)\right) \approx \frac{\sqrt{k}}{4\sqrt{2 \pi}ij} e^{-\delta k} \left(\frac{i}{j}\right)^{j-i} \exp\left(-\frac{(i-j)^2}{2n}-\frac{i^3+j^3}{6n^2} +\frac{\epsilon ij}{n} + O\left(\frac{k}{n}\right) + O\left(\frac{i^4+j^4}{n^3}\right) \right).
\end{equation}
\end{theorem} 
\begin{proof}
By Theorem \ref{t:unicyliccomponents} and \eqref{e:componentcount}, together with Stirling's formula,  if $\frac{1}{2} \leq \frac{i}{j} \leq 2$, then
\begin{align}
    \mathbb{E}\left(X(i,j,0)\right) &= \binom{n}{i}\binom{n}{j} C(i,j,0) p^{k} (1-p)^{kn -ij -k} \approx \frac{(n)_i}{i!}\frac{(n)_j}{j!}\sqrt{\frac{\pi}{8}}\sqrt{k}i^{j-\frac{1}{2}}j^{i-\frac{1}{2}}p^{k}(1-p)^{kn-ij-k} \nonumber\\
     &\approx\frac{e^k\sqrt{ k}}{4\sqrt{2 \pi }ij} \left(\frac{i}{j}\right)^{j-i}\frac{(n)_i(n)_j}{n^k} (1+\epsilon)^{k}\left(1-\frac{1+\epsilon}{n}\right)^{kn-ij-k}. \label{e:unicyclicinitial}
 \end{align}
Hence, by \eqref{e:unicyclicinitial}, \eqref{e:fallingfactorial} and \eqref{e:common}, we get
\begin{equation*}
\mathbb{E}\left(X(i,j,0)\right) \approx \frac{\sqrt{k}}{4\sqrt{2 \pi}ij} e^{-\delta k} \left(\frac{i}{j}\right)^{j-i} \exp\left(-\frac{(i-j)^2}{2n}-\frac{i^3+j^3}{6n^2} +\frac{\epsilon ij}{n} +  O\left(\frac{k}{n}\right) + O\left(\frac{i^4+j^4}{n^3}\right)\right).
\end{equation*}
\end{proof}

\subsection{Complex components} 
In Section \ref{s:counting} we will also prove the following upper bound on the number of connected bipartite graphs with a fixed excess which are appropriately balanced across the partition classes.

\begin{theorem}\label{t:complexcomponents}
There is a constant $c>0$ such that for any $i,j,\ell \in \mathbb{N}$ with $\ell \leq ij - i -j$ and $\frac{1}{2}\leq \frac{i}{j}\leq 2$,
\[
C(i,j,\ell)\leq i^{j}j^{i}(i+j)^{\frac{3\ell-1}{2}}\left(\frac{c}{\ell}\right)^\frac{\ell}{2}.
\]
\end{theorem}

We note for comparison that it is known that there is an absolute constant $c$ such that
\begin{align}
\hat{C}(i,\ell) \leq c \ell^{-\frac{\ell}{2}}i^{i+\frac{3\ell-1}{2}},
\label{e:complexgeneral}
\end{align}
see \cite[Corollary 5.21]{BollobasBook}.

As before, using these bounds we can give an upper bound on the expected number of components with a fixed excess which are appropriately balanced across the partition classes.

\begin{theorem} \label{t:complexexpectation}
There is a constant $c>0$ such that for any $i=i(n),j=j(n),\ell=\ell(n) \in \mathbb{N}$ satisfying $\ell \leq ij-i-j$, $k:=i+j \leq n$, and $\frac{1}{2} \leq \frac{i}{j} \leq 2$, we have  
\begin{align}
\mathbb{E}\left(X(i,j,\ell)\right) 
&\leq \frac{1}{\sqrt{ijk}}  \left(\frac{i}{j}\right)^{j-i}  \left(\frac{ck^3}{\ell n^2}\right)^\frac{\ell}{2} \exp \left(-\delta k + \frac{\epsilon k^2}{4n}  -\frac{(i-j)^2}{2n} + O\left(\frac{k}{n}\right) + \ell \log (1+\epsilon) + \frac{\ell(1+\epsilon)}{n} \right). \label{e:complexexpectation}
\end{align}
\end{theorem} 

\begin{proof}
By Theorem \ref{t:complexcomponents} and \eqref{e:componentcount}, together with Stirling's formula, if $\frac{1}{2} \leq \frac{i}{j} \leq 2$ and $\ell \leq ij-i-j$, then there is an absolute constant $c$ such that,
\begin{align}
    \mathbb{E}\left(X(i,j,\ell)\right) &= \binom{n}{i}\binom{n}{j} C(i,j,\ell) p^{k+\ell} (1-p)^{kn -ij -k-\ell} \nonumber\\
     &\leq \frac{(n)_i}{i!}\frac{(n)_j}{j!}i^{j}j^{i}(i+j)^{\frac{3\ell-1}{2}}\left(\frac{c}{\ell}\right)^\frac{\ell}{2}p^{k+\ell}(1-p)^{kn-ij-k-\ell} \nonumber\\
     &\leq \frac{e^{k}}{\sqrt{ijk}} \left(\frac{i}{j}\right)^{j-i}\left(\frac{ck^3}{\ell n^2}\right)^\frac{\ell}{2}\frac{(n)_i(n)_j}{n^k} (1+\epsilon)^{k+\ell}\left(1-\frac{1+\epsilon}{n}\right)^{kn-ij-k-\ell}. \label{e:complexinitial}
 \end{align}
 Hence, by \eqref{e:complexinitial}, \eqref{e:fallingfactorialbound} and \eqref{e:common}, we see that
\begin{align*}
\mathbb{E}\left(X(i,j,\ell)\right) &\leq \frac{1}{\sqrt{ijk}} \left(\frac{i}{j}\right)^{j-i}  \left(\frac{ck^3}{\ell n^2}\right)^\frac{\ell}{2} \exp \left(-\delta k + \frac{\epsilon ij}{n}  -\frac{(i-j)^2}{2n} + O\left(\frac{k}{n}\right) + \ell \log (1+\epsilon) + \frac{\ell(1+\epsilon)}{n} \right) \nonumber\\
&\leq \frac{1}{\sqrt{ijk}} \left(\frac{i}{j}\right)^{j-i}  \left(\frac{ck^3}{\ell n^2}\right)^\frac{\ell}{2} \exp \left(-\delta k + \frac{\epsilon k^2}{4n}  -\frac{(i-j)^2}{2n} + O\left(\frac{k}{n}\right) + \ell \log (1+\epsilon) + \frac{\ell(1+\epsilon)}{n} \right).
\end{align*}
\end{proof} 

\subsection{More about components}
Since we only have good estimates for $C(i,j,\ell)$ when $i$ and $j$ are comparable in size, it will be useful to show that the expected number of components of a given excess and order is dominated by the contribution from those which are `evenly spread' across the partition classes, and we should perhaps expect by the symmetry in the model that this is the case for most components. For the most part, we are able to get away with considering a relatively weak notion of `evenly spread'.

We say a component $C$ of $G(n,n,p)$ is \emph{balanced} if $|C \cap N_1| \leq 2 |C \cap N_2|$ and $|C \cap N_2| \leq 2|C \cap N_1|$, and \emph{unbalanced} otherwise. The following lemma will be useful for simplifying certain calculations, which roughly says that we do not expect there to be any large unbalanced components in $G(n,n,p)$.

\begin{lemma}\label{l:balanced} 
Let $\epsilon = \epsilon(n)>0$ be such that $\epsilon^3n \rightarrow \infty$ and $\epsilon=o(1)$, let $p=\frac{1 + \epsilon}{n}$, and let $\alpha=\alpha(n) \rightarrow \infty$ be an increasing function. 
\begin{enumerate}[(i)]
\item\label{i:unbalanced} With probability $1 - O\left(n^{-1} \right)$, $G(n,n,p)$ contains no unbalanced components of order $\geq 2000 \log n$.
\item\label{i:non-treeunbalanced} With probability $1 - e^{-\Omega(\alpha)}$, $G(n,n,p)$ contains no unbalanced non-tree components of order $\geq \alpha$.
\item\label{i:unbalancedcomplex} With probability $1 - O\left(n^{-1} \right)$, $G(n,n,p)$ contains no unbalanced complex components.
 \end{enumerate}	
\end{lemma}

\begin{proof}
Every unbalanced component of order $k$ with excess at least $\ell$ must contain a spanning tree (of order $k$) together with $\ell+1$ extra edges which is otherwise disconnected from the rest of the graph. Hence, $G(n,n,p)$ contains a component of order $k$ and excess at least $\ell$ if and only if $G(n,n,p)$ contains such a substructure. Let us denote by $Y(k,\ell)$ the number of such substructures. It follows that if $Y(k,\ell)=0$, then $G(n,n,p)$ contains no components of order $k$ with excess at least $\ell$.

In order to count the expected size of $Y(k,\ell)$, we note that we can specify such a substructure $S$ by choosing $i$ vertices in the first partition class and $j$ vertices in the second, such that that $i+j=k$ and either $j \geq 2i$ or $i \geq 2j$, choosing one of the $i^{j-1}j^{i-1}$ possible bipartite spanning trees on these vertices, and then choosing one of the at most $\binom{ij}{\ell+1}$ possible sets of $\ell+1$ extra edges. Note that the number of non-edges from these $k$ vertices to the other vertices in $G(n,n,p)$ is $i(n-j) + j(n-i) = kn - 2ij$ (see Figure \ref{f:excessell}).

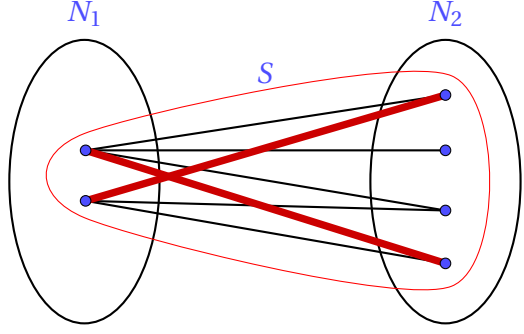
\begin{figure}
    \centering
\definecolor{ccqqqq}{rgb}{0.8,0.,0.}
\definecolor{ududff}{rgb}{0.30196078431372547,0.30196078431372547,1.}
\begin{tikzpicture}[line cap=round,line join=round,>=triangle 45,x=1.0cm,y=1.0cm, scale=0.8]
\clip(3.5,2.6) rectangle (12.5,8);
\draw [rotate around={90.:(5.,5.)},line width=0.8pt] (5.,5.) ellipse (2.357329608818574cm and 1.2477992164658225cm);
\draw [rotate around={90.:(11.,5.)},line width=0.8pt] (11.,5.) ellipse (2.357329608818563cm and 1.2477992164658167cm);
\draw [line width=0.8pt] (11.,3.64)-- (5.02,4.68);
\draw [line width=0.8pt] (5.02,4.68)-- (11.,4.52);
\draw [line width=0.8pt] (11.,4.52)-- (5.04,5.52);
\draw [line width=0.8pt] (5.02,5.52)-- (11.,5.52);
\draw [line width=0.8pt] (5.02,5.52)-- (11.,6.44);
\draw [line width=2.8pt,color=ccqqqq] (5.02,4.68)-- (11.,6.44);
\draw [line width=2.8pt,color=ccqqqq] (5.02,5.52)-- (11.,3.64);
\begin{scriptsize}
\draw [fill=ududff] (5.02,5.52) circle (2.5pt);
\draw [fill=ududff] (5.02,4.68) circle (2.5pt);
\draw [fill=ududff] (11.,6.44) circle (2.5pt);
\draw [fill=ududff] (11.,5.52) circle (2.5pt);
\draw [fill=ududff] (11.,4.52) circle (2.5pt);
\draw [fill=ududff] (11.,3.64) circle (2.5pt);
\draw[color=ududff] (5,7.8) node {\large $N_1$};
\draw[color=ududff] (11,7.8) node {\large $N_2$};
\draw[color=ududff] (8,6.8) node {\large $S$};
\draw [red] plot [smooth cycle] coordinates {(5,5.8) (5,4.4) (11.1,3.26) (11.1,6.76)};
\end{scriptsize}
\end{tikzpicture}
    \caption{A substructure $S$ (in the proof of Lemma \ref{l:balanced}) with $i=2$ vertices in $N_1$ and $j=4$ vertices in $N_2$ containing a spanning tree (whose edges are drawn with thin edges) and $\ell=2$ excess edges (which are drawn with thick edges), where none of the $kn-2ij$ edges from $V(S)$ to the rest of the graph are in $G(n,n,p)$.}
    \label{f:excessell}
\end{figure}

It follows that we can bound
\begin{align*}
		    \mathbb{E}\left(Y(k,\ell)\right) \leq \sum_{(i,j) \in U_k}\binom{n}{i}\binom{n}{j} i^{j-1}j^{i-1} \binom{ij}{\ell+1} p^{k+\ell}(1-p)^{kn-2ij}, 
		    \end{align*}
where $U_k = \left\{(i,j) \in \mathbb{N}^2 \colon i+j=k \text{ and } i \geq 2j \text{ or } j \geq 2i\right\}$, and we note that $|U_k|\leq k$.

Therefore, using \eqref{e:fallingfactorialbound}, \eqref{e:common} and Stirling's approximation, we can bound the expected number by
		    
\begin{align*}
		   \mathbb{E}(Y(k,\ell)) &\leq  n^{-\ell}e^{-\delta k } \sum_{(i,j) \in U_k}\frac{(ij)^{\ell - \frac{1}{2}}}{2\pi}\left(\frac{i}{j}\right)^{j-i} \exp\left(-\frac{i^2+j^2}{2n} +\frac{(1+\epsilon)2ij}{n}  + O\left(\frac{k}{n}\right) \right) \\
		    &\leq n^{-\ell} \sum_{(i,j) \in U_k} (ij)^{\ell - \frac{1}{2}} \left(\frac{i}{j}\right)^{j-i} \exp \left(\frac{ij(1+2\epsilon)}{n} + O(1) \right),
		   	\end{align*}
since $ e^{x} <1$ for $x<0$, $-i^2-j^2 + 2ij < 0$ and $i,j \leq k \leq n$.
		   	
However, if $j \geq 2i$ and $i+j=k$, then $j \geq \frac{2k}{3}$ and so $j-i \geq \frac{j}{2} \geq \frac{k}{3}$, and $ij\leq \frac{2k^2}{9}$. It follows that $\left(\frac{i}{j}\right)^{j-i} \leq \left(\frac{1}{2}\right)^{\frac{k}{3}}$. A similar calculation holds if $i \geq 2j$.  Hence the expected number of such substructures is at most
		\begin{align*}
	    \mathbb{E}(Y(k,\ell)) \leq n^{-\ell}\exp \left(\frac{2(1+2\epsilon)k^2}{9n} - \frac{k \log 2  }{3}+ O(1) \right)\sum_{(i,j) \in U_k}  (ij)^{\ell - \frac{1}{2}}
	    \leq n^{-\ell}e^{-\frac{k}{1000}}\sum_{(i,j) \in U_k}  (ij)^{\ell - \frac{1}{2}},
		\end{align*}
		since $\frac{2(1+2\epsilon)k^2}{9n} - \frac{k \log 2  }{3} +O(1)\leq k\left(\frac{2(1+2\epsilon)}{9} - \frac{ \log 2  }{3}\right) + O(1) \leq -\frac{k}{1000}$  when $\epsilon$ is sufficiently small.
		
Hence, if we let $Y_{\geq r}(\ell) = \sum_{k\geq r} Y(k,\ell)$, then with $r=2000 \log n$
	    \[
	    \mathbb{E}\left(Y_{\geq r}(-1)\right) \leq n \sum_{k\geq r} e^{-\frac{k}{1000}} \sum_{(i,j) \in U_k}  (ij)^{-\frac{3}{2}} \leq n \sum_{k\geq r} e^{-\frac{k}{1000}} =  O\left( \frac{1}{n} \right).
	    \]
Hence, by Markov's inequality, with probability $1 - O\left( n^{-1} \right)$, $Y_{\geq r}(-1)=0$ and in particular there are no unbalanced components of order at least $2000 \log n$.
		
Similarly, if $\alpha = \alpha(n) \rightarrow \infty$ is an increasing function, then
\[
\mathbb{E}\left(Y_{\geq \alpha}(0)\right) \leq \sum_{k\geq \alpha} e^{-\frac{k}{1000}} \sum_{(i,j) \in U_k}  (ij)^{-\frac{1}{2}} \leq  \sum_{k\geq \alpha} \sqrt{k} e^{-\frac{k}{1000}} =  O\left( e^{-\frac{\alpha}{2000}}\right),
\]
and so, again by Markov's inequality, with probability $1 -  e^{-\Omega(\alpha)}$, there are no unbalanced components of order at least $\alpha$ with excess greater than zero, and so in particular no unicyclic components of order at least $\alpha$.

Finally, we see that
\[
\mathbb{E}(Y_{\geq 1}(1)) \leq \frac{1}{n}\sum_{k\geq 1} e^{-\frac{k}{1000}} \sum_{(i,j) \in U_k} (ij)^{\frac{1}{2}} \leq  \frac{1}{n}\sum_{k\geq 1} k^2 e^{-\frac{k}{1000}} = O \left( \frac{1}{n} \right),
\]
and so as before with probability $1 - O \left( n^{-1}\right)$ there are no unbalanced complex components.
\end{proof} 

For most applications the rather coarse notion of balanced is enough for our purposes, but in one case we will need to restrict our attention to components which are much more evenly distributed over the partition classes. We say a component $C$ of $G(n,n,p)$ is \emph{$\epsilon$-uniform} if $\big||C \cap N_1| - |C \cap N_2|\big| < \epsilon^{\frac{1}{4}}\sqrt{n}$.

\begin{lemma}\label{l:uniform}
Let $\epsilon = \epsilon(n) > 0$ be such that $\epsilon^3 n \rightarrow \infty$ and $\epsilon = o (1)$, and let $p=\frac{1 + \epsilon}{n}$. Then with probability $1- o\left( n^{-1}\right)$, $G(n,n,p)$ contains no non-$\epsilon$-uniform tree components of order at most $n^\frac{2}{3}$.
\end{lemma}
\begin{proof}
As in the previous lemma, let us write $U_k = \left\{(i,j) \in \mathbb{N}^2 \colon i+j=k, |i-j| \geq \epsilon^{\frac{1}{4}}\sqrt{ n}\right\}$ for the pairs $(i,j)$ representing non-$\epsilon$-uniform components. Note that, if $(i,j) \in U_k$ and $k \leq n^{\frac{2}{3}}$, then
\[
\left(\frac{i}{j}\right)^{j-i} \leq \left(1-\frac{\epsilon^{\frac{1}{4}}\sqrt{n}}{ n^\frac{2}{3}}\right)^{\epsilon^{\frac{1}{4}}\sqrt{n}} \leq e^{-\sqrt{\epsilon} n^{\frac{1}{3}}}.
\]

Then, using \eqref{e:treeexpectation}, we can bound the expected number of non-$\epsilon$-uniform tree components of order at most $n^\frac{2}{3}$ by
      \begin{align*}
         \sum_{k=1}^{n^\frac{2}{3}} \sum_{(i,j) \in U_k} \mathbb{E}\left(X(i,j,-1)\right)
         &\leq \sum_{k=1}^{n^\frac{2}{3}} \sum_{(i,j) \in U_k} \frac{n}{2\pi(ij)^{\frac{3}{2}}} \left(\frac{i}{j}\right)^{j-i} \exp\left(\frac{\epsilon ij}{n} + o(1) \right)\\
          &\leq \sum_{k=1}^{n^\frac{2}{3}}n e^{\frac{\epsilon k^2}{4n}-\sqrt{\epsilon} n^{\frac{1}{3}}}\sum_{(i,j) \in U_k} \frac{1}{(ij)^\frac{3}{2}}\leq \sum_{k=1}^{n^\frac{2}{3}}\frac{n}{k^{\frac{1}{2}}}e^{\frac{\epsilon k^2}{4n}-\sqrt{\epsilon} n^{\frac{1}{3}}}.
          \end{align*}

However, since $k \leq n^{\frac{2}{3}}$ and $\epsilon^3 n \rightarrow \infty$, it follows that
\[
\frac{\epsilon k^2}{4n}-\sqrt{\epsilon} n^{\frac{1}{3}} = -\Omega\left( \sqrt{\epsilon} n^{\frac{1}{3}}\right) \leq - n^{\frac{1}{6}}.
\]
It follows that the expected number of non-$\epsilon$-uniform tree components of order at most $n^\frac{2}{3}$ is at most
\[
ne^{-n^{\frac{1}{6}}}\sum_{k=1}^{n^\frac{2}{3}}\frac{1}{k^{\frac{1}{2}}} \leq n^\frac{4}{3} e^{-n^\frac{1}{6}} =o\left(n^{-1}\right).
\]
Hence, the result follows by Markov's inequality.
\end{proof} 

It will also be useful to have a bound on the variance of the number of vertices in $\epsilon-$uniform tree components with small order, which is given by the following lemma, whose proof is given in Appendix \ref{s:Append B}.
\begin{lemma}\label{l:Var}
Let $\epsilon = \epsilon(n) >0$ be such that $\epsilon^3n \rightarrow \infty$ and $\epsilon = o(1)$, and let $p = \frac{1+\epsilon}{n}$. Given $\Tilde{k},a \in \mathbb N$, set  $Z_a=\sum_{k=1}^{\Tilde{k}}k^a Z(k)$ where $Z(k)$ is the number of $\epsilon$-uniform tree components of order $k$ in $G(n,n,p)$. If $\Tilde{k}\leq n^\frac{2}{3}$ and $\frac{3 \epsilon \Tilde{k}^2}{n}\leq 1$, then $\emph{Var}(Z_1) = O\left(\frac{n}{\epsilon}\right).$
\end{lemma}

\section{A finer look at component structure  of $G(n,n,p)$}\label{s:Finer}
Using the bounds from Section \ref{s:compt} on the expected number of components with a fixed order and excess, we can describe more precisely the component structure of $G(n,n,p)$.

\subsection{Distribution of the number of components: proof of Theorems~\ref{t:trees}--\ref{t:complex}}\label{s:components}

Firstly, as indicated in Theorem \ref{t:trees}, we show that whp there are no tree components in $G(n,n,p)$ whose order is significantly larger than $\frac{1}{\delta}\left(\log \left(|\epsilon|^3n\right) - \frac{5}{2}\log\log\left( |\epsilon|^3n\right)\right)$. Moreover, we show that the number of tree components of order around this tends to a Poisson distribution.
\begin{proof}[Proof of Theorem \ref{t:trees}]\ \\
\textbf{Part \eqref{i:treecomp}:}
Let us write $k_i=\frac{1}{\delta}\left( \log\left( |\epsilon|^3 n\right) - \frac{5}{2} \log \log \left( |\epsilon|^3 n\right) + r_i\right)$ for $i\in \{1,2\}$. Then for all $k_1 \leq k \leq k_2$, we have that $\frac{k}{n},\frac{k^3}{n^2}, \frac{k^4}{n^3}$ and $\frac{\epsilon k^2}{n}$ are all $o(1)$. Therefore, it follows from \eqref{e:treeexpectation} that
\begin{align*}
\mathbb{E}(Y_{r_1,r_2}) &\approx \frac{n}{2\pi}\sum_{k=k_1}^{k_2} e^{-\delta k } \sum_{i+j=k}\frac{1}{(ij)^{\frac{3}{2}}}\left(\frac{i}{j}\right)^{j-i} \exp\left(-\frac{(i-j)^2}{2n}\right)=\frac{4n}{\pi}\sum_{k=k_1}^{k_2} e^{-\delta k }  \sum_{d=-k+1}^{k-1}\frac{1}{(k^2 - d^2)^{\frac{3}{2}}}\left(\frac{k-d}{k+d}\right)^{d}\exp\left(-\frac{d^2}{2n}\right),
\end{align*}
where the last equality holds by reparameterising over $d=j-i$. Hence, by Lemma \ref{l:difference}, we have
\begin{align}\label{eq:E(Y_(r1,r2))}
    \mathbb{E}(Y_{r_1,r_2}) \approx \frac{2 \sqrt{2}n}{\sqrt{\pi}}\sum_{k=k_1}^{k_2} \frac{e^{-\delta k }}{k^{\frac{5}{2}}}.
\end{align}
 Now, for any $\frac{r_1}{\delta} \leq a \leq \frac{r_2}{\delta}$ and
    \[
    k=\frac{1}{\delta}\left( \log\left( |\epsilon|^3 n\right) - \frac{5}{2} \log \log \left( |\epsilon|^3 n\right)\right) + a,
    \]
 we have that
    \begin{equation*}
    k^{\frac{5}{2}} \approx 4 \sqrt{2} |\epsilon|^{-5} \left(\log\left( |\epsilon|^3 n\right)\right)^{\frac{5}{2}},
    \end{equation*}
     since $\delta \approx \frac{\epsilon^2}{2}$,
    and hence in this range
    \begin{equation}\label{e:kapproximation}
    \frac{e^{-\delta k}}{k^{\frac{5}{2}}} = \frac{\left(\log\left( |\epsilon|^3 n\right)\right)^{\frac{5}{2}} e^{-\delta a} }{|\epsilon^3|n k^{\frac{5}{2}}} \approx \frac{|\epsilon^2|e^{-\delta a}}{4\sqrt{2}n} \approx \frac{\delta e^{-\delta a}}{2\sqrt{2}n}.
    \end{equation}
    Hence, substituting \eqref{e:kapproximation} into \eqref{eq:E(Y_(r1,r2))} we obtain
    \begin{align*}
    \mathbb{E}(Y_{r_1,r_2})\approx  \frac{1}{\sqrt{\pi}}\sum_{a=\frac{r_1}{\delta} }^{\frac{r_2}{\delta}}\delta e^{-\delta a}
    \approx\frac{1}{\sqrt{\pi}}\int_{r_1}^{r_2}e^{-t}dt
    =\frac{1}{\sqrt{\pi}}\left(e^{-r_1}-e^{-r_2}\right)
    =:\lambda.
\end{align*}

Next, we calculate the expected value of $(Y_{r_1,r_2})_2$, i.e., the second factorial moment of $Y_{r_1,r_2}$, which is the expected number of ordered pairs of tree components whose orders lie between $r_1$ and $r_2$. We see that $\mathbb{E}\left((Y_{r_1,r_2})_2\right)$ can be calculated as
\begin{align*} \sum_{k = k_1}^{k_2} \sum_{i+j=k} \binom{n}{i}\binom{n}{j} i^{j-1}j^{i-1}p^{k-1}(1-p)^{kn-ij-k+1}\sum_{k' = k_1}^{k_2} \sum_{r+s=k'} \binom{n-i}{r}\binom{n-j}{s} r^{s-1} s^{r-1} p^{k'-1}(1-p)^{k'n-rs-is-rj-k'+1},
\end{align*}
and we note that the inner sum is the expected number of tree components of order between $k_1$ and $k_2$ in $G(n_1,n_2,p)$, where $n_1=n-i,n_2=n-j$. However, since $i,j \leq k_2 = o(n)$ the same argument as before shows that this inner sum is asymptotically equal to  $ \mathbb{E}(Y_{r_1,r_2})$, and hence 
\[
\mathbb{E}\left((Y_{r_1,r_2})_2\right) \approx \left(\mathbb{E}(Y_{r_1,r_2})\right)^2 \approx \lambda^2.
\]

A similar argument shows that the $i$th factorial moment $\mathbb{E}\left((Y_{r_1,r_2})_i\right) \approx \lambda^i$ for each $i \in \mathbb{N}$, and hence $Y_{r_1,r_2}$ converges in distribution to Po$(\lambda)$ by Lemma \ref{l:Poisson}.

\textbf{Part \eqref{i:treecomp2}:} Let us write $k_3=\frac{1}{\delta}\left( \log\left( |\epsilon|^3 n\right) - \frac{5}{2} \log \log \left( |\epsilon|^3 n\right) + \alpha \right)$ and $Y_{\geq \alpha}$ for the number of tree components of order at least $k_3$. From \eqref{e:treesinitial}, but using \eqref{e:fallingfactorialbound} instead of \eqref{e:fallingfactorial} to bound the falling factorial term, we can bound the expected value of $Y_{\geq \alpha}$ from above as

\begin{align*}
		   \mathbb{E}(Y_{\geq \alpha}) \leq n\sum_{k = k_3}^n e^{-\delta k } \sum_{i+j=k}\frac{1}{2\pi(ij)^{\frac{3}{2}}}\left(\frac{i}{j}\right)^{j-i} \exp\left(-\frac{(i-j)^2}{2n}-\frac{i^3+j^3}{6n^2} +\frac{\epsilon ij}{n} + O\left(\frac{k}{n}\right)\right).
\end{align*}

For any $k\leq n$ and $i+j=k$, we have that $\frac{\epsilon ij}{n} \leq \frac{\epsilon k^2}{4n}$ and also $\frac{i^3 + j^3}{ 6n^2} \geq \frac{k^3}{24n^2}$, and so
\begin{align*}
\mathbb{E}(Y_{\geq \alpha}) \leq O\left(n \sum_{k= k_3}^n \exp\left(-\delta k  + \frac{\epsilon k^2}{4n} - \frac{k^3}{24n^2}\right) \sum_{i+j=k}\frac{1}{2\pi(ij)^{\frac{3}{2}}}\left(\frac{i}{j}\right)^{j-i}\exp\left(-\frac{(i-j)^2}{2n} \right) \right).
\end{align*}
Then, reparameterising with $d=j-i$ and applying Lemma \ref{l:difference} as before gives us that, 
\begin{align}\label{e:largetreebound} 
\mathbb{E}(Y_{\geq \alpha}) = O\left( n\sum_{k = k_3}^n\frac{1}{k^{\frac{5}{2}}} \exp\left(-\delta k  - \frac{k^3}{24n^2} + \frac{\epsilon k^2}{4n}\right)\right).
\end{align}
Let $s= \epsilon n$, then we are interested in the function
\[
-\delta k  + \frac{\epsilon k^2}{4n} - \frac{k^3}{24n^2} = \frac{k}{n^2}\left( -\frac{\delta s^2}{\epsilon^2} + \frac{ sk}{4} - \frac{k^2}{24} \right).
\]
Now, since $-\frac{\delta s^2}{\epsilon^2} + \frac{ sk}{4} - \frac{k^2}{24}$ as a function of $k$ is a parabola, whose maximum comes at $k=3s$, we can bound
\begin{align}\label{e:parabolaminimum}
\frac{k}{n^2}\left(-\frac{\delta s^2}{\epsilon^2} + \frac{ sk}{4} - \frac{k^2}{24}\right) \leq k\left( -\delta + \frac{3\epsilon^2}{4} - \frac{9\epsilon^2}{24}\right) \leq -\frac{\delta k}{5}.
\end{align}
Hence, by \eqref{e:largetreebound} and \eqref{e:parabolaminimum}, we have
\begin{align*}
\mathbb{E}(Y_{\geq \alpha}) =O \left( n \sum_{k = k_3}^n \frac{1}{k^{\frac{5}{2}}} \exp\left(-\frac{\delta k}{5}\right) \right). 
\end{align*}

Hence, if $\alpha \geq 10 \log \left(|\epsilon|^3n\right)$, then
\begin{align}
\mathbb{E}(Y_{\geq \alpha}) &= O\left(n \sum_{k = k_3}^n \frac{1}{k^{\frac{5}{2}}} \exp\left(-\frac{\delta k}{5}\right)\right) =O\left(n \sum_{k \geq \alpha / \delta}\frac{1}{k^{\frac{5}{2}}} \exp\left(-\frac{\delta k}{5}\right)\right)
= O\left(e^{-\frac{\alpha}{10}} n  \sum_{k \geq \alpha/\delta}\frac{1}{k^{\frac{5}{2}}} \exp\left(-\frac{\delta k}{10}\right)\right) \nonumber\\
&= O\left(e^{-\frac{\alpha}{10}}\frac{n}{(\frac{\alpha}{\delta})^{\frac{5}{2}}} \sum_{k \geq \alpha/\delta}\exp\left(-\frac{\delta k}{10}\right) \right)
= O\left( e^{-\frac{\alpha}{10}} \frac{n\delta^{\frac{5}{2}}e^{-\frac{\alpha}{10}} }{\left(\log \left( |\epsilon|^3 n\right)\right)^{\frac{5}{2}}\left(1-e^{-\frac{\delta}{10}}\right)} \right)= O\left(e^{-\frac{\alpha}{10}} \frac{n\delta^{\frac{5}{2}}}{\delta\left(\log \left(|\epsilon|^3 n\right)\right)^{\frac{5}{2}}|\epsilon|^3 n}\right) \nonumber\\
&= O\left(e^{-\frac{\alpha}{10}} \frac{1}{\left(\log \left(|\epsilon|^3 n\right)\right)^{\frac{5}{2}}}\right) \leq  e^{-\Omega(\alpha)}. \nonumber
\end{align}

Finally, if $\alpha \leq 10\log\left( |\epsilon|^3n\right) := \hat{\alpha}$, let $k_4 = \frac{1}{\delta}\left( \log \left( |\epsilon|^3 n\right) - \frac{5}{2} \log \log\left( |\epsilon|^3 n\right) + \hat{\alpha} \right)$. We can argue as in the first part that 
\[
\mathbb{E}(Y_{\alpha,\hat{\alpha}})  = e^{-\Omega(\alpha)},
\]
since as in \eqref{e:kapproximation}, as long as $k = \frac{1}{\delta}\left( \log \left( |\epsilon|^3 n\right) - \frac{5}{2} \log \log \left( |\epsilon|^3 n\right) + \alpha \right) =\Theta\left(\frac{\log \left(|\epsilon|^3n\right)}{\delta}\right)$ we have that $e^{-\delta k}k^{-\frac{5}{2}} = \Theta\left(\delta e^{-\delta \alpha}n^{-1}\right)$. It follows that, 
\begin{align*}
\mathbb{E}(Y_{\geq \alpha}) = \mathbb{E}(Y_{\alpha,\hat{\alpha}}) + \mathbb{E}(Y_{\geq \hat{\alpha}}) = e^{-\Omega(\alpha)} + e^{-\Omega(\hat{\alpha})} = e^{-\Omega(\alpha)},
\end{align*}
and so the result follows from Markov's inequality.
\end{proof} 

Secondly, as indicated in Theorem \ref{t:unicyclic}, we show that whp there are no unicyclic components in $G(n,n,p)$ of order significantly larger than $\frac{1}{\delta}$, and moreover, that the number of unicyclic components of order around this tends to a Poisson distribution.
\begin{proof}[Proof of Theorem \ref{t:unicyclic}]\ \\
{\bf Part \eqref{i:unicomp}:} Let us write $s_i = \frac{u_i}{\delta}$ for $i \in \{1,2\}$. We first note that, by Lemma \ref{l:balanced}, $G(n,n,p)$ contains no unbalanced non-tree components of order $\geq s_1$ with probability $1 - e^{-\Omega(s_1)}$, and hence whp $Z_{u_1,u_2} = Z'_{u_1,u_2}$ where $Z'_{u_1,u_2}$ is the number of unicyclic balanced components with order between $s_1$ and $s_2$.

Let us write $B_k = \left\{ (i,j) \in \mathbb{N}^2 \colon i+j=k \text{ and } \frac{1}{2}\leq \frac{i}{j} \leq 2 \right\}$. Since for $s_1 \leq k \leq s_2$ we have that $\frac{k}{n},\frac{k^3}{n^2}, \frac{k^2}{n^2}$ and $\frac{k^4}{n^3}$ are all $o(1)$, it follows from \eqref{e:unicyclicexpectation} and Lemma \ref{l:difference} that 

\begin{align}
	    \mathbb{E}(Z'_{u_1,u_2})&=\sum_{k=s_1}^{s_2}\sum_{(i,j) \in B_k}\mathbb{E}\left(X(i,j,0)\right)\approx \frac{1}{4\sqrt{2 \pi}}\sum_{k=s_1}^{s_2} \sqrt{k}e^{-\delta k} \sum_{(i,j) \in B_k} \frac{1}{ij} \left(\frac{i}{j}\right)^{j-i} \exp\left(-\frac{(i-j)^2}{2n} \right)\nonumber\\
	    &\approx \frac{1}{2} \sum_{k=s_1}^{s_2} \frac{1}{k} e^{-\delta k}\approx \frac{1}{2} \int_{u_1}^{u_2}\frac{e^{-t}}{t} dt:=\nu. \label{e:distunicyclic}
	    \end{align}
	    
As in Theorem \ref{t:trees} \eqref{i:treecomp} a similar argument shows that $\mathbb{E}\left((Z'_{u_1,u_2})_i\right) \approx \nu^i$ for all $i \in \mathbb{N}$ and hence $Z'_{u_1,u_2}$, and so also $Z_{u_1,u_2}$, converges in distribution to  Po$(\nu)$.

\textbf{Part \eqref{i:unicomp2}:} Let $s_3 = \frac{\alpha}{\delta}$ and let $Z_{\geq \alpha}$ and $Z'_{\geq \alpha}$ be the number of unicyclic components and balanced unicyclic components respectively of order at least $s_3$. Note that, as before, $Z_{\geq \alpha}=Z'_{\geq \alpha}$ with probability $1 - e^{-\Omega(s_3)} =  1  - e^{-\Omega(\alpha)}  $.

A similar argument as in Theorem \ref{t:trees} \eqref{i:treecomp2} shows that for any $i+j=k \leq n$
\begin{align}\label{eq: E(Z'geq alpha)}
\mathbb{E}\left(Z'_{\geq \alpha}\right) = O\left(\sum_{k= s_3}^n \sqrt{k} \exp\left(-\delta k  - \frac{k^3}{24n^2} + \frac{\epsilon k^2}{4n}\right) \sum_{i+j=k}\frac{1}{ij} \left(\frac{i}{j}\right)^{j-i}\exp\left(-\frac{(i-j)^2}{2n} \right)\right) 
= O\left(\sum_{k= s_3}^n \frac{1}{k} e^{-\frac{\delta k}{5}}\right).
\end{align}
On the other hand, it can be shown, see for example \cite[Formulas 5.1.1 and 5.1.20]{Abramowitz}, that
\[
E_1(x) := \int_{x}^\infty \frac{e^{-t}}{t} \,dt \leq e^{-x} \log \left(1 + \frac{1}{x} \right)
\]
and hence 
\begin{align}\label{eq: exponential integration}
    \sum_{k=s_3}^n \frac{1}{k}e^{-\frac{\delta k}{5}}\approx \int_{s_3}^n \frac{1}{u}e^{-\frac{\delta u}{5}}du=\int_{\frac{\alpha}{5}}^{\frac{\delta n}{5}}\frac{e^{-t}}{t}dt\leq e^{-\frac{\alpha}{5}}\log\left(1+\frac{5}{\alpha}\right)= e^{-\Omega(\alpha)}
\end{align}
for $\alpha\geq 5$.

By \eqref{eq: E(Z'geq alpha)} and \eqref{eq: exponential integration}, it follows that $\mathbb{E}\left(Z'_{\geq \alpha}\right)=e^{-\Omega(\alpha)}$. In the case where $1 < \alpha \leq 5$ we can use \eqref{e:distunicyclic} to see
\[
\mathbb{E}\left(Z'_{\geq \alpha}\right)=\mathbb{E}\left(Z'_{\alpha,5}\right) + \mathbb{E}\left(Z'_{\geq 5}\right) \leq \frac{1}{2} E_1(\alpha) + e^{-\Omega(1)} \leq e^{-\Omega(1)}.
\]
\end{proof} 

Finally, as indicated in Theorem \ref{t:complex}, we show that whp there are no large complex components in $G(n,n,p)$, and in fact no complex components at all in the subcritical regime.
\begin{proof}[Proof of Theorem \ref{t:complex}]\ \\
\textbf{Part \eqref{i:complexcomp}:} To show the first part, recalling that $\epsilon<0$, we use the observation that, since each complex component must contain a connected subgraph of excess precisely two, it is sufficient to show that whp $G(n,n,p)$ contains no such subgraphs.

In fact, it is sufficient to show that whp $G(n,n,p)$ contains no subgraph which is minimal with respect to the properties of being connected and having excess two, and we note that any such graph consists of a pair of cycles, which are either joined by a path or whose intersection is a path. Let us denote the number of such subgraphs by $Q$. The key observation is that any such graph of order $k$ can be built by taking a path on $k$ vertices and adding an edge from each of its endpoints to another vertex in the path. Hence, we can choose such a subgraph on $k$ vertices by first choosing the $i=\left\lfloor \frac{k}{2} \right\rfloor$ vertices of the path lying in one partition class and the $j=\left\lceil \frac{k}{2} \right\rceil$ vertices of the path lying in the other partition class, choosing the order which the vertices appear in the path in at most $i!j!$ many ways and then choosing for each endpoint of the path one of the at most $k$ many edges from this endpoint to another vertex in the path. It follows that
		\begin{align*}
		      \mathbb{E}(Q) &\leq 2\sum_{k=3}^{n} \binom{n}{k}\binom{n}{k}(k!)^2k^2p^{2k+1}+2\sum_{k=2}^{k_0} \binom{n}{k}\binom{n}{k+1}(k+1)!k!k^2p^{2k+2}\\
       &\leq 2\sum_{k=3}^{n}\frac{n^k}{k!}\frac{n^k}{k!}(k!)^2k^2 \left(\frac{1+\epsilon}{n}\right)^{2k+1}+2\sum_{k=2}^{k_0}\frac{n^k}{k!}\frac{n^{k+1}}{(k+1)!}(k+1)!k!k^2\left(\frac{1+\epsilon}{n}\right)^{2k+2} \\
       &\leq 4\sum_{k=2}^n \frac{k^2}{n}e^{2\epsilon k}\leq \frac{4}{n}\int_0^\infty x^2e^{2\epsilon x}dx=\frac{1}{|\epsilon|^3n}.
		\end{align*}
		Therefore, by Markov's inequality, with probability at least $1-\frac{1}{|\epsilon|^3 n}$ there are no complex components.

{\bf Part \eqref{i:complexcomp2} :} Recall that $\epsilon>0$.
As in  Theorem \ref{t:unicyclic}  \eqref{i:unicomp}, let $A(k,\ell)$ and $A'(k,\ell)$ be the number of components and balanced components respectively of order $k$ with excess $\ell \geq 1$. If we write $A = \sum_{k=1}^{n^{\frac{2}{3}}}\sum_{\ell \geq 1} A(k,\ell)$ and $A' = \sum_{k=1}^{n^{\frac{2}{3}}}\sum_{\ell \geq 1} A'(k,\ell)$, then by Lemma \ref{l:balanced} with probability $1 - O\left( n^{-1}\right)= 1- O\left(\left(\epsilon^3 n\right)^{-1}\right)$, $A = A'$.

Since $\frac{k}{n} = o(1)$ for $k\leq n^{\frac{2}{3}}$, we see by \eqref{e:complexexpectation} in Theorem \ref{t:complexexpectation} that
\begin{align}
\mathbb{E}(A') &= \sum_{k=1}^{n^{\frac{2}{3}}} \sum_{(i,j) \in B_k} \sum_{\ell =1}^{ij-i-j} \mathbb{E}\left(X(i,j,\ell)\right) \nonumber \\
& \lesssim  \sum_{k=1}^{n^{\frac{2}{3}}} \frac{1}{\sqrt{k}}\exp\left(-\delta k + \frac{\epsilon k^2}{4n} \right) \sum_{(i,j) \in B_k} \frac{1}{\sqrt{ij}} \left(\frac{i}{j}\right)^{j-i} \exp\left(-\frac{(i-j)^2}{2n} \right)\sum_{\ell =1}^{ij-i-j} \left(\frac{ck^3}{\ell n^2}\right)^\frac{\ell}{2} \text{exp}\left( \ell \log (1+\epsilon) + \frac{\ell(1+\epsilon)}{n} \right). \label{e:complexcomponents}
\end{align}

Let us first deal with the innermost sum of \eqref{e:complexcomponents}
\[
\sum_{\ell =1}^{ij-i-j} \left(\frac{ck^3}{\ell n^2}\right)^\frac{\ell}{2} \text{exp}\left( \ell \log (1+\epsilon) + \frac{\ell(1+\epsilon)}{n} \right). 
\]
The ratio of consecutive terms in the sum is 
\begin{align*}
 \sqrt{\frac{ck^3}{n^2}}\frac{\ell^{\frac{\ell}{2}}}{(\ell+1)^\frac{\ell+1}{2}}\exp\left(\log(1+\epsilon)+\frac{1+\epsilon}{n}\right),
\end{align*}
which is strictly less than 1 when $\ell$ is large enough compared to $c$. However, for any constant $\ell\geq 1$ the individual terms can be seen to have order
\[
O\left( \frac{k^3}{n^2} \right)^{\frac{\ell}{2}} = O\left(\frac{k^{\frac{3}{2}}}{n}\right)
\]
since $k \leq n^{\frac{2}{3}}$. It follows that
\begin{equation}
   \sum_{\ell =1}^{ij-i-j} \left(\frac{ck^3}{\ell n^2}\right)^\frac{\ell}{2} \text{exp}\left( \ell \log (1+\epsilon) + \frac{\ell(1+\epsilon)}{n} \right) =  O\left(\frac{k^{\frac{3}{2}}}{n}\right). \label{e:complexinner}
\end{equation}
Next, we see that the second sum can be evaluated using Lemma \ref{l:difference} to give
\begin{equation}\label{e:complexouter}
\sum_{(i,j) \in B_k}\frac{1}{\sqrt{ij}} \left(\frac{i}{j}\right)^{j-i} \exp\left(-\frac{(i-j)^2}{2n} \right) = O\left(k^{-\frac{1}{2}} \right).
\end{equation}
Hence, by \eqref{e:complexinner} and \eqref{e:complexouter}, and using that, since $k \leq n^{\frac{2}{3}} = o(\epsilon n)$ we have $\frac{\epsilon k^2 }{4n} = o(\delta k)$, we see that

\begin{align*}
\mathbb{E}(A') =  O\left(\frac{1}{n} \sum_{k=1}^{n^{\frac{2}{3}}} \sqrt{k} e^{-\frac{\delta k}{2}} \right) = O\left(\frac{1}{n} \int_0^{\infty} \sqrt{x} e^{-\frac{\delta x}{2}} \right) =  O\left(\frac{1}{\delta^{\frac{3}{2}} n}\right)=O\left(\frac{1}{\epsilon^3 n}\right).
\end{align*}

Hence, the result follows by Markov's inequality.
\end{proof}

\subsection{Largest and second largest components: proof of Theorem~\ref{t:giant}}\label{s:giant}
In order to show that there is in fact a unique giant component in $G(n,n,p)$ for an appropriate range of $\epsilon$, we follow a relatively standard approach. First, we estimate quite precisely the number of vertices which are contained in small tree or unicyclic components, noting that by the lemmas in the previous section there are whp no large tree or unicyclic components and no small complex components. It follows that whp all the remaining vertices are contained in large complex components, and by a sprinkling argument we are able to show that whp these vertices are in fact all contained in a single component.

Throughout this section, let us consider two quantities related to $\epsilon=\epsilon(n) >0$ satisfying  $\epsilon=o(1)$:  firstly $\delta$ as defined in~\eqref{e:definedelta}, i.e.,
\begin{equation*}
\delta = \epsilon - \log(1+\epsilon) \approx \frac{\epsilon^2}{2},
\end{equation*}
and secondly $\epsilon'$ as in Theorem \ref{t:giant}, which is defined implicitly as the unique positive solution to 
\begin{equation}\label{e:defineepsilon'}
(1-\epsilon')e^{\epsilon'} = (1+\epsilon)e^{-\epsilon}.
\end{equation}
We note that $\epsilon'=\epsilon-\frac{2}{3}\epsilon^2+O(\epsilon^3)$. We also note that $\epsilon'$ has the following natural interpretation in terms of branching processes: If we consider a Po$(1+\epsilon)$ branching process and condition on the event that it does not survive, then it can be shown that this model is distributed as a Po$\left(1-\epsilon'\right)$ branching process. Whenever we use the terms $\epsilon'$ and $\delta$ they refer to these quantities for a fixed $\epsilon$, which should be clear from the context.  

As indicated in Theorem \ref{t:complex}, in the weakly subcritical regime whp there are no complex components. However, for our proof it will be necessary to know more, namely that in this regime we do not expect to have many vertices contained in `large' components. The proof of this fact can be deduced from a standard comparison to a branching process and we defer the details to Appendix \ref{s:Appendix C}.

\begin{theorem}\label{t:subcriticalsmallcomponents}
Let $\epsilon = \epsilon(n) >0$ be such that $\epsilon^3n \rightarrow \infty$ and $\epsilon=o(1)$, and let $p=\frac{1 - \epsilon}{n}$. Then the expected number of vertices in $G(n,n,p)$ in components of order at least $\sqrt{\frac{n}{3\epsilon}}$ is $o\left(\sqrt{\frac{n}{\epsilon}}\right)$.
\end{theorem}

Let us begin then, by estimating the number of vertices contained in small tree or unicyclic components.

\begin{lemma}\label{l:vertices}
Let  $\epsilon = \epsilon(n) >0$ be such that $\epsilon^3n \gg \omega \rightarrow \infty$ and $\epsilon = o(1)$, and let $p=\frac{1+\epsilon}{n}$. 
Let $Y(-1)$ and $Y(0)$ denote the number of vertices in tree and unicyclic components of order at most $n^{\frac{2}{3}}$ in $G(n,n,p)$ respectively. 	
Then with probability $1 - O\left(\omega^{-1}\right)$, we have
\[
Y(0) \leq \frac{4\omega}{\delta},
\]
and
\[
\left|Y(-1) - \frac{2(1-\epsilon')}{1+\epsilon}n\right| \leq \frac{\omega \sqrt{n}}{\sqrt{\epsilon}}. 
\]
	
\end{lemma}

\begin{proof}
 First, we bound $Y(0)$. As before, we let 
 \[
 B_k = \left\{ (i,j) \in \mathbb{N}^2 \colon i+j=k \text{ and } \frac{1}{2} \leq \frac{i}{j} \leq 2 \right\} \,\,\, \text{ and }  \,\,\,
 U_k = \left\{(i,j) \in \mathbb{N}^2 \colon i+j=k\right\} \setminus B_k.
 \] 
 Then we can split the calculation of $\mathbb{E}(Y(0))$ into two parts
 \begin{align*}
 \mathbb{E}(Y(0)) &=  \sum_{k\leq n^{\frac{2}{3}}}k\sum_{(i,j) \in B_k}\mathbb{E}\left(X(i,j,0)) \right)+  \sum_{k\leq n^{\frac{2}{3}}}k\sum_{(i,j) \in U_k}\mathbb{E}\left(X(i,j,0)\right):= S_1 + S_2.
 \end{align*}
 
Since if $i+j = k$, then $\frac{\epsilon ij}{n} \leq \frac{\epsilon k^2}{4n} = o(\delta k)$ for $k \leq n^{\frac{2}{3}}$, it follows from \eqref{e:unicyclicexpectation} and Lemma \ref{l:difference} that
 \begin{align*}
     S_1 \lesssim \frac{1}{4\sqrt{2\pi}}\sum_{k \leq n^{\frac{2}{3}}}k^{\frac{3}{2}}e^{-\delta k} \sum_{(i,j) \in B_k}\frac{1}{ij}\left(\frac{i}{j}\right)^{j-i} \text{exp}\left( - \frac{(i-j)^2}{2n} + \frac{\epsilon ij}{n} \right)\leq \frac{1}{2} \sum_{k\leq n^{\frac{2}{3}}}e^{-\frac{\delta k}{2} } \leq \frac{1}{2}\int_{0}^{\infty}e^{-\frac{\delta x}{2}}dx\leq \frac{1}{\delta}.
 \end{align*}
 
 Furthermore, using the very naive bound that $C(i,j,0) \leq ij C(i,j-1) \leq i^jj^i$, we can calculate as in \eqref{e:unicyclicexpectation}
 \begin{align*}
     S_2 &\leq \sum_{k\leq n^{\frac{2}{3}}}k\sum_{(i,j) \in U_k}\mathbb{E}(X(i,j,0))\leq \sum_{k\leq n^{\frac{2}{3}}}k\sum_{(i,j) \in U_k} \binom{n}{i}\binom{n}{j} i^j j^i p^{k} (1-p)^{kn -ij -k} \\
     &\approx \frac{1}{2 \pi}\sum_{k\leq n^{\frac{2}{3}}} k e^{-\delta k}\sum_{(i,j) \in U_k} \frac{1}{(ij)^{\frac{1}{2}}} \left(\frac{i}{j}\right)^{j-i} \exp\left(- \frac{(i-j)^2}{2n} + \frac{\epsilon ij}{n} \right)\\
     &\leq \sum_{k\leq n^{\frac{2}{3}}} ke^{-\frac{\delta k}{2}} \left(\frac{1}{2}\right)^{\frac{k}{3}}\sum_{(i,j) \in U_k} \frac{1}{(ij)^{\frac{1}{2}}}\leq \sum_{k\leq n^{\frac{2}{3}}} k^{\frac{3}{2}} \left(\frac{1}{2} \right)^{\frac{k}{3}}= O(1).
 \end{align*}
 The first part of the lemma then follows by Markov's inequality.
 
So, let us consider the bound on $Y(-1)$. Firstly, we note that by part \eqref{i:treecomp2} of Theorem \ref{t:trees} with $\alpha(n)=\sqrt{\frac{\epsilon^3n}{16}}$, with probability $1-o\left(e^{-\sqrt{\frac{\epsilon^3n}{16}}}\right) = 1- O\left(\omega^{-1}\right)$ there are no tree components in $G(n,n,p)$ of order at least $\sqrt{\frac{n}{3\epsilon}}$, and by Lemma \ref{l:uniform} with probability $1-o\left(n^{-1}\right) = 1- O\left(\omega^{-1}\right)$ there are no non-$\epsilon$-uniform tree components in $G(n,n,p)$ of order at most $n^\frac{2}{3}$. Hence, with probability $1 - O\left(\omega^{-1}\right)$, $Y(-1) = Z_1$ where $Z_1$ is, as in Lemma \ref{l:Var}, the number of vertices in $\epsilon$-uniform tree components in $G(n,n,p)$ of order at most $\tilde{k} = \sqrt{\frac{n}{3\epsilon}}$. 

Next, following a technique of Bollob\'{a}s \cite[Theorem 6.6]{BollobasBook}, we consider the model $G(n,n,p')$ where $p'=\frac{1-\epsilon'}{n}$. Let us write $Y'(-1)$ and $Y'(0)$ for the number of vertices in tree and unicyclic components in $G(n,n,p')$ of order at most $n^{\frac{2}{3}}$ respectively, and similarly $Z_1'$ for the number of vertices in $\epsilon$-uniform tree components in $G(n,n,p')$ of order at most $\sqrt{\frac{n}{3\epsilon}}$.

We will show that almost every vertex in $G(n,n,p')$ lies in $\epsilon$-uniform tree components of order at most $\sqrt{\frac{n}{3\epsilon}}$, and we are able to calculate the ratio $\mathbb{E}(Z_1) / \mathbb{E}(Z_1')$ quite precisely. Combining this with the bound on the variance of $Z_1$ from Lemma \ref{l:Var} we are able to deduce the second part of the lemma.

Indeed, by Theorem \ref{t:subcriticalsmallcomponents} the expected number of vertices in components of order greater than $\sqrt{\frac{n}{3\epsilon}}$ in $G(n,n,p')$ is $o\left(\sqrt{\frac{n}{\epsilon}}\right)$. Furthermore, as we demonstrate
below, similar calculations to the first part of the lemma will prove  that the expected number of vertices in unicyclic and complex components of order at most $n^{\frac{2}{3}}$ in $G(n,n,p')$ is $o\left(\sqrt{\frac{n}{\epsilon}}\right)$.

Indeed, if we let $Y'(\geq 1)$ be the number of vertices in complex components of order at most $n^{\frac{2}{3}}$ in $G(n,n,p')$, then
\begin{align*}
\mathbb{E}\left(Y'( \geq 1)\right) = \sum_{k\leq n^{\frac{2}{3}}}k\sum_{(i,j) \in B_k} \sum_{\ell = 1}^{ij-i-j}\mathbb{E}\left(X'(i,j,\ell)\right) +  \sum_{k\leq n^{\frac{2}{3}}}k\sum_{(i,j) \in U_k}\sum_{\ell = 1}^{ij-i-j}\mathbb{E}\left(X'(i,j,\ell)\right):= S'_1 + S'_2,
\end{align*}
where $X'(i,j,\ell)$ is the number of components in $G(n,n,p')$ with $i$ vertices in $N_1$, $j$ vertices in $N_2$, $i+j+\ell$ edges.

One can bound $S'_2$ in a similar fashion as with $S_2$, since the exponentially small term $\left(\frac{i}{j} \right)^{j-i}$ is the dominating term. For $S'_1$ we use \eqref{e:complexinitial}, and as in Theorem \ref{t:complex} we can argue
\begin{align*}
S'_1 &\leq \sum_{k\leq n^{\frac{2}{3}}}k\sum_{(i,j) \in B_k} \sum_{\ell = 1}^{ij-i-j}\sqrt{k} e^{-k} \left(\frac{i}{j}\right)^{j-i}\left(\frac{ck^3}{\ell n^2}\right)^\frac{\ell}{2}\frac{(n)_i(n)_j}{n^k} (1-\epsilon')^{k+\ell}\left(1-\frac{1-\epsilon'}{n}\right)^{kn-ij-k-\ell}\\
&\leq \sum_{k\leq n^{\frac{2}{3}}}k^{\frac{3}{2}} (1-\epsilon')^{k} \sum_{(i,j) \in B_k} \sum_{\ell = 1}^{ij-i-j} \left(\frac{ck^3}{\ell n^2}\right)^\frac{\ell}{2} = O\left( \frac{1}{n} \sum_{k\leq n^{\frac{2}{3}}}k^{4} e^{-\epsilon' k} \right)= O\left( \frac{1}{\epsilon^5 n} \right)=o\left(\sqrt{\frac{n}{\epsilon}}\right),
\end{align*}
as long as $\epsilon^{3}n \rightarrow \infty$.
Furthermore, arguments similar to those used to bound $\mathbb{E}(Y(0))$ immediately imply that the expected number of vertices in small unicyclic components is at most $O\left( (\epsilon')^{-2}\right)  =  o \left(\sqrt{\frac{n}{\epsilon}}\right)$. Finally, the expected number of vertices in non-$\epsilon$-uniform components of order at most $n^{\frac{2}{3}}$ in $G(n,n,p')$ is $o(1)$, as follows from the proof of Lemma \ref{l:uniform}. 
It follows that 
\[
\mathbb{E}\left(Z'_1\right) = 2n - o\left( \sqrt{\frac{n}{\epsilon}}\right).
\]

Let us write $Z_1(k)$ and $Z_1'(k)$ for the number of vertices in $\epsilon$-uniform tree components of order $k \leq \sqrt{\frac{n}{3 \epsilon}}$ in $G(n,n,p)$ and $G(n,n,p')$ respectively, and let us consider the ratio
     \begin{align*}
         \frac{\mathbb{E}\left(Z_1(k)\right)}{\mathbb{E}\left(Z'_1(k)\right)}&=\frac{\sum_{i+j=k}\binom{n}{i}\binom{n}{j}i^{j-1}j^{i-1}\left(\frac{1+\epsilon}{n}\right)^{k-1}\left(1-\frac{1+\epsilon}{n}\right)^{kn-ij-k+1}}{\sum_{i+j=k}\binom{n}{i}\binom{n}{j}i^{j-1}j^{i-1}\left(\frac{1-\epsilon'}{n}\right)^{k-1}\left(1-\frac{1-\epsilon'}{n}\right)^{kn-ij-k+1}},
     \end{align*}
     where the sums run over the $\epsilon$-uniform pairs $(i,j)$.
     
Note that,
 \begin{align*}
      \left(1-\frac{1+\epsilon}{n}\right)^{kn-ij-k+1} =\left(\frac{n-1}{n}\right)^{kn-ij-k+1}\left(1-\frac{\epsilon}{n-1}\right)^{kn-ij-k+1}=\left(\frac{n-1}{n}\right)^{kn-ij-k+1} \exp \left(-\epsilon k + \frac{\epsilon ij}{n} + o(1) \right),
\end{align*}
and similarly
\[
\left(1-\frac{1-\epsilon'}{n}\right)^{kn-ij-k+1} =\left(\frac{n-1}{n}\right)^{kn-ij-k+1} \exp \left(\epsilon' k - \frac{\epsilon' ij}{n} + o(1) \right).
\]
Hence, we have 
\begin{align*}
         \frac{\mathbb{E}(Z(k))}{\mathbb{E}(Z'(k))}&=\frac{1-\epsilon'}{1+\epsilon}\left(\frac{(1+\epsilon)e^{-\epsilon}}{(1-\epsilon')e^{\epsilon'}}\right)^k\frac{e^{O\left(\frac{\epsilon k^2}{n}\right)}}{e^{-O\left(\frac{\epsilon' k^2}{n}\right)}}\frac{ \sum_{i+j=k}\binom{n}{i}\binom{n}{j}i^{j-1}j^{i-1}\left(\frac{n-1}{n}\right)^{kn-ij-k+1}}{\sum_{i+j=k}\binom{n}{i}\binom{n}{j}i^{j-1}j^{i-1}\left(\frac{n-1}{n}\right)^{kn-ij-k+1}}=\frac{1-\epsilon'}{1+\epsilon}e^{O\left(\frac{\epsilon k^2}{n}\right)},
     \end{align*} 
  since the second factor is equal to $1$ by the definition of $\epsilon'$ in \eqref{e:defineepsilon'}. 
  
 So, we see that 
 \[\frac{\mathbb{E}\left(Z_1(k)\right)}{\mathbb{E}\left(Z'_1(k)\right)}=\frac{1-\epsilon'}{1+\epsilon}+O\left(\frac{\epsilon k^2}{n}\right),
 \]
 or, in other words,
 \[
 \mathbb{E}\left(Z_1(k)\right) = \frac{1-\epsilon'}{1+\epsilon}\mathbb{E}\left(Z'_1(k)\right)+O\left(\frac{\epsilon k^2}{n}\right)\mathbb{E}\left(Z'_1(k)\right).
 \]
 
 Hence, by writing $\mathbb{E}\left(Z'(k)\right) = k\mathbb{E}\left(\hat{Y}(k)\right)$  where $\hat{Y}(k)$ is the number of $\epsilon$-uniform tree components of order $k$ in $G(n,n,p')$ we see that
    \begin{align}
        \mathbb{E}(Z_1)&= \sum_{k=1}^{\sqrt{\frac{n}{3 \epsilon}}} \mathbb{E}(Z_1(k))\nonumber= \frac{1-\epsilon'}{1+\epsilon} \sum_{k=1}^{\sqrt{\frac{n}{3 \epsilon}}} \mathbb{E}\left(Z'_1(k)\right) + O\left(\frac{\epsilon}{n}\right) \sum_{k=1}^{\sqrt{\frac{n}{3 \epsilon}}} k^3 \mathbb{E}\left(\hat{Y}(k)\right)\nonumber = \frac{1-\epsilon'}{1+\epsilon} \mathbb{E}\left(Z'_1\right) + O\left(\frac{\epsilon}{n}\right) \sum_{k=1}^{\sqrt{\frac{n}{3 \epsilon}}} k^3 \mathbb{E}\left(\hat{Y}(k)\right) \nonumber\\&=  \frac{1-\epsilon'}{1+\epsilon}\left(2n - o \left(\sqrt{\frac{n}{\epsilon}}\right)\right)+O\left(\frac{\epsilon}{n}\right)\sum_{k=1}^{\sqrt{\frac{n}{3 \epsilon}}} k^3\mathbb{E}\left(\hat{Y}(k)\right). \label{e:expectationofZ}
    \end{align}
  
     The final sum can be bounded by the corresponding sum over all possible tree components, $\epsilon$-uniform or not. That is, writing $X'(i,j,-1)$ for the number of tree components with $i$ vertices in one partition class and $j$ vertices in the other in $G(n,n,p')$ and noting that $\delta= \epsilon - \log(1+\epsilon) = \log(1-\epsilon') - \epsilon'$, we can bound in a similar manner to \eqref{e:treeexpectation}
     \begin{align*}
         \sum_{k=1}^  {\sqrt{\frac{n}{3 \epsilon}}}k^3\mathbb{E}\left(\hat{Y}(k))\right)&\leq \sum_{k=1}^  {\sqrt{\frac{n}{3 \epsilon}}} k^3 \sum_{i+j=k} \mathbb{E}\left(X'(i,j,-1)\right)\leq \sum_{k=1}^  {\sqrt{\frac{n}{3 \epsilon}}}k^3\sum_{i+j=k}\binom{n}{i}\binom{n}{j}i^{j-1}j^{i-1}p'^{k-1}(1-p')^{kn-ij-k+1}\\ 
         &\lesssim \sum_{k=1}^  {\sqrt{\frac{n}{3 \epsilon}}}k^3n e^{-\delta k}\sum_{i+j=k}\frac{1}{(ij)^{\frac{3}{2}}}\left(\frac{i}{j}\right)^{j-i}\exp\left(-\frac{(i-j)^2}{2n}\right):= S.
     \end{align*}
     
     To bound $S$, we split into two cases. Let us take $s$ such that $s = \delta^{-\frac{3}{7}}$, and first consider the case when $k\leq s$, where
     \begin{align*}
           S_1 &:= n\sum_{k=1}^{s}k^3 e^{-\delta k}\sum_{i+j=k}\frac{1}{(ij)^{\frac{3}{2}}}\left(\frac{i}{j}\right)^{j-i}\exp\left(-\frac{(i-j)^2}{2n}\right)\leq n\sum_{k=1}^{s}k^3 \sum_{i+j=k}\frac{1}{k^\frac{3}{2}}\leq n\sum_{k=1}^{s} k^\frac{5}{2}\leq  ns^\frac{7}{2}=O\left(\frac{n}{\delta^{\frac{3}{2}}}\right)= o\left(\left(\frac{n}{\epsilon}\right)^{\frac{3}{2}}\right).
     \end{align*}
     
     Conversely, when $k\geq s$, we see that, by Lemma \ref{l:difference} 
     \begin{align*}
        S_2&:=n \sum_{k=s}^{\sqrt{\frac{n}{3 \epsilon}}}k^3 e^{-\delta k}\sum_{i+j=k}\frac{1}{(ij)^{\frac{3}{2}}}\left(\frac{i}{j}\right)^{j-i}\exp\left(-\frac{(i-j)^2}{2n}\right)=O\left( n \sum_{k=s}^{\sqrt{\frac{n}{3 \epsilon}}}\sqrt{k}e^{-\delta k}\right)\\
        &=O\left( n\int_{0}^{\infty}\sqrt{x}e^{-\delta x}dx\right)=O\left(\frac{n}{\delta^{\frac{3}{2}}}\right)= o\left(\left(\frac{n}{\epsilon}\right)^{\frac{3}{2}}\right).
     \end{align*}

      Hence, $S = S_1 + S_2 =  o\left(\left(\frac{n}{\epsilon}\right)^{\frac{3}{2}}\right)$ and so by \eqref{e:expectationofZ}
      \[
      \mathbb{E}(Z_1)=\frac{1-\epsilon'}{1+\epsilon}\left(2n-o\left(\sqrt{\frac{n}{\epsilon}}\right)\right)+O\left(\frac{\epsilon}{n}\right)o\left(\left(\frac{n}{\epsilon}\right)^{\frac{3}{2}}\right)=\frac{2(1-\epsilon')}{1+\epsilon}n+ o\left(\sqrt{\frac{n}{\epsilon}}\right).
      \]
      
      Finally, by Lemma \ref{l:Var} with $\tilde{k} = \sqrt{\frac{n}{3\epsilon}}$, which can be seen to satisfy $\frac{3\epsilon \Tilde{k}^2}{n} \leq 1$, we conclude that $\text{Var}(Z_1) = O\left(\frac{n}{\epsilon}\right)$. Hence, by Chebyshev's inequality, $\left|Z_1-\frac{2(1-\epsilon')}{1+\epsilon}n\right|\leq \omega n^{\frac{1}{2}}\epsilon^{-\frac{1}{2}}$ with probability $1-O\left(\omega^{-1}\right)$. Thus, with probability $1 - O\left(\omega^{-1}\right)$,
        $$\left|Y(-1)-\frac{2(1-\epsilon')}{1+\epsilon}n\right|\leq \omega n^{\frac{1}{2}}\epsilon^{-\frac{1}{2}}.$$
\end{proof} 

Using Lemma \ref{l:vertices}, we can give a good bound on the number of vertices which are contained in components of order at least $n^{\frac{2}{3}}$ in $G(n,n,p)$. Then, using a sprinkling argument we can deduce that whp all these vertices are contained in a unique `giant' component, and determine asymptotically its order.

\begin{proof}[Proof of Theorem~\ref{t:giant}]
Let $\mathscr{L}(G)$ denote the set of vertices lying in components of $G$ of order larger than $n^\frac{2}{3}$, which we call \emph{large}. We first estimate quite precisely the size of $\mathscr{L}(G(n,n,p))$ and then show that there is only one large component in $G(n,n,p)$.

Indeed, if we let $\omega =\frac{\left(\epsilon^3 n\right)^\frac{1}{6}}{100}$, then by part \eqref{i:complexcomp2} of Theorem \ref{t:complex}, with probability $1 - O\left(\left(\epsilon^3 n\right)^{-1}\right) \geq 1 -O\left(\omega^{-1}\right)$ there are no small complex components in $G(n,n,p)$. 

Now, by Lemma \ref{l:vertices}, with probability $1-O\left(\omega^{-1}\right)$ the number of vertices in small unicyclic components is at most $\frac{\omega}{\delta} \ll n^{\frac{2}{3}}$ and the number of vertices in small tree components $Y(-1)$ is such that
 \[
 \frac{2(1-\epsilon')}{1+\epsilon}n - \frac{n^{\frac{2}{3}}}{100} = \frac{2(1-\epsilon')}{1+\epsilon}n - \frac{\omega\sqrt{n}}{\sqrt{\epsilon}} \leq Y(-1) \leq \frac{2(1-\epsilon')}{1+\epsilon}n + \frac{\omega\sqrt{n}}{\sqrt{\epsilon}} = \frac{2(1-\epsilon')}{1+\epsilon}n + \frac{n^{\frac{2}{3}}}{100}.
 \]
  
It follows that with probability $1-O\left(\omega^{-1}\right)$
  \[ \left||\mathscr{L}(G(n,n,p))|-2n\left(1-\frac{1-\epsilon'}{1+\epsilon}\right)\right|\leq \frac{n^\frac{2}{3}}{50}.
 \]
 Note that $\epsilon' = \epsilon + O(\epsilon^2)$, and so $|\mathscr{L}(G(n,n,p))| \approx 4 \epsilon n$.
 
In order to show the existence of a unique large component, we use a sprinkling argument. Let $$p_1 = p - \frac{n^{-\frac{4}{3}}}{10}\,\,\, \text{and }\,\,\, p_2 = \frac{p-p_1}{1-p_1} \geq \frac{n^{-\frac{4}{3}}}{20},$$ and let us write $p_1 = \frac{1+\epsilon_1}{n}$, where $\epsilon_1 = \epsilon - \frac{1}{10n^{\frac{1}{3}}}$. A standard argument allows us to couple an independent pair $(G(n,n,p_1),G(n,n,p_2))$ with $G(n,n,p)$ so that $G(n,n,p_1)\cup G(n,n,p_2) = G(n,n,p)$. 

It is clear that $\omega =\frac{\left(\epsilon^3 n\right)^\frac{1}{6}}{100} \approx \frac{\left(\epsilon_1^3 n\right)^\frac{1}{6}}{100}$. Hence, the same argument as before shows that with probability $1-O\left(\omega^{-1}\right)$
\[ 
\left||\mathscr{L}(G(n,n,p_1))|-2n\left(1-\frac{1-\epsilon_1'}{1+\epsilon_1}\right)\right|\leq \frac{n^\frac{2}{3}}{50},
\]
where $\epsilon_1'$ is defined as the solution to $(1-\epsilon_1')e^{\epsilon_1'} = (1+\epsilon_1)e^{-\epsilon_1}$.
 
Next, some basic analysis tells us that these bounds for $|\mathscr{L}(G(n,n,p))|$ and $|\mathscr{L}(G(n,n,p_1))|$ are not far apart. More precisely, we claim that
 \[
 \frac{1-\epsilon'_1}{1+\epsilon_1}-\frac{1-\epsilon'}{1+\epsilon} \leq \frac{2}{5n^\frac{1}{3}}.
 \]

Indeed, consider the function $y(x)$ where $y$ is given as the unique positive solution to $(1-y)e^y=(1+x)e^{-x}$. Then, by the derivative of implicit functions formula, 
\begin{equation}\label{e:implicitderivative}\frac{dy}{dx}= \frac{x}{y}e^{-x-y}= \frac{x(1-y)}{y(1+x)},
\end{equation}
where the last equality follows from the fact that $(1-y)e^y=(1+x)e^{-x}$.

Thus, by the mean value theorem there is a $\psi\in [\epsilon_1,\epsilon]$ such that
    \[\epsilon'-\epsilon_1'=y(\epsilon)-y(\epsilon_1)=y'(\psi)(\epsilon-\epsilon_1)<2(\epsilon-\epsilon_1),
    \]
    since \[y'(\psi) = \frac{\psi(1-y(\psi))}{y(\psi)(1+\psi)} <\frac{\psi}{y(\psi)}=\frac{\psi}{\psi+O(\psi^2)}\leq 2.\]
    
Hence, it follows that
\begin{align*}
\frac{1-\epsilon'_1}{1+\epsilon_1}-\frac{1-\epsilon'}{1+\epsilon} &= \frac{\epsilon  - \epsilon_1 + \epsilon' - \epsilon_1' - \epsilon_1' \epsilon +\epsilon'\epsilon_1}{(1+\epsilon)(1+\epsilon_1)} \leq 3(\epsilon - \epsilon_1) +\epsilon'\epsilon_1- \epsilon_1' \epsilon \leq 3(\epsilon - \epsilon_1) + \epsilon(\epsilon' - \epsilon_1') \leq (3+2\epsilon) (\epsilon - \epsilon_1) \leq \frac{4}{10 n^{\frac{1}{3}}}.
\end{align*}

Hence, with probability $1-O\left(\omega^{-1}\right)$
 \begin{align*}
      |\mathscr{L}(G(n,n,p))|-|\mathscr{L}(G(n,n,p_1))|\leq  \frac{1-\epsilon'_1}{1+\epsilon_1}2n-\frac{1-\epsilon'}{1+\epsilon}2n +\frac{n^\frac{2}{3}}{25}\leq \frac{4}{5}n^\frac{2}{3}+\frac{n^\frac{2}{3}}{25}< n^\frac{2}{3}.
 \end{align*}
 
Since, by our coupling, $G(n,n,p_1) \subseteq G(n,n,p)$, it follows that in this event every large component of $G(n,n,p)$ contains a large component of $G(n,n,p_1)$. Hence, in order to show that there is a unique large component in $G(n,n,p)$ it is sufficient to show that all the large components in $G(n,n,p_1)$ are contained in a single component in $G(n,n,p)$.

 By Lemma \ref{l:balanced}, with probability $1 - O\left(n^{-1}\right) \geq 1 -O\left( \omega^{-1} \right)$, each component of order larger than $n^{\frac{2}{3}}$ in $G(n,n,p_1)$ is balanced, and so we can partition the vertices in $\mathscr{L}(G(n,n,p_1))$ into subsets $V_1,W_1,V_2,W_2,\ldots V_m,W_m$ such that $\frac{n^\frac{2}{3}}{3}\leq |V_i| , |W_i|\leq n^\frac{2}{3}$ and $V_i$ and $W_i$ lie in the same component in $G(n,n,p_1)$ for each $i$, say in a greedy manner.
 
 Now, let us consider the edges in $G(n,n,p_2)$. Either all vertices in $\mathscr{L}(G(n,n,p_1))$ are contained in one component of $G(n,n,p_1) \cup G(n,n,p_2)$, or there is a family $\mathcal{A}=\left\{(V_{i_1},W_{i_1}),(V_{i_2},W_{i_2}),\dots,(V_{i_{r}},W_{i_r})\right\}$, where $1\leq r\leq \frac{m}{2}$ such that there is no edge in $G(n,n,p_2)$ with one end point in $(V_i,W_i)\in \mathcal{A}$ and the other in $(V_j,W_j)\notin \mathcal{A}$ (see Figure \ref{f:noedges}). Note that, for any such family $\mathcal{A}$, there are at least $\frac{2}{9}r(m-r)n^{\frac{4}{3}}$ non-edges in $G(n,n,p_2)$ with one end point in $(V_i,W_i)\in \mathcal{A}$ and the other in $(V_j,W_j)\notin \mathcal{A}$.
 
 Hence, the probability that such a family $\mathcal{A}$ exists is bounded by 
 \begin{align*}
     \sum_{r=1}^{\frac{m}{2}} \binom{m}{r} (1-p_2)^{\frac{2}{9}r(m-r)n^{\frac{4}{3}}}\leq \sum_{r=1}^{\frac{m}{2}}\left(\frac{em}{r}\right)^r \left(1-\frac{n^{-\frac{4}{3}}}{20}\right)^{\frac{2}{9}r(m-r)n^{\frac{4}{3}}}\leq \sum_{r=1}^{\frac{m}{2}}\left(\frac{em}{r}e^{-\frac{m-r}{100}}\right)^r\leq \sum_{r=1}^{\frac{m}{2}}\left(\frac{em}{r}e^{-\frac{m}{200}}\right)^r.
 \end{align*}
 However, since $|\mathscr{L}(G(n,n,p_1))| \approx 4\epsilon n$, it follows that $m = \Theta\left(\epsilon n^{\frac{1}{3}}\right) = \Theta(\omega^2)$, 
  and hence 
 \[
 \sum_{r=1}^{\frac{m}{2}}\left(\frac{em}{r}e^{-\frac{m}{200}}\right)^r = e^{-\Omega(m)} =  O\left(\omega^{-1}\right).
 \]
It follows that, with probability $1- O\left(\omega^{-1}\right)$, $\mathscr{L}(G(n,n,p))$ consists of just the vertices in the largest component $L_1\left(G(n,n,p)\right)$, and so the claim follows.

For the last part, since with probability $1-O\left(\omega^{-1}\right)$,
 \[
 |L_1| \approx \frac{2(\epsilon + \epsilon')}{1+\epsilon}n \approx 4 \epsilon n,
 \] 
 and by Theorems \ref{t:trees} and \ref{t:unicyclic}, with probability $1-O\left(\omega^{-1}\right)$, there are no large tree or unicyclic components, it suffices to show that with sufficiently small probability there are no complex components in $G(n,n,p)$ of order around $4 \epsilon n$ which are very unbalanced. We shall bound from above the expected number of complex components $C$ of $G(n,n,p)$ with order in the interval $\left[3\epsilon n,5\epsilon n\right]$, which have $|C\cap N_1| \geq (1+2\sqrt{\epsilon})|C\cap N_2|$ or $|C\cap N_2|\geq (1+2\sqrt{\epsilon})|C\cap N_1|$. As in Lemma \ref{l:balanced} we can bound the expected number of such components by the expected number of trees with $2$ extra edges, otherwise disconnected from the rest of the graph, which can be bounded as in Lemma \ref{l:balanced}
\begin{align*}
    &\sum_{k=3\epsilon n}^{5\epsilon n}2\sum_{\substack{i+j=k,\\j\geq (1+2\sqrt{\epsilon})i}}\binom{n}{i}\binom{n}{j}i^{j+1}j^{i+1}p^{k+1}(1-p)^{kn-2ij}\leq \frac{2}{n} \sum_{k=3\epsilon n}^{5\epsilon n}\sum_{\substack{i+j=k,\\j\geq (1+2\sqrt{\epsilon})i}}(ij)^\frac{1}{2}\left(\frac{i}{j}\right)^{j-i}\exp\left(\frac{(1+2\epsilon)ij}{n} + O\left(\frac{k}{n}\right)\right)\\
 &\lesssim \frac{2}{n}\sum_{k=3\epsilon n}^{5\epsilon n} \left(\frac{1}{1+2\sqrt{\epsilon}}\right)^{\frac{\sqrt{\epsilon}}{1+\sqrt{\epsilon}}k}\exp\left(\frac{(1+2\epsilon)k^2}{4n} \right) \sum_{\substack{i+j=k,\\j\geq (1+2\sqrt{\epsilon})i}}  (ij)^\frac{1}{2}\\ &\leq \frac{2}{n}\sum_{k=3\epsilon n}^{5\epsilon n} k^2 \exp \left( -\frac{2 \epsilon}{(1+2\sqrt{\epsilon})(1+\sqrt{\epsilon})} k + \frac{5(1+2\epsilon) \epsilon }{4}k \right)\leq \frac{2}{n}\sum_{k=3\epsilon n}^{5\epsilon n} k^2 e^{-\Omega( \epsilon k)}= O\left( \frac{1}{\epsilon^3 n} \right)= O\left(\frac{1}{\omega}\right),
\end{align*}
where we used that $\frac{i}{j}\leq \frac{1}{1+2\sqrt{\epsilon}}$ and $j-i\geq \frac{\sqrt{\epsilon}}{1+\sqrt{\epsilon}}k$.
Hence, the result follows from Markov's inequality.
\end{proof} 
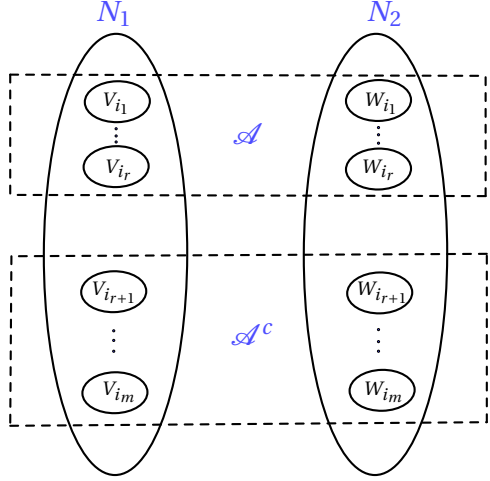
\begin{figure}
    \centering
    \definecolor{ududff}{rgb}{0.30196078431372547,0.30196078431372547,1.}
\begin{tikzpicture}[line cap=round,line join=round,>=triangle 45,x=1.0cm,y=1.0cm, scale=0.8]
\clip(2.,-2.) rectangle (11.,6.2);
\draw [rotate around={1.36392753160295:(4.36,4.53)},line width=0.8pt] (4.36,4.53) ellipse (0.5405627775045765cm and 0.3401589575822978cm);
\draw [rotate around={1.36392753160292:(4.34,3.45)},line width=0.8pt] (4.34,3.45) ellipse (0.5405627775045835cm and 0.3401589575823023cm);
\draw [rotate around={1.3639275316028892:(8.7,4.55)},line width=0.8pt] (8.7,4.55) ellipse (0.5405627775045685cm and 0.34015895758229314cm);
\draw [rotate around={1.36392753160292:(8.7,3.41)},line width=0.8pt] (8.7,3.41) ellipse (0.5405627775045974cm and 0.3401589575823112cm);
\draw [rotate around={1.3639275316029051:(4.3,1.37)},line width=0.8pt] (4.3,1.37) ellipse (0.5405627775045867cm and 0.34015895758230447cm);
\draw [rotate around={1.363927531602919:(4.32,-0.31)},line width=0.8pt] (4.32,-0.31) ellipse (0.5405627775046186cm and 0.3401589575823244cm);
\draw [rotate around={1.3639275316029202:(8.76,-0.27)},line width=0.8pt] (8.76,-0.27) ellipse (0.5405627775046397cm and 0.34015895758233794cm);
\draw [rotate around={1.36392753160292:(8.72,1.35)},line width=0.8pt] (8.72,1.35) ellipse (0.5405627775045548cm and 0.34015895758228437cm);
\draw [line width=0.8pt,dash pattern=on 3pt off 3pt] (2.58,4.98)-- (10.48,4.96);
\draw [line width=0.8pt,dash pattern=on 3pt off 3pt] (10.48,2.96)-- (10.48,4.96);
\draw [line width=0.8pt,dash pattern=on 3pt off 3pt] (10.48,2.96)-- (2.58,2.98);
\draw [line width=0.8pt,dash pattern=on 3pt off 3pt] (2.58,4.98)-- (2.58,2.98);
\draw [line width=0.8pt,dash pattern=on 3pt off 3pt] (2.6,1.96)-- (2.58,-0.86);
\draw [line width=0.8pt,dash pattern=on 3pt off 3pt] (2.58,-0.86)-- (10.5,-0.8);
\draw [line width=0.8pt,dash pattern=on 3pt off 3pt] (10.5,-0.8)-- (10.5,2.);
\draw [line width=0.8pt,dash pattern=on 3pt off 3pt] (10.5,2.)-- (2.6,1.96);
\draw [rotate around={89.83488293688639:(4.33,1.99)},line width=0.8pt] (4.33,1.99) ellipse (3.668886578188433cm and 1.191523698300298cm);
\draw [rotate around={89.83488293688639:(8.65,1.99)},line width=0.8pt] (8.65,1.99) ellipse (3.668886578188457cm and 1.191523698300306cm);
\begin{scriptsize}
\draw[color=black] (4.34,4.5) node {$V_{i_1}$};
\draw[color=black] (4.39,3.4) node {$V_{i_r}$};
\draw[color=black] (8.75,4.5) node {$W_{i_1}$};
\draw[color=black] (8.71,3.4) node {$W_{i_r}$};
\draw[color=black] (4.31,1.36) node {$V_{i_{r+1}}$};
\draw[color=black] (4.39,-0.3) node {$V_{i_m}$};
\draw[color=black] (8.79,-0.3) node {$W_{i_m}$};
\draw[color=black] (8.77,1.36) node {$W_{i_{r+1}}$};
\draw [fill=ududff] (4.3,0.38) circle (0.5pt);
\draw [fill=ududff] (4.3,0.56) circle (0.5pt);
\draw [fill=ududff] (4.3,0.72) circle (0.5pt);
\draw [fill=ududff] (8.7,0.74) circle (0.5pt);
\draw [fill=ududff] (8.7,0.54) circle (0.5pt);
\draw [fill=ududff] (8.7,0.38) circle (0.5pt);
\draw [fill=ududff] (8.7,3.84) circle (0.5pt);
\draw [fill=ududff] (8.7,3.98) circle (0.5pt);
\draw [fill=ududff] (8.7,4.1) circle (0.5pt);
\draw [fill=ududff] (4.34,3.86) circle (0.5pt);
\draw [fill=ududff] (4.34,4.08) circle (0.5pt);
\draw [fill=ududff] (4.34,3.96) circle (0.5pt);
\draw[color=ududff] (4.3,6) node {\large $N_1$};
\draw[color=ududff] (8.8,6) node {\large $N_2$};
\draw[color=ududff] (8.65,6.34) node {$U_{1}$};
\draw[color=ududff] (6.54,4) node {\large $\mathcal{A}$};
\draw[color=ududff] (6.61,0.66) node {\large $\mathcal{A}^c$};
\end{scriptsize}
\end{tikzpicture}
    \caption{A partition of the vertices in $\mathscr{L}(G(n,n,p_1))$ into $\mathcal{A}$ and $\mathcal{A}^c$ with no edges between $V_{i_s} \in \mathcal{A}$ and $W_{i_t} \in \mathcal{A}^c$ or between $V_{i_s} \in \mathcal{A}^c$ and $W_{i_t} \in \mathcal{A}$.}
    \label{f:noedges}
\end{figure}

\subsection{The excess of the giant component: proof of Theorem~\ref{t:excess}}\label{s:excess}
Using Theorem \ref{t:giant}, we can quite easily give a bound on the excess of the giant component which is of the correct asymptotic order. Indeed, we can bound the order of the giant component in quite a small interval, and then using Theorem \ref{t:complexcomponents} we can bound the probability that any component of this order has too large an excess. This is enough to show that whp the excess of the giant component is $O\left(\epsilon^3n \right)$. We formalise this in Lemma \ref{l:excess}.

 Note that, this can be seen to be of the correct order by a simple sprinkling argument: If we take $p_1 = \frac{1+\frac{\epsilon}{2}}{n}$ and $p_2 = \frac{p-p_1}{1-p_1} \geq \frac{\epsilon}{2n}$ then our previous results imply that, for an appropriate range of $\epsilon$, whp there is a giant component of order $\Theta(\epsilon n)$ in $G(n,n,p_1)$ which is equally distributed across the partition classes. However, then whp there are $\Theta\left( (\epsilon n)^2p_2\right) = \Theta\left( \epsilon^3 n\right)$ many edges of $G(n,n,p_2)$ on the vertex set of the giant component.

In order to find the \emph{correct leading constant}, we follow an argument of {\L}uczak \cite{Luczak} and use a multi-round exposure argument, starting with a supercritical $p'$ which is significantly smaller than $p$. By our weaker bound on the excess we can show that at the start of our process the excess of the giant component in $G(n,n,p')$ is $o\left(\epsilon^3 n\right)$, and we can also estimate quite precisely the change in the excess of the giant component between each stage of the multi-round exposure as we increase $p'$ to $p$, giving us an asymptotically tight bound on the excess of the giant component.

So, let us begin by deriving our weak upper bound on the excess of the giant component.

\begin{lemma}\label{l:excess}
Let $\epsilon = \epsilon(n) >0$ be such that $\epsilon^3 n 
\rightarrow \infty$ and $\epsilon 
=o(1)$, and let $p = \frac{1+\epsilon}{n}$. Then with probability $1 - O\left( (\epsilon^3 n)^{-\frac{1}{6}}\right)$ the excess of the largest component in $G(n,n,p)$ is $O\left(\epsilon^3 n\right)$.
\end{lemma}

\begin{proof}
We first note that, by Theorem \ref{t:giant}, with probability $1-O\left(\left(\epsilon^3 n\right)^{-\frac{1}{6}}\right)$, the largest component $L_1$ of $G(n,n,p)$ is balanced and satisfies 
\[
\left| L_1 -  \frac{2(\epsilon + \epsilon')}{1+\epsilon} n\right| < \frac{n^{\frac{2}{3}}}{50}.
\]

Let $X$ be the number of balanced components in $G(n,n,p)$ of order between
\[
\frac{2(\epsilon + \epsilon')}{1+\epsilon} n - \frac{n^{\frac{2}{3}}}{50}=:k_1 \quad\text{and}\quad k_2:= \frac{2(\epsilon + \epsilon')}{1+\epsilon}n + \frac{n^{\frac{2}{3}}}{50},
\]
which have excess at least $C\epsilon^3n$, where we will choose $C$ sufficiently large later. Then $\mathbb{E}(X)$ can be bounded above using Theorem \ref{t:complexexpectation} as
 \begin{align*}
    \mathbb{E}(X) & \lesssim  \sum_{k = k_1}^{k_2} \frac{1}{\sqrt{k}} \exp\left(-\delta k + \frac{\epsilon k^2}{4n}\right)\sum_{(i,j) \in B_k} \frac{1}{\sqrt{ij}}\left(\frac{i}{j}\right)^{j-i} \exp\left(-\frac{(i-j)^2}{2n} \right) \sum_{\ell = C \epsilon^3 n}^{ij-k} \left(\frac{ck^3(1+\epsilon)^2e^{\frac{2(1+\epsilon)}{n}}}{\ell n^2}\right)^{\frac{\ell}{2}},
     \end{align*}
     where $B_k$ is as before the set of balanced pairs $(i,j)$, since for $k_1 \leq k \leq k_2$ we have $\frac{k}{n} = o(1)$.

Let us first deal with the innermost sum. Since $k = \Theta(\epsilon n)$, for large enough $C$ we can bound
\[
\sum_{\ell = C \epsilon^3 n}^{ij-k} \left(\frac{ck^3(1+\epsilon)^2e^{\frac{2(1+\epsilon)}{n}}}{\ell n^2}\right)^{\frac{\ell}{2}} \leq \sum_{\ell = C \epsilon^3 n}^{ij-k} \left(\frac{1}{e^2}\right)^{\frac{\ell}{2}} = O\left( e^{-C \epsilon^3 n} \right).
\]
The middle sum can be dealt with by Lemma \ref{l:difference} as usual to see that
\[
\sum_{(i,j) \in B_k} \frac{1}{\sqrt{ij}}\left(\frac{i}{j}\right)^{j-i} \exp\left(-\frac{(i-j)^2}{2n} \right) = O\left( k^{-\frac{1}{2}} \right).
\]
Therefore, we can bound
\begin{align*}
 \mathbb{E}(X) = O\left(\sum_{k = k_1}^{k_2} \frac{1}{k} \exp\left(-\delta k + \frac{\epsilon k^2}{4n} - C \epsilon^3 n\right)  \right).
\end{align*}

However, since $k=\Theta(\epsilon n)$, and so both $\delta k$ and $\frac{\epsilon k^2}{n}$ are $O(\epsilon^3n)$, for $C$ large enough
\[
 \mathbb{E}(X) =O\left(\frac{k_2 - k_1}{\epsilon n} e^{-\Omega(\epsilon^3 n)}  \right) = O\left( \frac{1}{\epsilon n^{\frac{1}{3}}}e^{-\Omega(\epsilon^3 n)}  \right) = o(1).
\]
\end{proof} 

Using Lemma \ref{l:excess}, we can then determine asymptotically the excess of the giant component. As previously mentioned, we will argue via a multi-round exposure argument, taking a sequence $p_0 \leq p_1 \leq \ldots\leq p_s$ of probabilities such that $p_0$ is supercritical, but significantly smaller than $p$, and $p_s=p$. Via a standard coupling argument, we can think of sampling $G(n,n,p_0)$ and then sampling an independent sequence of bipartite random graphs $G(n,n,p'_i)$ where $p'_i =\frac{p_{i+1}-p_i}{1-p_i}$ so that for each $1\leq i \leq s$
\[
G(n,n,p_0) \cup \left(\bigcup_{j=0}^{i-1} G(n,n,p'_j)\right) \hspace{0.2cm}  \sim \hspace{0.2cm} G(n,n,p_i).
\]
This gives the inclusions $G(n,n,p_0) \subseteq G(n,n,p_1) \subseteq \ldots \subseteq G(n,n,p_s)$.

Our choice of $p_0$, together with Theorem \ref{t:excess}, guarantees that the excess of $L_1\left(G(n,n,p_0)\right)$ is significantly smaller than $\epsilon^3 n$. We then estimate precisely the change in the excess of the giant component in each of the sprinkling steps. To do so, we bound whp from above and below the number $\Delta_i$ of extra excess edges in the giant component when adding each $G(n,n,p'_i)$. Here, it is essential that the probability of failure in each step is small enough that the sum of these probabilities over all $0 \leq i \leq s-1$ is still small. Using that the excess of $L_1\left(G(n,n,p_0)\right)$ is significantly smaller than $\epsilon^3 n$, we can then asymptotically determine the excess of $L_1\left(G(n,n,p_s)\right)$ as a sum of the $\Delta_i$, which we can approximate by an integral.

\begin{theorem}\label{t:excessprecise}
Let $\epsilon = \epsilon(n) > 0$ be such that $\epsilon^3 n \gg \omega \rightarrow \infty$ and $\epsilon \leq \frac{1}{\omega}$, and let $p = \frac{1+ \epsilon}{n}$. Then with probability  $1-O\left( \omega^{-0.05}\right)$
\[
\text{excess}\left(L_1\left(G(n,n,p)\right)\right) \approx \frac{4}{3} \epsilon^3 n.
\]
\end{theorem}

\begin{proof}
 For each $i\in \mathbb{N}\cup \left\{0\right\}$, let \[\epsilon_i=\omega^{0.3}n^{-\frac{1}{3}}\left(1+\omega^{- 0.1}\right)^{i} \,\,\, \text{ and } \,\,\,  p_i=\frac{1+\epsilon_i}{n}.
 \]
Throughout the proof we work under the assumption that $i$ is small enough so that $\epsilon_i = o(1)$.
 
By a standard coupling argument we can think of moving from $G(n,n,p_i)$ to $G(n,n,p_{i+1})$ via sprinkling. That is, we choose independently for each $i$ a random graph $G(n,n,p'_i)$ where
$$p'_i=\frac{p_{i+1}-p_i}{1-p_i}=\frac{\epsilon_{i+1}-\epsilon_i}{n-1-\epsilon_i},$$
in such a way that $G(n,n,p_{i+1})=G(n,n,p_i)\cup G(n,n,p'_i)$ for each $i$. We note that, if we write $L_{1,i}$ for the largest component of $G(n,n,p_i)$ for each $i$, then by Theorem \ref{t:giant}, 
\begin{equation}
    |L_{1,i}\cap N_1|\approx \frac{\epsilon_i+\epsilon_i'}{1+\epsilon_i}n \,\,\,\,\text{    and    }\,\,\,\, |L_{1,i}\cap N_2|\approx \frac{\epsilon_i+\epsilon_i'}{1+\epsilon_i}n,\label{e:giantsize}
\end{equation}
with probability $1-O\left(\left(\epsilon_i^3n\right)^{-\frac{1}{6}}\right)$. Furthermore, by Lemma \ref{l:excess} with probability $1-O\left(\left(\epsilon_i^3n\right)^{-\frac{1}{6}}\right)$,
\begin{equation}
    a_i:=\text{excess}\left(L_{1,i}\right) = O\left(\epsilon_i^3 n \right). \label{e:giantexcess}
\end{equation}

Note that, by \eqref{e:giantexcess}, with probability $1-O\left(\left(\epsilon_0^3n\right)^{-\frac{1}{6}}\right)=1-O\left(\omega^{-0.05}\right)$, $a_0 = O\left(\omega^{0.9}\right) = o\left(\epsilon^3 n\right)$, and so to begin with we may assume that the excess is much smaller than $\epsilon^3 n$.

We show that we can control quite precisely how the excess of the giant component changes in each sprinkling step. More precisely, we claim that for each $i$, with probability $1-O\left(\left(\epsilon_i^3n\right)^{-\frac{1}{6}}\right)$
\begin{align}
     \Delta_i:=a_{i+1}-a_i \approx \frac{(\epsilon_i+\epsilon'_i)^2}{(1+\epsilon_i)^2}n(\epsilon_{i+1}-\epsilon_i).\label{e:excesschange}
 \end{align}
In order to show \eqref{e:excesschange} we bound from above and below the number of new excess edges added in step $i+1$.

\begin{claim}\label{c:lower}
With probability $1-O\left(\left(\epsilon_i^3n\right)^{-\frac{1}{6}}\right)$,
\begin{equation*}
\Delta_i \gtrsim\frac{(\epsilon_i+\epsilon_i')^2}{(1+\epsilon_i)^2}n(\epsilon_{i+1}-\epsilon_i).
\end{equation*}
\end{claim}

 \begin{proof}[Proof of Claim \ref{c:lower}]
We note that every edge in $G(n,n,p'_i) \setminus E(L_{1,i})$ which has both ends in $L_{1,i}$ adds one to the quantity $\Delta_i$  (see Figure \ref{f:lowerbound}). 
\begin{figure}
\centering
\definecolor{ududff}{rgb}{0.30196078431372547,0.30196078431372547,1.}
\begin{tikzpicture}[line cap=round,line join=round,>=triangle 45,x=1.0cm,y=1.0cm, scale=0.8]
\clip(2.,0.1) rectangle (12.,4.2);
\draw [->,line width=0.8pt] (5.92,2.44) -- (7.56,2.44);
\draw [rotate around={17.878696595841344:(4.1,2.54)},line width=0.8pt] (4.1,2.54) ellipse (1.2946594190723053cm and 1.1188132155961696cm);
\draw [rotate around={17.87869659584132:(9.5,2.58)},line width=0.8pt] (9.5,2.58) ellipse (1.2946594190723064cm and 1.1188132155961705cm);
\draw [rotate around={34.032127817338285:(9.45,2.5)},line width=0.8pt] (9.45,2.5) ellipse (1.7393127045374825cm and 1.4703430498239176cm);
\draw [line width=0.8pt] (9.5,1.92)-- (10.34,2.74);
\begin{scriptsize}
\draw[color=ududff] (3.5,2.8) node {$L_{1,i}$};
\draw [fill=ududff] (4.,2.) circle (2.5pt);
\draw[color=ududff] (3.8,1.8) node {$u$};
\draw [fill=ududff] (4.9,2.72) circle (2.5pt);
\draw[color=ududff] (4.95,2.96) node {$v$};
\draw[color=ududff] (8.88,2.8) node {$L_{1,i}$};
\draw[color=ududff] (7.9,3.7) node {$L_{1,i+1}$};
\draw [fill=ududff] (10.34,2.74) circle (2.5pt);
\draw[color=ududff] (10.35,2.98) node {$v$};
\draw [fill=ududff] (9.5,1.92) circle (2.5pt);
\draw[color=ududff] (9.4,1.68) node {$u$};
\draw[color=black] (9.5,0.3) node {step $i+1$};
\draw[color=black] (4,0.3) node {step $i$};
\end{scriptsize}
\end{tikzpicture}
    \caption{In step $i+1$ every edge $uv$ in $G(n,n,p'_i) \setminus E(L_{1,i})$ with $u,v \in V\left(L_{1,i}\right)$ contributes to $\Delta_i$.}
     \label{f:lowerbound}
\end{figure}
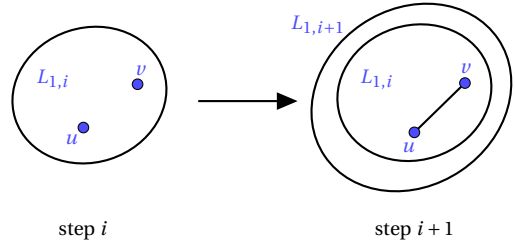
Hence, by \eqref{e:giantsize} and \eqref{e:giantexcess} with probability $1-O\left(\left(\epsilon_i^3n\right)^{-\frac{1}{6}}\right)$, $\Delta_i$ stochastically dominates a binomial random variable $Y\sim$ Bin$(m,q)$ with parameters
 \[
 m \approx\left(\frac{\epsilon_i+\epsilon_i'}{1+\epsilon_i}n\right)^2 - 2\frac{\epsilon_i+\epsilon_i'}{1+\epsilon_i}n - O(\epsilon_i^3 n)\,\,\,\, \text{    and    }\,\,\,\, q = p'_i.
\] 

Now, we see that
\[
\mathbb{E}(Y) \gtrsim \frac{(\epsilon_i+\epsilon_i')^2}{(1+\epsilon_i)^2}n^2 \frac{\epsilon_{i+1}-\epsilon_i}{n-1-\epsilon_i} \approx \frac{(\epsilon_i+\epsilon_i')^2}{(1+\epsilon_i)^2}n(\epsilon_{i+1}-\epsilon_i),
\]
and so $\mathbb{E}(Y) = \Omega(\epsilon_i^3 n)$. Hence, by Lemma \ref{l:chernoff} we obtain
\[
\mathbb{P}\left(|\mathbb{E}(Y)-Y| \geq \frac{\mathbb{E}(Y)}{(\epsilon_i^3 n)^{\frac{1}{4}}} \right) \leq \exp \left( -\Omega\left( \frac{\mathbb{E}(Y)}{\left(\epsilon_i^3 n\right)^{\frac{1}{2}}} \right)\right) \leq \exp \left( -\Omega\left(\left(\epsilon_i^3 n\right)^{\frac{1}{2}} \right)\right) = O\left(\left(\epsilon_i^3n\right)^{-\frac{1}{6}}\right).
\]

Hence, with probability $1-O\left(\left(\epsilon_i^3n\right)^{-\frac{1}{6}}\right)$, we get
\[
Y \gtrsim \mathbb{E}(Y) \gtrsim\frac{(\epsilon_i+\epsilon_i')^2}{(1+\epsilon_i)^2}n(\epsilon_{i+1}-\epsilon_i),
\]
and so with at least this probability
\begin{equation*}
\Delta_i \gtrsim\frac{(\epsilon_i+\epsilon_i')^2}{(1+\epsilon_i)^2}n(\epsilon_{i+1}-\epsilon_i).
\end{equation*}
\end{proof}

\begin{claim}\label{c:upper}
With probability $1-O\left(\left(\epsilon_i^3n\right)^{-\frac{1}{6}}\right)$
\begin{equation*}
\Delta_i \lesssim\frac{(\epsilon_{i}+\epsilon_{i}')^2}{(1+\epsilon_{i})^2}n(\epsilon_{i+1}-\epsilon_i).
\end{equation*}
\end{claim}
\begin{proof}[Proof of Claim \ref{c:upper}]
For an upper bound, we need to be slightly more careful. We note that there are two ways that edges in $G(n,n,p'_i)$ can contribute to $\Delta_i$. Firstly, edges in $G(n,n,p'_i)$ which have both endpoints in $L_{1,i+1}$ can add one to this quantity. However, there are some other edges, specifically excess edges in non-giant components of $G(n,n,p_i)$ which are joined to $L_{1,i+1}$ by an edge of $G(n,n,p'_i)$, which also add to this quantity (see Figure \ref{f:upperbound}).

We first show that the contribution from the former of these is approximately what we expect, and then show that the contribution from the latter is negligible.

\begin{figure}
    \centering
\definecolor{ududff}{rgb}{0.30196078431372547,0.30196078431372547,1.}
\begin{tikzpicture}[line cap=round,line join=round,>=triangle 45,x=1.0cm,y=1.0cm, scale=0.8]
\clip(0.,0.) rectangle (14.,4.45);
\draw [->,line width=0.8pt] (5.92,2.44) -- (7.56,2.44);
\draw [rotate around={17.878696595841326:(2.18,2.34)},line width=0.8pt] (2.18,2.34) ellipse (1.2946594190723082cm and 1.118813215596172cm);
\draw [line width=0.8pt] (11.7,1.94)-- (12.32,2.28);
\draw [rotate around={17.878696595841376:(9.5,2.34)},line width=0.8pt] (9.5,2.34) ellipse (1.2946594190723015cm and 1.118813215596167cm);
\draw [line width=0.8pt] (10.,2.54)-- (11.62,2.54);
\draw [rotate around={4.65335158726839:(10.44,2.4)},line width=0.8pt] (10.44,2.4) ellipse (2.450697426693055cm and 1.7400913416254773cm);
\draw [rotate around={26.565051177077805:(4.56,2.33)},line width=0.8pt] (4.56,2.33) ellipse (0.8189696528566512cm and 0.7471353908764746cm);
\draw [rotate around={26.565051177078097:(11.92,2.33)},line width=0.8pt] (11.92,2.33) ellipse (0.8189696528566617cm and 0.7471353908764856cm);
\draw [line width=0.8pt] (4.36,1.96)-- (4.98,2.3);
\begin{scriptsize}
\draw [fill=ududff] (11.7,1.94) circle (2.0pt);
\draw [fill=ududff] (12.32,2.28) circle (2.0pt);
\draw[color=ududff] (12.33,2.53) node {$y$};
\draw [fill=ududff] (10.,2.54) circle (2.0pt);
\draw[color=ududff] (9.9,2.3) node {$u$};
\draw [fill=ududff] (11.62,2.54) circle (2.0pt);
\draw[color=ududff] (11.68,2.78) node {$v$};
\draw[color=ududff] (8.91,2.2) node {$L_{1,i}$};
\draw[color=ududff] (1.6,2.2) node {$L_{1,i}$};
\draw[color=ududff] (8.3,3.8) node {$L_{1,i+1}$};
\draw[color=ududff] (4.4,1.74) node {$x$};
\draw[color=ududff] (11.8,1.74) node {$x$};
\draw [fill=ududff] (4.36,1.96) circle (2.0pt);
\draw [fill=ududff] (4.98,2.3) circle (2.0pt);
\draw[color=ududff] (4.99,2.54) node {$y$};
\draw [fill=ududff] (2.72,2.56) circle (2.0pt);
\draw[color=ududff] (2.6,2.3) node {$u$};
\draw [fill=ududff] (4.34,2.56) circle (2.0pt);
\draw[color=ududff] (4.44,2.78) node {$v$};
\draw[color=black] (3,0.2) node {step $i$};
\draw[color=black] (10.5,0.2) node {step $i+1$};
\end{scriptsize}
\end{tikzpicture}
    \caption{In step $i+1$ the only contribution to $\Delta_i$ comes from edges $uv$ in $ G(n,n,p'_i)$ with $u,v \in V\left(L_{1,i+1}\right)$ or excess edges $xy$ in components of $G(n,n,p_i)$ joined to $L_{1,i}$ by such an edge.}
     \label{f:upperbound}
\end{figure}
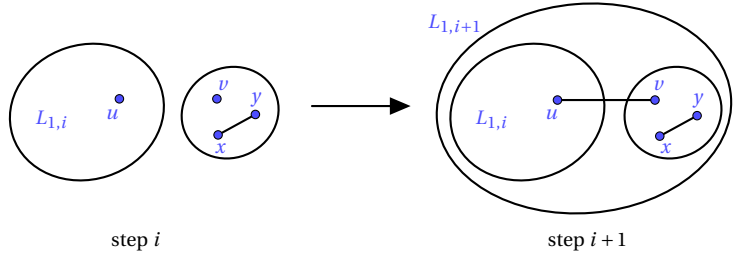
For the first of these, let $\mathcal{A}$ be the event that 
\[
\max \left\{ |C \cap N_{j}|  \colon C \text{ a component of } G(n,n,p_{i+1})\} \text{ for } j=1,2 \right\}\lesssim\frac{\epsilon_i+\epsilon_i'}{1+\epsilon_i}n.
\]
Note that,
\[
\frac{\epsilon_i+\epsilon_i'}{1+\epsilon_i}n \approx \frac{\epsilon_{i+1}+\epsilon_{i+1}'}{1+\epsilon_{i+1}}n.
\]
Then, by \eqref{e:giantsize}, $\mathbb{P}(\mathcal{A})\geq 1-O\left(\left(\epsilon^3_{i+1}n\right)^{-\frac{1}{6}}\right)$ and  $\mathcal{A}$ is a decreasing property. Thus,  by Harris'  inequality (Lemma \ref{l:Harris}), given any set of edges $F$ the probability that $F \subseteq G(n,n,p'_i)$ conditioned on $\mathcal{A}$ is strictly less than the probability that $F \subseteq G(n,n,p'_i)$. Hence, with probability $1-O\left(\left(\epsilon_i^3 n\right)^{-\frac{1}{6}}\right)$ the number of edges added to the vertex set of the new giant component is stochastically dominated by a binomial random variable Bin$(m',q'):=Z$ with parameters
\[
m'\approx\left(\frac{\epsilon_i+\epsilon_i'}{1+\epsilon_i}n\right)^2 \text{    and    } q' = p'_i.
\]

As before,  we have
\[
\mathbb{E}(Z) \lesssim\frac{(\epsilon_i+\epsilon_i')^2}{(1+\epsilon_i)^2}n^2. \frac{\epsilon_{i+1}-\epsilon_i}{n-1-\epsilon_i} \approx\frac{(\epsilon_i+\epsilon_i')^2}{(1+\epsilon_i)^2}n(\epsilon_{i+1}-\epsilon_i),
\]
and so again by Lemma \ref{l:chernoff} with probability $1-O\left(\left(\epsilon_i^3 n\right)^{-\frac{1}{6}}\right)$,
\[
Z \lesssim \mathbb{E}(Z) \lesssim \frac{(\epsilon_i+\epsilon_i')^2}{(1+\epsilon_i)^2}n(\epsilon_{i+1}-\epsilon_i).
\]

Now, let us bound the contribution to $\Delta_i$ from excess edges in non-giant components of $G(n,n,p_i)$. By Theorem \ref{t:complex} with probability $1- O\left(\left(\epsilon_i^3n\right)^{-\frac{1}{6}}\right)$ there are no complex components of order at most $n^{\frac{2}{3}}$ and by Theorem \ref{t:giant} with probability $1-O\left(\left(\epsilon_i^3n\right)^{-\frac{1}{6}}\right)$ there are no components apart from the giant component of order greater than $n^{\frac{2}{3}}$. Hence, it follows that with at least this probability every non-tree component in $G(n,n,p_i)$ except $L_{1,i}$ is unicyclic, and so the contribution to $\Delta_i$ from excess edges in non-giant components of $G(n,n,p_i)$ is equal to the number of unicyclic components in $G(n,n,p_i)$ which are joined to $L_{1,i}$ by $G(n,n,p'_i)$. We can bound this from above by the number of edges in $G(n,n,p'_i)$ which join such components to $L_{1,i}$.

Then, by Lemma \ref{l:vertices}, with probability $1-O\left(\left(\epsilon_i^3n\right)^{-\frac{1}{6}}\right)$ the number of vertices in unicyclic components of $G(n,n,p_i)$ is at most
\[
O\left(\frac{(\epsilon_i^3n)^{\frac{1}{6}}}{\epsilon_i^2} \right) = o\left( n^{\frac{2}{3}} \right).
\]
Hence, since rather crudely $|V(L_{1,i})| \leq 5 \epsilon_i n$, the expected number of edges in $G(n,n,p'_i)$ which connect unicyclic components in $G(n,n,p_i)$ to $L_{1,i}$ is less than 
 \begin{align*}
     5\epsilon_in n^\frac{2}{3}p_i' = O\left(\epsilon_i(\epsilon_{i+1}-\epsilon_i)n^\frac{2}{3}\right).
 \end{align*}
Then, by Markov's inequality, with probability $1-O\left(\left(\epsilon_i^3n\right)^{-\frac{1}{6}}\right)$ the number of such edges is at most
 \begin{align*}
    O\left(\epsilon_i(\epsilon_{i+1}-\epsilon_i)n^\frac{2}{3}\left(\epsilon_i^3n\right)^\frac{1}{6} \right)=o\left(\frac{(\epsilon_{i}+\epsilon_{i}')^2}{(1+\epsilon_{i})^2}n(\epsilon_{i+1}-\epsilon_i)\right).
 \end{align*}
 
It follows that, with probability $1-O\left(\left(\epsilon_i^3n\right)^{-\frac{1}{6}}\right)$
\begin{equation*}
\Delta_i \lesssim\frac{(\epsilon_{i}+\epsilon_{i}')^2}{(1+\epsilon_{i})^2}n(\epsilon_{i+1}-\epsilon_i).
\end{equation*}
\end{proof}

Hence, by Claims \ref{c:lower} and \ref{c:upper}, \eqref{e:excesschange} holds with probability $1-O\left(\left(\epsilon_i^3n\right)^{-\frac{1}{6}}\right)$. Therefore, by a union bound, \eqref{e:excesschange} holds for all $i \in \mathbb{N}$ such that $\epsilon_i = o(1)$ with probability 
 \[
 1-O\left(\sum_{i=0}^\infty \left(\epsilon_i^3n\right)^{-\frac{1}{6}}\right)\geq 1- O\left( \omega^{-0.05}\right),
 \]
which can be seen by noting that the sum is a geometric series.
 
Let $s \in \mathbb{N}$ be such that $\epsilon_{s-1} \leq \epsilon \leq \epsilon_s$. Then
 \begin{align}\label{eq: a_1}
     a_s - a_0 = \sum_{i=0}^{s-1}\Delta_i\approx \sum_{i=0}^{s-1}\frac{(\epsilon_i+\epsilon_i')^2}{(1+\epsilon_i)^2}n(\epsilon_{i+1}-\epsilon_i) \approx n\int_{\epsilon_0}^{\epsilon_s}\frac{(x+y)^2}{(1+x)^2}dx,
 \end{align}
 where $y=y(x)$ is implicitly given by $(1-y)e^y=(1+x)e^{-x}$, and we can approximate the sum by the integral since $\epsilon_{i+1}-\epsilon_i= o(1).$ 

 Note that, by \eqref{e:implicitderivative}, $\frac{dy}{dx}= \frac{x(1-y)}{y(1+x)}$ and so we can calculate $\frac{d}{dx}\left(\frac{x^2-y^2}{1+x}\right)=\frac{(x+y)^2}{(1+x)^2}$. Using this, and the fundamental theorem of calculus, we can conclude from \eqref{eq: a_1} that with probability ${1- O\left( \omega^{-0.05}\right)}$
 \begin{align*}
     a_s - a_0 &\approx n\left(\frac{\epsilon_s^2-\epsilon_s'^2}{1+\epsilon_s}-
     \frac{\epsilon_0^2-\epsilon_0'^2}{1+\epsilon_0}\right) \approx n \left(\frac{1}{1+\epsilon_s}\left(\frac{4}{3}\epsilon_s^3+O\left(\epsilon_s^4\right)\right)-\frac{\epsilon_0^2-\epsilon_0'^2}{1+\epsilon_0}\right)\approx\frac{4}{3}\epsilon_s^3 n,
 \end{align*}
 where we used that $\epsilon_s'=\epsilon_s-\frac{2}{3}\epsilon_s^2+O\left(\epsilon_s^3\right)$. Since, as previously mentioned, $a_0 = o(\epsilon^3 n)$ it follows that $a_s \approx \frac{4}{3}\epsilon_s^3 n$. A similar argument shows that $a_{s-1}\approx\frac{4}{3}\epsilon_{s-1}^3 n$. 
 
 Then, since 
 \[
 \epsilon_{s-1} \leq \epsilon \leq (1+ \omega^{-0.1})\epsilon_{s-1} \,\,\, \text{ and } \,\,\, \frac{\epsilon_s}{1+ \omega^{-0.1}} \leq \epsilon \leq \epsilon_s,
 \]
 and we can couple the three random bipartite graphs such that $G(n,n,p_{s-1}) \subseteq G(n,n,p) \subseteq G(n,n,p_s)$, it follows that 
\[
\text{excess}\left(L_1\left(G(n,n,p)\right)\right) \approx \frac{4}{3} \epsilon^3 n.
\]

\end{proof}

\begin{proof}[Proof of Theorem \ref{t:excess}]
The theorem follows directly from Theorem \ref{t:excessprecise}.
\end{proof}

\section{Counting bipartite graphs: proofs of Theorems~\ref{t:unicyliccomponents} and~\ref{t:complexcomponents}}\label{s:counting}

\subsection{Unicyclic bipartite graphs: proof of Theorem~\ref{t:unicyliccomponents}} Since a connected unicyclic graph is the union of a cycle and a forest, we are able to deduce Theorem \ref{t:unicyliccomponents} from the following formula for the number of bipartite forests, due to Moon \cite{M67}. 

\begin{lemma}[\cite{M67}]\label{l:bipartiteforest}
Given $i,j,s,t \in \mathbb{N}$ satisfying $s\leq i$ and $t \leq j$, let $F(i,j,s,t)$ denote the number of bipartite forests with partition classes $I = \{x_1,\ldots, x_i\}$ and $J= \{y_1,\ldots,y_j\}$ with $s + t$ components where the vertices $x_1,\ldots, x_s,y_1,\ldots,y_t$ belong to distinct components. Then 
\begin{align}\label{e:counting}
F(i,j,s,t)
= i^{j-t-1}j^{i-s-1}\left(sj + ti - st\right).
\end{align}
\end{lemma}

Using Lemma \ref{l:bipartiteforest}, we can prove Theorem \ref{t:unicyliccomponents}. 

\begin{proof}[Proof of Theorem~\ref{t:unicyliccomponents}]
We note that every connected unicyclic bipartite graph with $i$ vertices in one partition class and $j$ in the other contains a unique cycle, which has length $2r$ for some $r \leq \min \{i,j\}$, and if we delete the edges of this cycle, then what remains is a forest with $2r$ components, each meeting one vertex of the cycle (see Figure \ref{f:unicyclic}).

\begin{figure}
    \centering
   \definecolor{ududff}{rgb}{0.30196078431372547,0.30196078431372547,1.}
\definecolor{ffqqtt}{rgb}{1.,0.,0.2}
\begin{tikzpicture}[line cap=round,line join=round,>=triangle 45,x=1.0cm,y=1.0cm, scale=0.6]
\clip(4.8,1.8) rectangle (11.2,6.);
\fill[line width=0.8pt,color=white] (6.66,3.4) -- (7.02,3.4) -- (7.02,3.76) -- (6.66,3.76) -- cycle;
\fill[line width=0.8pt,color=white] (7.82,4.8) -- (8.18,4.8) -- (8.18,5.16) -- (7.82,5.16) -- cycle;
\fill[line width=0.8pt,color=white] (6.,4.84) -- (6.36,4.84) -- (6.36,5.2) -- (6.,5.2) -- cycle;
\fill[line width=0.8pt,color=white] (8.82,3.4) -- (9.18,3.4) -- (9.18,3.76) -- (8.82,3.76) -- cycle;
\fill[line width=0.8pt,color=white] (10.34,1.94) -- (10.7,1.94) -- (10.7,2.3) -- (10.34,2.3) -- cycle;
\fill[line width=0.8pt,color=white] (10.56,2.78) -- (10.92,2.78) -- (10.92,3.14) -- (10.56,3.14) -- cycle;
\fill[line width=0.8pt,color=white] (7.3,1.96) -- (7.66,1.96) -- (7.66,2.32) -- (7.3,2.32) -- cycle;
\fill[line width=0.8pt,color=white] (10.6,3.72) -- (10.96,3.72) -- (10.96,4.08) -- (10.6,4.08) -- cycle;
\fill[line width=0.8pt,color=white] (8.5,1.98) -- (8.86,1.98) -- (8.86,2.34) -- (8.5,2.34) -- cycle;
\fill[line width=0.8pt,color=white] (5.38,2.4) -- (5.74,2.4) -- (5.74,2.76) -- (5.38,2.76) -- cycle;
\fill[line width=0.8pt,color=white] (9.76,4.76) -- (10.12,4.76) -- (10.12,5.12) -- (9.76,5.12) -- cycle;
\fill[line width=0.8pt,color=white] (5.02,3.4) -- (5.38,3.4) -- (5.38,3.76) -- (5.02,3.76) -- cycle;
\draw [line width=0.8pt,color=ffqqtt] (6.66,3.4)-- (7.02,3.4);
\draw [line width=0.8pt,color=ffqqtt] (7.02,3.4)-- (7.02,3.76);
\draw [line width=0.8pt,color=ffqqtt] (7.02,3.76)-- (6.66,3.76);
\draw [line width=0.8pt,color=ffqqtt] (6.66,3.76)-- (6.66,3.4);
\draw [line width=0.8pt,color=ffqqtt] (7.82,4.8)-- (8.18,4.8);
\draw [line width=0.8pt,color=ffqqtt] (8.18,4.8)-- (8.18,5.16);
\draw [line width=0.8pt,color=ffqqtt] (8.18,5.16)-- (7.82,5.16);
\draw [line width=0.8pt,color=ffqqtt] (7.82,5.16)-- (7.82,4.8);
\draw [line width=0.8pt,color=ffqqtt] (6.,4.84)-- (6.36,4.84);
\draw [line width=0.8pt,color=ffqqtt] (6.36,4.84)-- (6.36,5.2);
\draw [line width=0.8pt,color=ffqqtt] (6.36,5.2)-- (6.,5.2);
\draw [line width=0.8pt,color=ffqqtt] (6.,5.2)-- (6.,4.84);
\draw [line width=0.8pt,color=ffqqtt] (8.82,3.4)-- (9.18,3.4);
\draw [line width=0.8pt,color=ffqqtt] (9.18,3.4)-- (9.18,3.76);
\draw [line width=0.8pt,color=ffqqtt] (9.18,3.76)-- (8.82,3.76);
\draw [line width=0.8pt,color=ffqqtt] (8.82,3.76)-- (8.82,3.4);
\draw [line width=0.8pt] (6.9,4.6)-- (7.82,4.98);
\draw [line width=0.8pt] (6.9,4.6)-- (6.86,3.76);
\draw [line width=0.8pt] (8.,3.)-- (8.82,3.4);
\draw [line width=0.8pt] (9.04,4.6)-- (8.18,4.98);
\draw [line width=0.8pt] (9.04,4.6)-- (9.,3.76);
\draw [line width=0.8pt] (8.,5.78)-- (8.,5.16);
\draw [line width=0.8pt] (9.9,3.54)-- (9.18,3.58);
\draw [line width=0.8pt] (9.54,2.86)-- (9.18,3.4);
\draw [line width=0.8pt] (6.04,3.28)-- (6.66,3.5);
\draw [line width=0.8pt,color=ffqqtt] (10.34,1.94)-- (10.7,1.94);
\draw [line width=0.8pt,color=ffqqtt] (10.7,1.94)-- (10.7,2.3);
\draw [line width=0.8pt,color=ffqqtt] (10.7,2.3)-- (10.34,2.3);
\draw [line width=0.8pt,color=ffqqtt] (10.34,2.3)-- (10.34,1.94);
\draw [line width=0.8pt,color=ffqqtt] (10.56,2.78)-- (10.92,2.78);
\draw [line width=0.8pt,color=ffqqtt] (10.92,2.78)-- (10.92,3.14);
\draw [line width=0.8pt,color=ffqqtt] (10.92,3.14)-- (10.56,3.14);
\draw [line width=0.8pt,color=ffqqtt] (10.56,3.14)-- (10.56,2.78);
\draw [line width=0.8pt,color=ffqqtt] (7.3,1.96)-- (7.66,1.96);
\draw [line width=0.8pt,color=ffqqtt] (7.66,1.96)-- (7.66,2.32);
\draw [line width=0.8pt,color=ffqqtt] (7.66,2.32)-- (7.3,2.32);
\draw [line width=0.8pt,color=ffqqtt] (7.3,2.32)-- (7.3,1.96);
\draw [line width=0.8pt,color=ffqqtt] (10.6,3.72)-- (10.96,3.72);
\draw [line width=0.8pt,color=ffqqtt] (10.96,3.72)-- (10.96,4.08);
\draw [line width=0.8pt,color=ffqqtt] (10.96,4.08)-- (10.6,4.08);
\draw [line width=0.8pt,color=ffqqtt] (10.6,4.08)-- (10.6,3.72);
\draw [line width=0.8pt,color=ffqqtt] (8.5,1.98)-- (8.86,1.98);
\draw [line width=0.8pt,color=ffqqtt] (8.86,1.98)-- (8.86,2.34);
\draw [line width=0.8pt,color=ffqqtt] (8.86,2.34)-- (8.5,2.34);
\draw [line width=0.8pt,color=ffqqtt] (8.5,2.34)-- (8.5,1.98);
\draw [line width=0.8pt,color=ffqqtt] (5.38,2.4)-- (5.74,2.4);
\draw [line width=0.8pt,color=ffqqtt] (5.74,2.4)-- (5.74,2.76);
\draw [line width=0.8pt,color=ffqqtt] (5.74,2.76)-- (5.38,2.76);
\draw [line width=0.8pt,color=ffqqtt] (5.38,2.76)-- (5.38,2.4);
\draw [line width=0.8pt,color=ffqqtt] (9.76,4.76)-- (10.12,4.76);
\draw [line width=0.8pt,color=ffqqtt] (10.12,4.76)-- (10.12,5.12);
\draw [line width=0.8pt,color=ffqqtt] (10.12,5.12)-- (9.76,5.12);
\draw [line width=0.8pt,color=ffqqtt] (9.76,5.12)-- (9.76,4.76);
\draw [line width=0.8pt,color=ffqqtt] (5.02,3.4)-- (5.38,3.4);
\draw [line width=0.8pt,color=ffqqtt] (5.38,3.4)-- (5.38,3.76);
\draw [line width=0.8pt,color=ffqqtt] (5.38,3.76)-- (5.02,3.76);
\draw [line width=0.8pt,color=ffqqtt] (5.02,3.76)-- (5.02,3.4);
\draw [line width=0.8pt] (6.04,3.28)-- (5.38,3.56);
\draw [line width=0.8pt] (6.04,3.28)-- (5.74,2.76);
\draw [line width=0.8pt] (6.36,4.84)-- (6.9,4.6);
\draw [line width=0.8pt] (5.4,5.48)-- (6.,5.2);
\draw [line width=0.8pt] (6.08,5.84)-- (6.16,5.2);
\draw [line width=0.8pt] (5.2,4.86)-- (6.,4.94);
\draw [line width=0.8pt] (9.04,4.6)-- (9.76,4.76);
\draw [line width=0.8pt] (10.12,5.12)-- (10.6,5.44);
\draw [line width=0.8pt] (8.,3.)-- (7.48,2.32);
\draw [line width=0.8pt] (8.,3.)-- (8.68,2.34);
\draw [line width=0.8pt] (9.9,3.54)-- (10.6,3.88);
\draw [line width=0.8pt] (9.54,2.86)-- (10.34,2.3);
\draw [line width=0.8pt] (9.9,3.54)-- (10.56,3.14);
\draw [line width=0.8pt] (7.02,3.4)-- (8.,3.);
\begin{scriptsize}
\draw [fill=ududff] (6.9,4.6) circle (4.0pt);
\draw [fill=ududff] (8.,3.) circle (4.0pt);
\draw [fill=ududff] (9.04,4.6) circle (4.0pt);
\draw [fill=ududff] (8.,5.78) circle (4.0pt);
\draw [fill=ududff] (9.9,3.54) circle (4.0pt);
\draw [fill=ududff] (9.54,2.86) circle (4.0pt);
\draw [fill=ududff] (6.04,3.28) circle (4.0pt);
\draw [fill=ududff] (6.08,5.84) circle (4.0pt);
\draw [fill=ududff] (5.4,5.48) circle (4.0pt);
\draw [fill=ududff] (5.2,4.86) circle (4.0pt);
\draw [fill=ududff] (10.6,5.44) circle (4.0pt);
\end{scriptsize}
\end{tikzpicture}
    \caption{Every connected unicyclic bipartite graph contains an even cycle $C$ whose deletion (i.e., when we delete its edges) leaves a bipartite forest with $|V(C)|$ many components.}
   \label{f:unicyclic}
\end{figure}
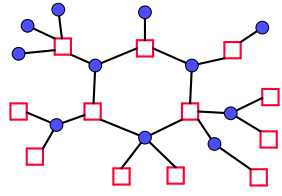

Hence, we can count $C(i,j,0)$ by first choosing a cycle of length $2r$, of which there are $\frac{(i)_r(j)_r}{2r}$ many possibilities, and then choosing from the $F(i,j,r,r)$ many possibilities for the forest left by the deletion of this cycle.
 
Hence, it follows from \eqref{e:counting} that
 \begin{align}\label{e:unicycliccount}
 C(i,j,0) = \sum_{r=2}^{\min\{i,j\}} \frac{(i)_r (j)_r}{2r} F(i,j,r,r)= \frac{1}{2} i^{j-1} j^{i-1} \sum_{r=2}^{\min\{i,j\}} \frac{(i)_r(j)_r}{i^r j^r}\left(i+j-r\right),
 \end{align}
proving the first part of Theorem \ref{t:unicyliccomponents}.

So let us suppose further that $i,j = \omega(1)$ and $\frac{1}{2} \leq \frac{i}{j} \leq 2$. By \eqref{e:fallingfactorialbound} we can conclude that
\[
\frac{(i)_r(j)_r}{i^r j^r}\leq  \exp \left(-\frac{(r-1)^2}{2i} - \frac{(r-1)^2}{2j} \right),
\] 
and furthermore by \eqref{e:fallingfactorial} it follows that if $r = o\left(i^{\frac{2}{3}}\right)$ and $r=  o\left(j^{\frac{2}{3}}\right)$, then
\[
\frac{(i)_r(j)_r}{i^r j^r} \approx \exp \left(-\frac{r^2}{2i} - \frac{r^2}{2j}\right).
\]

We split \eqref{e:unicycliccount} into two parts. Firstly, when $r \leq i^{\frac{5}{9}}$ we note that $r = o\left(i^{\frac{2}{3}}\right)$ and $r= o\left(j^{\frac{2}{3}}\right)$, and hence
\begin{align*}
\sum_{r=2}^{i^{\frac{5}{9}}} \frac{(i)_r(j)_r}{i^r j^r} \left(i+j-r\right)\approx\left(i+j\right) \sum_{r=2}^{i^{\frac{5}{9}}}  \exp \left(-\frac{r^2}{2i}-\frac{r^2}{2j}\right)\approx \left(i+j\right)\sqrt{\frac{\pi ij}{2(i+j)}},
\end{align*}
where the final line follows from a standard estimate that 
\[
\sum_{r=1}^{\infty} e^{-\frac{r^2}{2n}} \approx \int_0^\infty e^{-\frac{x^2}{2n}} dx = \sqrt{\frac{\pi n}{2}}.
\]

Conversely, when $r > i^{\frac{5}{9}}$ we can naively bound
\begin{align*}
\sum_{r=i^{\frac{5}{9}}+1}^{\min \{i,j\}} \frac{(i)_r(j)_r}{i^r j^r}\left(i+j-r\right) \leq\left(i+j\right) i \exp \left(-\frac{(i^{\frac{5}{9}}-1)^2}{2i} - \frac{(i^{\frac{5}{9}}-1)^2}{2j}\right)\leq\left(i+j\right) i \exp \left( -\Omega\left(i^{\frac{1}{9}}\right) \right)= o\left(i+j\right).
 \end{align*}
 
 It follows that
 \begin{align*}
 C(i,j,0) \approx \sqrt{\frac{\pi ij}{8(i+j)}}\left(i+j\right) i^{j-1} j^{i-1}= \sqrt{\frac{\pi}{8}} \sqrt{i+j} i^{j-\frac{1}{2}} j^{i-\frac{1}{2}}.
 \end{align*}
\end{proof}

\subsection{Bipartite graphs with positive excess: proof of Theorem~\ref{t:complexcomponents}}
We use similar counting arguments as Bollob\'{a}s \cite{BollobasBook} in his proof of \eqref{e:complexgeneral}. The case where $\ell > i+j$ will be significantly easier, so let us first assume that $\ell \leq i+j$.

Given a graph $H$, let the {\em core} $C(H)$ of $H$ be the maximal subgraph of $H$ with minimum degree at least two. Furthermore, we call a path in $C(H)$ {\em maximal bare} if all its internal vertices are of degree two and its endvertices have degree at least three. We obtain the {\em kernel} $K(H)$ of $H$ by replacing each maximal bare path in $C(H)$ by an edge, i.e., we delete all edges and internal vertices of the path and add a new edge between the two endpoints. Using the kernel and core of a graph, we can construct all connected bipartite graphs with partition classes $I=\left\{x_1, \ldots, x_i\right\}$ and $J=\left\{y_1, \ldots, y_j\right\}$ and $i+j+\ell$ edges as follows:

\begin{itemize}
    \item[(C1)] Choose the vertex set $V(K)$ of the kernel. We set $t_1:=|I\cap V(K)|$, $t_2:=|J\cap V(K)|$, and $t=t_1+t_2$. As $K$ has minimum degree at least three, we have $t\leq 2\ell$.
    \item[(C2)] Choose for the kernel $K$ a connected (not necessarily bipartite) multigraph on vertex set $V(K)$ having $t+\ell$ edges;
    \item[(C3)] Choose the vertex set $V(C)$ of the core $C$. As $V(K)$ is already fixed, we have to pick only the vertices in $\left(V(C)\setminus V(K)\right)$. We define $u_1$ and $u_2$ such that $|V(C)\cap I|=t_1+u_1$ and $|V(C)\cap J|=t_2+u_2$, respectively. Furthermore, let $u=u_1+u_2$.
    \item[(C4)] Subdivide the edges of $K$ by the vertices from $\left(V(C)\setminus V(K)\right)$ such that the resulting core has no loops or multiple edges and is bipartite with respect to the vertex bipartition $(I,J)$. (The requirement of being bipartite can only be fulfilled for appropriate choices of the set $V(C)\setminus V(K)$ in step (C3));
    \item[(C5)]
    Choose a forest $F$ on vertex set $I\cup J$ having $u_1+u_2+t_1+t_2$ tree components such that all vertices from $V(C)$ lie in different components and that there is no edge between a vertex in $I$ and a vertex in $J$ ($F$ is a \lq rooted bipartite forest\rq). We obtain the graph by replacing each vertex $v$ in $C$ by the tree of $F$ rooted at $v$.
\end{itemize}
To show the claimed bound on $C(i,j,\ell)$, we estimate the number of choices we have in each construction step. For fixed $t_1$ and $t_2$ the number of options for $V(K)$ in step (C1) is
\begin{align}\label{eq:9}
\binom{i}{t_1}\binom{j}{t_2}.
\end{align}

The kernel $K$ in step (C2) is determined by choosing for each pair $v, w\in V(K)$ the number of edges between $v$ and $w$ and for each $x\in V(K)$ the number of loops at $x$. Hence, we have to \lq partition\rq\ the $t+\ell$ many edges into $\binom{t}{2}+t$ many (possible empty) groups, each of them corresponding to a pair of vertices or a single vertex. Hence, the number of choices in step (C2) is at most
\begin{align}\label{eq:7}
  \binom{t+\ell+\binom{t}{2}+t-1}{\binom{t}{2}+t-1}=\binom{t+\ell+\binom{t}{2}+t-1}{t+\ell}\leq \left(\frac{e\left(t+\ell+\binom{t}{2}+t-1\right)}{t+\ell}\right)^{t+\ell} =O(1)^\ell \ell^{t+\ell},
\end{align}
where we used in the last inequality that $t\leq 2\ell$.

In step (C3) we have for fixed $u_1$ and $u_2$ 
\begin{align}\label{eq:8}
\binom{i-t_1}{u_1}\binom{j-t_2}{u_2}
\end{align}
choices for the set $V(C)\setminus V(K)$. When constructing $C$ from $K$, the number of vertices in $I$ that are placed on some fixed edge differs at most by one from the number of vertices in $J$ that are placed on the same edge. As there are at most $3\ell$ edges in $K$, we obtain $\left|u_1-u_2\right|\leq 3\ell$.

To bound the possible choices in (C4), we fix an ordering $v_1w_1, \ldots, v_{t+\ell}w_{t+\ell}$ of the edges in $K$ and an orientation for each of these edges (say from $v_s$ to $w_s$). We can construct each possible core $C$ by choosing permutations $\alpha_1, \ldots, \alpha_{u_1}$ and $\beta_1, \ldots, \beta_{u_2}$ of the vertices in $\left(V(C)\setminus V(K)\right)\cap I$ and $\left(V(C)\setminus V(K)\right)\cap J$, respectively, and non-negative integers $r_1, \ldots, r_{t+\ell}$ with $r_1+\ldots+r_{t+\ell}=u$. Then we consider the edges $v_1w_1, \ldots, v_{t+\ell}w_{t+\ell}$ one after another and subdivide each edge $v_sw_s$ with $r_s$ many vertices. If $v_s\in I$, we start with the vertex in $\beta_1, \ldots, \beta_{u_2}$ with smallest index which has not been used for a previous edge and then place alternatingly the first unused vertex from $\alpha_1, \ldots, \alpha_{u_1}$ and $\beta_1, \ldots, \beta_{u_2}$ on the edge $v_sw_s$. We proceed similarly in the case $v_s\in J$. Using this construction, we obtain the following upper bound for the number of different ways of performing (C4)
\begin{align}\label{eq:12}
    (u_1)!(u_2)!\cdot \binom{u+t+\ell-1}{t+\ell-1}.
\end{align}
We note that only certain choices of $r_1, \ldots, r_{t+\ell}$ lead to bipartite graphs.

Due to Lemma \ref{l:bipartiteforest} the number of choices for $F$ in step (C5) is
\begin{align}\label{eq:10}
    i^{j-t_2-u_2-1}j^{i-t_1-u_1-1}\left((t_1+u_1)j+(t_2+u_2)i-(t_1+u_1)(t_2+u_2)\right) \leq  i^{j-t_2-u_2}j^{i-t_1-u_1}\left(\frac{t_1+u_1}{i}+\frac{t_2+u_2}{j}\right).
\end{align}

Combining \eqref{eq:9}, \eqref{eq:7}, \eqref{eq:8}, \eqref{eq:12}, and \eqref{eq:10} we get

\begin{align}\label{eq:13}
    C(i,j,\ell)\leq \sum_{t_1,t_2,u_1,u_2}\binom{i}{t_1}\binom{j}{t_2}O(1)^\ell \ell^{t+\ell}\binom{i-t_1}{u_1}\binom{j-t_2}{u_2} (u_1)!(u_2)! \binom{u+t+\ell-1}{t+\ell-1}i^{j-t_2-u_2}j^{i-t_1-u_1}\left(\frac{t_1+u_1}{i}+\frac{t_2+u_2}{j}\right),
\end{align}
where the sum is over all non-negative integers $t_1, t_2, u_1, u_2$ satisfying $t_1+t_2\leq 2\ell$, $u_1\leq i-t_1$, $u_2\leq j-t_2$, and $\left|u_1-u_2\right|\leq 3\ell$.

Due to \eqref{e:fallingfactorialbound} we get
\begin{align*}
    \binom{i}{t_1}\binom{i-t_1}{u_1}(u_1)!=\frac{(i)_{t_1+u_1}}{(t_1)!}\leq
    \frac{i^{t_1+u_1}}{(t_1)!} \exp\left(-\frac{(t_1+u_1)^2}{2i}\right)O(1),
\end{align*}
where we used in the last inequality that $t_1+u_1\leq i$. Similarly, we have
\begin{align*}
 \binom{j}{t_2}\binom{j-t_2}{u_2}(u_2)!\leq \frac{j^{t_2+u_2}}{(t_2)!} \exp\left(-\frac{(t_2+u_2)^2}{2j}\right)O(1).
\end{align*}
Hence, we get in \eqref{eq:13}
\begin{align}\label{eq:14}
C(i,j,\ell)\leq i^jj^iO(1)^\ell \ell^\ell \sum_{t_1,t_2,u_1,u_2}& \frac{\ell^t}{(t_1)!(t_2)!}\left(\frac{i}{j}\right)^{t_1+u_1-t_2-u_2}\binom{u+t+\ell-1}{t+\ell-1} \nonumber
\\
&~~~\times\exp\left(-\frac{(t_1+u_1)^2}{2i}-\frac{(t_2+u_2)^2}{2j}\right) \left(\frac{t_1+u_1}{i}+\frac{t_2+u_2}{j}\right).
\end{align}
Using $\frac{1}{2}\leq\frac{i}{j}\leq 2$, $|u_1-u_2|\leq 3\ell$, and $t_1, t_2\leq 2\ell$, we have
\begin{align*}
  \left(\frac{i}{j}\right)^{t_1+u_1-t_2-u_2}=O(1)^\ell. 
\end{align*}
Thus, we obtain in \eqref{eq:14}
\begin{align}\label{eq:15}
C(i,j,\ell)\leq i^jj^iO(1)^\ell \ell^\ell \sum_{t_1,t_2,u_1,u_2}S(t_1,t_2,u_1,u_2),
\end{align}
where 
\begin{align*}
 S(t_1,t_2,u_1,u_2):=\frac{\ell^t}{(t_1)!(t_2)!}\exp\left(-\frac{(t_1+u_1)^2}{2i}-\frac{(t_2+u_2)^2}{2j}\right)\binom{u+t+\ell-1}{t+\ell-1}\left(\frac{t_1+u_1}{i}+\frac{t_2+u_2}{j}\right).   
\end{align*}
To analyse \eqref{eq:15}, we distinguish two cases depending how large $u$ is. First, we consider the case $u\leq 4\ell$. Then we have
\begin{align}\label{eq:16}
 \binom{u+t+\ell-1}{t+\ell-1}\leq 2^{u+t+\ell-1}\leq 2^{4\ell+2\ell+\ell-1}=O(1)^\ell.   
\end{align}
Similarly, we have
\begin{align}\label{eq:20}
    \left(\frac{t_1+u_1}{i}+\frac{t_2+u_2}{j}\right)=\frac{O(\ell)}{\sqrt{ij}}.
\end{align}
Furthermore, using Lemma \ref{l:spencer} we get
\begin{align}\label{eq:17}
    \sum_{u_1}\exp\left(-\frac{(t_1+u_1)^2}{2i}\right)\leq \sum_{k=1}^\infty \exp\left(-\frac{k^2}{2i}\right)\leq \int_0^\infty \exp\left(-\frac{x^2}{2i}\right)dx+O(1)=O\left(\sqrt{i}\right).
\end{align}
Analogously, we obtain
\begin{align}\label{eq:18}
    \sum_{u_2}\exp\left(-\frac{(t_2+u_2)^2}{2j}\right)=O\left(\sqrt{j}\right).
\end{align}
Combining \eqref{eq:16}, \eqref{eq:20}, \eqref{eq:17}, and \eqref{eq:18} we get
\begin{align}\label{eq:19}
    \sum_{t_1,t_2}~\sum_{u_1+u_2\leq 4\ell}S(t_1,t_2,u_1,u_2)\leq O(1)^\ell  \sum_{t_1,t_2} \frac{\ell^t}{(t_1)!(t_2)!}.
\end{align}
Furthermore, we have
\begin{align*}
    \sum_{t_1,t_2} \frac{\ell^t}{(t_1)!(t_2)!}=\sum_{t_1} \frac{\ell^{t_1}}{(t_1)!}\sum_{t_2} \frac{\ell^{t_2}}{(t_2)!}\leq \exp(2\ell)=O(1)^\ell.
\end{align*}
Hence, we get
\begin{align}\label{eq:21}
    \sum_{t_1,t_2}~\sum_{u_1+u_2\leq 4\ell}S(t_1,t_2,u_1,u_2)\leq O(1)^\ell.
\end{align}

Now we consider the case $u>4\ell$. Due to $|u_1-u_2|\leq 3\ell$, we have $u_1, u_2>\ell/2$. Combining that with $t_1, t_2\leq 2\ell$ and $\frac{1}{2}\leq\frac{i}{j}\leq 2$, we obtain
\begin{align}\label{eq:22}
    \left(\frac{t_1+u_1}{i}+\frac{t_2+u_2}{j}\right)=\Theta(1)\frac{u}{i+j}.
\end{align}
Furthermore, we have
\begin{align}\label{eq:23}
  \binom{u+t+\ell-1}{t+\ell-1}\leq \frac{\left(u+t+\ell-1\right)^{t+\ell-1}}{(t+\ell-1)!}=\frac{O(1)^\ell u^{t+\ell-1}}{(t+\ell-1)!}.
\end{align}
Using the inequality $\frac{x^2}{r}+\frac{y^2}{s}\geq \frac{(x+y)^2}{r+s}$ for $r,s,x,y\in\mathbb{R}_{>0}$ we have
\begin{align}\label{eq:24}
    \exp\left(-\frac{(t_1+u_1)^2}{2i}-\frac{(t_2+u_2)^2}{2j}\right)\leq \exp\left(-\frac{\left(t+u\right)^2}{2(i+j)}\right).
\end{align}
Combining \eqref{eq:22}, \eqref{eq:23}, and \eqref{eq:24} we have
\begin{align}\label{eq:25}
\sum_{t_1,t_2}~\sum_{u_1+u_2>4\ell}S(t_1,t_2,u_1,u_2)
&\leq \sum_{t_1,t_2}~\sum_{u_1+u_2>4\ell}\frac{\ell^t}{(t_1)!(t_2)!} \exp\left(-\frac{\left(t+u\right)^2}{2(i+j)}\right)\frac{O(1)^\ell u^{t+\ell-1}}{(t+\ell-1)!}\frac{u}{i+j}\nonumber
\\
&=\sum_{t_1,t_2}\frac{O(1)^\ell \ell^t}{(t+\ell-1)!(i+j)(t_1)!(t_2)!}
\sum_{u_1+u_2>4\ell}\exp\left(-\frac{\left(t+u\right)^2}{2(i+j)}\right)u^{t+\ell}.
\end{align}
Due to $|u_1-u_2|\leq 3\ell$, there are for each fixed value of $u$ at most $3\ell+1$ different pairs $(u_1, u_2)$ with $u_1+u_2=u$. Hence, we have
\begin{align*}
    \sum_{u_1+u_2>4\ell}\exp\left(-\frac{\left(t+u\right)^2}{2(i+j)}\right)u^{t+\ell}&=O(1)^\ell\sum_{u=4\ell+1}^{i+j}\exp\left(-\frac{\left(t+u\right)^2}{2(i+j)}\right)u^{t+\ell} \leq O(1)^\ell\int_{0}^\infty x^{t+\ell}\exp\left(-\frac{x^2}{2(i+j)}\right)dx
    \\
    &= O(1)^\ell\left(i+j\right)^{\frac{(t+\ell+1)}{2}}\int_{0}^\infty y^{t+\ell}\exp\left(-\frac{y^2}{2}\right)dy \leq O(1)^\ell\left(i+j\right)^{\frac{t+\ell+1}{2}} \sqrt{(t+\ell)!},
\end{align*}
where we used in the last inequality that $\int_{0}^\infty y^{n}\exp\left(\frac{-y^2}{2}\right)dy\leq \sqrt{n!}$, which can be shown by repeatedly applying integration by parts (see e.g., \cite[Exercise 9 of Chapter 1]{BollobasBook}). Plugging this into \eqref{eq:25} we obtain
\begin{align}\label{eq:26}
 \sum_{t_1,t_2}~\sum_{u_1+u_2>4\ell}S(t_1,t_2,u_1,u_2)
&\leq \sum_{t_1,t_2}\frac{O(1)^\ell \ell^t(i+j)^{\frac{t+\ell-1}{2}}\sqrt{(t+\ell)!}}{(t+\ell-1)!(t_1)!(t_2)!}\nonumber
\\
&=O(1)^\ell (i+j)^{\frac{\ell-1}{2}}\sum_{t_1,t_2}\frac{\ell^t(i+j)^{\frac{t}{2}}}{\sqrt{(t+\ell)!}(t_1)!(t_2)!}.
\end{align}
Using $(t+\ell)!\geq \left(\frac{t+\ell}{e}\right)^{t+\ell}=O(1)^\ell \ell^{t+\ell}$, we obtain in \eqref{eq:26}
\begin{align}\label{eq:27}
 \sum_{t_1,t_2}~\sum_{u_1+u_2>4\ell}S(t_1,t_2,u_1,u_2)\leq O(1)^\ell(i+j)^{\frac{\ell-1}{2}}\ell^{-\frac{\ell}{2}}\sum_{t_1,t_2}\frac{\ell^{\frac{t}{2}}(i+j)^{\frac{t}{2}}}{(t_1)!(t_2)!}.
\end{align}
Furthermore, we have
\begin{align}\label{eq:28}
    \frac{1}{(t_1)!(t_2)!}\leq\frac{1}{t!}\binom{t}{\left\lfloor \frac{t}{2}\right\rfloor}\leq \frac{1}{t!}\left(\frac{et}{\left\lfloor \frac{t}{2}\right\rfloor}\right)^{\left\lfloor \frac{t}{2}\right\rfloor}=\frac{O(1)^{\ell}}{t!}.
\end{align}
As $t_1,t_2\leq 2\ell$, there are for every fixed value of $t$ at most $2\ell+1$ many pairs $(t_1,t_2)$ with $t_1+t_2=t$. Using that and \eqref{eq:28} in \eqref{eq:27} we get
\begin{align}\label{eq:29}
    \sum_{t_1,t_2}~\sum_{u_1+u_2>4\ell}S(t_1,t_2,u_1,u_2)\leq O(1)^\ell(i+j)^{\frac{\ell-1}{2}}\ell^{-\frac{\ell}{2}}\sum_{t=1}^{2\ell}\frac{\ell^{\frac{t}{2}}(i+j)^{\frac{t}{2}}}{t!}.
\end{align}
For consecutive terms in the sum in \eqref{eq:29} we have
\begin{align*}
    \frac{\ell^{\frac{t+1}{2}}(i+j)^{\frac{t+1}{2}}}{(t+1)!} :\frac{\ell^{\frac{t}{2}}(i+j)^{\frac{t}{2}}}{t!}&=\frac{\sqrt{\ell(i+j)}}{t+1}
    \geq \frac{\sqrt{\ell(i+j)}}{3\ell}\geq \frac{1}{3},
\end{align*}
where we used in the last inequality that $\ell\leq i+j$. Hence, we have for all $t\leq 2\ell$
\begin{align*}
    \frac{\ell^{\frac{t}{2}}(i+j)^{\frac{t}{2}}}{t!}\leq 3^{2\ell-t}\frac{\ell^\ell (i+j)^\ell}{(2\ell)!}\leq O(1)^\ell \frac{\ell^\ell(i+j)^\ell}{\left(\frac{2\ell}{e}\right)^{2\ell}}=O(1)^\ell\frac{(i+j)^\ell}{\ell^\ell}.
\end{align*}
This yields in \eqref{eq:29}
\begin{align}\label{eq:30}
    \sum_{t_1,t_2}~\sum_{u_1+u_2>4\ell}S(t_1,t_2,u_1,u_2)\leq O(1)^\ell(i+j)^{\frac{3\ell-1}{2}}\ell^{-\frac{3\ell}{2}}.
\end{align}
This concludes the case $u>4\ell$.
Combining \eqref{eq:15}, \eqref{eq:21}, and \eqref{eq:30} we obtain
\begin{align*}
C(i,j,\ell)&\leq i^jj^iO(1)^\ell \ell^\ell \left(O(1)^\ell+O(1)^\ell(i+j)^{\frac{3\ell-1}{2}}\ell^{-\frac{3\ell}{2}}\right)
\\
&=i^jj^iO(1)^\ell \ell^\ell(i+j)^{\frac{3\ell-1}{2}}\ell^{-\frac{3\ell}{2}}
\\
&\leq i^{j}j^{i}(i+j)^{\frac{3\ell-1}{2}}\left(\frac{c}{\ell}\right)^{\frac{\ell}{2}},
\end{align*}
for a suitable $c>0$.

In the case that $\ell > i+j:=k$ we can argue more directly. Indeed, we can naively bound $C(i,j,\ell)$ by looking at the total number of bipartite graphs, connected or disconnected, with partitions classes of size $i$ and $j$ and $k+\ell$ edges, which is clearly $\binom{ij}{k+\ell}$.

Using the elementary bound $\binom{n}{r}\leq \left(\frac{en}{r}\right)^{r}$ for all $r \leq n$, it follows that
\begin{align}
C(i,j,\ell)&\leq \binom{ij}{k+\ell}\leq\left(\frac{eij}{k+\ell}\right)^{k+\ell}=\sqrt{k}e^{k+\ell}i^jj^ik^{\frac{3\ell-1}{2}}\ell^{-\frac{\ell}{2}}\left(\frac{i}{k}\right)^{i+\ell}\left(\frac{j}{k}\right)^{j+\ell}\left(\frac{k}{k+\ell}\right)^{k+\frac{\ell}{2}}\left(\frac{\ell}{k+\ell}\right)^{\frac{\ell}{2}} \nonumber
\\
&=O(1)^\ell i^jj^ik^{\frac{3\ell-1}{2}}\ell^{-\frac{\ell}{2}} \label{e:largeell},
\end{align}
where we used in the last equality that $\sqrt{k}e^{k+\ell}=O(1)^\ell$ since $\ell>k$. We note that with a more careful calculation it can be shown that the $O(1)^{\ell}$ term in \eqref{e:largeell} is in fact at most one.

\section{Discussion} \label{s:Discuss}

We have presented some initial results about the structure of $G(n,n,p)$ in the weakly supercritical regime, however many interesting questions still remain. For example, {\L}uczak \cite{Luczakcycle} described in more detail the distribution of cycles in $G(n,p)$ in this regime. In particular, if we let the \emph{girth} of a graph be the length of the shortest cycle and the \emph{circumference} be the length of the longest cycle, then {\L}uczak determined asymptotically the girth and circumference of the giant component of $G(n,p)$ and the length of the longest cycle outside of the giant component.
 
 \begin{question}
 Let $\epsilon=\epsilon(n) >0$ be such that $\epsilon^3n \to \infty$ and $\epsilon = o(1)$, and let $p = \frac{1+\epsilon}{n}$. What is the girth and circumference of the giant component in $G(n,n,p)$? What is the length of the longest cycle outside of the giant component?
 \end{question}
 
 Using some of the results of {\L}uczak \cite{Luczakcycle} on the distribution of cycles in the weakly supercritical regime in $G(n,p)$, together with Euler's formula, Dowden, Kang and Krivelevich \cite{Kang} were able to determine asymptotically the genus of $G(n,p)$ in this regime, in particular showing that whp the genus is asymptotically given by half of the excess of the giant component. It is natural to ask if a similar statement holds in the bipartite model.
 \begin{question}
 Let $\epsilon=\epsilon(n) >0$ be such that $\epsilon^3n \to \infty$ and $\epsilon = o(1)$, and let $p = \frac{1+\epsilon}{n}$. Is it true that whp the genus $g$ of $G(n,n,p)$ is such that
 \[
 g \approx \frac{1}{2}\text{excess}\left(L_1\left(G(n,n,p)\right)\right) \approx \frac{2}{3}\epsilon^3 n?
 \]
 \end{question}
 
Theorems \ref{t:trees}-\ref{t:complex} suggest an interesting relationship between the component structure of $G\left(n,n,\frac{1+\epsilon}{n}\right)$ and that of $G\left(n,n,\frac{1-\epsilon}{n}\right)$ in the weakly super- and subcritical regimes. In the case of the binomial random graph model, a much more precise relationship can be given. Given a graph $G$, let us write $G^L$ for the graph obtained by deleting a component of $G$ of maximum order, say $L$. Roughly speaking, it is known that $G^L\left(n,\frac{1+\epsilon}{n}\right)$ and $G\left(n-|L|,\frac{1-\epsilon}{n-|L|}\right)$ have approximately the same distribution. For a more detailed discussion of this phenomenon, known as the \emph{symmetry rule}, see for example \cite[Section 5.6]{Janson}. Using similar techniques as in \cite{Luczak}, which uses bounds on the excess of the giant component to prove a symmetry rule, we expect that Theorem \ref{t:excess} can be used to show a similar statement in the bipartite binomial random graph model.
\begin{conjecture}
Let $\epsilon =\epsilon(n) >0$ be such that $\epsilon^3 n \to \infty$ and $\epsilon  = o(1)$, and let $p= \frac{1+\epsilon}{n}$. If we let \[
n^{\pm} = (1 - 2 \epsilon \pm o(\epsilon))n \qquad \text{and} \qquad p^{\pm} = \frac{1 - \epsilon \pm o(\epsilon)}{n^{\pm}},
\]
then we can couple $G^L(n,n,p)$ with $G(n^-,n^-,p^-)$ and $G(n^+,n^+,p^+)$ such that whp 
\[
G(n^-,n^-,p^-) \subseteq G^L(n,n,p) \subseteq G(n^+,n^+,p^+).
\]
\end{conjecture}

\section*{Acknowledgements}
The authors would like to thank the reviewers for their helpful suggestions and comments.

 \bibliographystyle{plain}
\bibliography{ComponentBehaviourReferences}

\begin{appendix}
	
\section{Proof of Lemma~\ref{l:difference}}\label{s: Appendix A} 
Let us write
$$g(y):=\frac{1}{(k^2-y^2)^m}\left(\frac{k-y}{k+y}\right)^{y}\exp\left(-\frac{y^2}{2n}\right).$$
If we let $h(y):=\log \left(g(y)\right)=-m\log(k^2-y^2)+y\log\left(\frac{k-y}{k+y}\right)-\frac{y^2}{2n}$, then
$$h'(y)=\frac{2my}{k^2-y^2}+\log\left(\frac{k-y}{k+y}\right)-\frac{2ky}{k^2-y^2}-\frac{y}{n},$$
$$h''(y)=\frac{2mk^2-4k^3+2my^2}{(k^2-y^2)^2}-\frac{1}{n},$$
$$h'''(y)=\frac{2y(6mk^4-4mk^2y^2-2my^4-8k^5+8k^3y^2)}{(k^2-y^2)^4}.$$
Note that $0$ is a solution of $h'(y)=0$ and, since $m$ is fixed and $h''(y)<0$ on $[-L-1,L+1]$, $0$ is the unique solution on $[-L-1,L+1]$. Hence, $h(y)$ is increasing on $[-L-1,0]$ and deceasing on $[0,L+1]$, and this is also true for $g(y)$.

Therefore, by Lemma \ref{l:spencer} we can bound the difference between  
$$I:=\int_{-L}^{L}g(y) dy,$$
and $S$ as $|S-I| \leq 12g(0)$.
We will later show that $I = \omega\left(g(0)\right)$, and hence $S \approx I$.

In order to estimate $I$, we approximate $g$ by a Gaussian function. By the mean value form of the remainder in Taylor's theorem, for any $y\in [-L,L]$ there is a real number $z$ between $0$ and $y$ such that
$$h(y)=h(0)+\frac{h''(0)}{2}y^2+\frac{h'''(z)}{6}y^3.$$
Note that, if $|z| = o(k)$, then $|h'''(z)|=o\left(\frac{1}{k^2}\right)$. Therefore, for any $|y|\leq k^\frac{3}{5}$ we have 
$$h(y)=h(0)+\frac{h''(0)}{2}y^2+o\left(\frac{y^3}{k^2}\right)=h(0)+\frac{h''(0)}{2}y^2+o(1).$$

Hence, if we let $R = \min \{ k^{\frac{3}{5}}, L\}$ then
\[
I= \int_{-R}^{R} \exp \left(h(0) + \frac{h''(0)}{2} y^2 + o(1)\right) dy + \int_{L \geq |y|\geq R} e^{h(y)} dy.
\]
The first integral we can evaluate in a standard manner as
\begin{align*}
    \int_{-R}^{R} \exp \left(h(0) + \frac{h''(0)}{2} y^2 + o(1)\right) dy &\approx \int_{-R}^{R} \exp \left(h(0) + \frac{h''(0)}{2}y^2\right) dy\approx e^{h(0)}\int_{-\infty}^{\infty} \exp \left(\frac{h''(0)}{2} y^2\right)dy \\&\approx \sqrt{\frac{2 \pi}{|h''(0)|}}e^{h(0)} = \sqrt{\frac{\pi}{2}}k^{\frac{1}{2}-2m},
\end{align*}
where we used that $h''(0)=\frac{2m}{k^2} - \frac{4}{k}-\frac{1}{n}\approx -\frac{4}{k}$ and also that $R = \omega(1)$.

If $R = L$, then the second integral is $0$, and so we may assume that $R = k^{\frac{3}{5}}$. In order to bound the second integral we note that all the terms in $h(y)$ are negative, and in particular if $|y| \leq L \leq k$
\[
\log\left(\frac{k-y}{k+y}\right) = \log\left( 1 - \frac{2y}{k+y}\right) \leq -\frac{2y}{k+y}.
\]
Hence, if $L \geq |y| \geq R$, then
\[
h(y) \leq y \log\left(\frac{k-y}{k+y}\right) \leq -\frac{y^2}{k+y} \leq -\frac{k^{\frac{1}{5}}}{2}.
\]

It follows that
\[
\int_{L \geq |y|\geq R}  e^{h(y)} dy \leq 2 \int_{k^{\frac{3}{5}}}^{\infty} \exp \left(-\frac{y^{\frac{1}{5}}}{2}\right) dy =O\left(e^{-\frac{1}{2}k^\frac{3}{25}}k^\frac{12}{25}\right)= o\left(k^{\frac{1}{2}-2m}\right).
\]

Hence, $I \approx \sqrt{\frac{\pi}{2}}k^{\frac{1}{2}-2m}$ and, noting that $g(0) = k^{-2m}= o(I)$, the result follows.
\qed
\section{Proof of Lemma~\ref{l:Var}}\label{s:Append B} 	
Recall that we write $X(i,j,-1)$ for the number of tree components with $i$ vertices in $N_1$ and $j$ vertices in $N_2$, and let 
\[
\Lambda_k = \left\{ (i,j) \in \mathbb{N}^2 \colon i+j = k \text{ and } |i-j| < \epsilon^{\frac{1}{4}}  \sqrt{n}\right\}.
\]

Then,
     \[
    Z_a = \sum_{k=1}^{\Tilde{k}} k^a \sum_{(i,j)\in \Lambda_k} X(i,j,-1),
    \]
and so
     \begin{align*}
         \mathbb{E}\left(Z_1^2\right) &= \sum_{k_1=1}^{\Tilde{k}} \sum_{k_2=1}^{\Tilde{k}} k_1 k_2 \sum_{(i,j)\in \Lambda_{k_1}} \sum_{(s,t) \in \Lambda_{k_2}} \mathbb{E}\left( X(i,j,-1)X(s,t,-1)\right).
     \end{align*}
Let us write $\mu_{i,j} = \mathbb{E}\left(X(i,j,-1)\right)$. Then, when $(i,j) \neq (s,t)$ we have, by comparison with \eqref{e:treescount},
\begin{align*}
\mathbb{E}\left(X(i,j,-1)X(s,t,-1)\right) &= \binom{n}{i}\binom{n}{j}\binom{n-i}{s}\binom{n-j}{t} C(i,j,-1)C(s,t,-1) p^{k_1 +k_2-2}(1-p)^{n(k_1+k_2) - ij - st - sj - ti - k_1 - k_2 +2}\\
&= \mu_{i,j}\mu_{s,t}\frac{(n)_{i+s}}{(n)_i(n)_s}\frac{(n)_{j+t}}{(n)_j(n)_t}(1-p)^{-it-sj},
\end{align*}
and when $(i,j) = (s,t)$ we have
\[
\mathbb{E}\left( X(i,j,-1)^2\right) = \mu_{i,j} + \mu_{i,j}^2\frac{(n)_{2i}}{(n)^2_i}\frac{(n)_{2j}}{(n)^2_j}(1-p)^{-2ij}.
\]

Now, it can be seen that if $0\leq x\leq y \leq 1$, then
\[
1-y \leq (1-x)e^{x-y}, 
\]
and so
\begin{align}\label{e:fallingquotient}
\frac{(n)_{i+s}}{(n)_i(n)_s} = \prod_{m=0}^{i-1} \frac{1-\frac{s+m}{n}}{1-\frac{m}{n}} \leq \text{exp}\left(\sum_{m=0}^{i-1} \frac{m}{n} - \frac{s+m}{n} \right) = \exp \left(-\frac{is}{n}\right),
\end{align}
and a similar bound holds for $\frac{(n)_{j+t}}{(n)_j(n)_t}$. Hence, using \eqref{e:fallingquotient} and the fact that $(1-p)^x \leq e^{-px}$ for any positive $p$ and $x$, we have
\begin{align}
\mathbb{E}\left(Z_1^2\right) &= \sum_{k_1=1}^{\Tilde{k}} \sum_{k_2=1}^{\Tilde{k}}k_1 k_2  \sum_{(i,j)\in \Lambda_{k_1}}\sum_{(s,t) \in \Lambda_{k_2}} \mathbb{E}\left(X(i,j,-1)X(s,t,-1)\right) \nonumber\\
&= \mathbb{E}(Z_{2}) + \sum_{k_1=1}^{\Tilde{k}} \sum_{k_2=1}^{\Tilde{k}} k_1 k_2  \sum_{(i,j)\in \Lambda_{k_1}} \sum_{(s,t) \in \Lambda_{k_2}} \mu_{i,j}\mu_{s,t}\frac{(n)_{i+s}}{(n)_i(n)_s}\frac{(n)_{j+t}}{(n)_j(n)_t}(1-p)^{-it-sj}\nonumber\\
&\leq  \mathbb{E}(Z_{2}) + \sum_{k_1=1}^{\Tilde{k}} \sum_{k_2=1}^{\Tilde{k}} k_1 k_2 \sum_{(i,j)\in \Lambda_{k_1}} \sum_{(s,t) \in \Lambda_{k_2}} \mu_{i,j}\mu_{s,t} \exp\left(-\frac{is}{n}-\frac{jt}{n}+(it+sj)\left(\frac{1+\epsilon}{n}\right)\right)\nonumber\\
&= \mathbb{E}(Z_{2}) + \sum_{k_1=1}^{\Tilde{k}} \sum_{k_2=1}^{\Tilde{k}} k_1 k_2 \sum_{(i,j)\in \Lambda_{k_1}} \sum_{(s,t) \in \Lambda_{k_2}} \mu_{i,j}\mu_{s,t} \exp\left(\frac{(i-j)(t-s)}{n}+(it+sj)\frac{\epsilon}{n}\right)\nonumber\\
&\leq \mathbb{E}(Z_{2}) + \sum_{k_1=1}^{\Tilde{k}} \sum_{k_2=1}^{\Tilde{k}} k_1 k_2 \sum_{(i,j)\in \Lambda_{k_1}} \sum_{(s,t) \in \Lambda_{k_2}} \mu_{i,j}\mu_{s,t}\exp\left(\frac{(i-j)(t-s)}{n}+\frac{2\epsilon k_1 k_2}{n}\right). \label{e:Z2calc}
\end{align}

Now, since we are only looking at $\epsilon$-uniform components, if $(i,j) \in \Lambda_{k_1}$ and  $(s,t) \in \Lambda_{k_2}$, then 
  \[
  \frac{(i-j)(t-s)}{n} \leq \sqrt{\epsilon} = o(1).
 \] 
Hence, since $0 \leq \frac{2\epsilon k_1 k_2}{n} \leq \frac{2}{3}$, $e^x \leq 1+x+x^2$ for $|x|\leq 1$ and $e^x \leq 1+2x$ for $0 \leq x \leq 1$ it follows that
 \begin{align}
\exp\left(\frac{(i-j)(t-s)}{n}+\frac{2\epsilon k_1 k_2}{n}\right) &\leq \left( 1 + \frac{4\epsilon k_1 k_2}{n}\right) \left(1 + \frac{(i-j)(t-s)}{n} + \frac{(i-j)^2(t-s)^2}{n^2}\right) \nonumber \\
&\leq 1 +\frac{4\epsilon k_1 k_2}{n} + \frac{(i-j)(t-s)}{n}\left(1 +\frac{4\epsilon k_1 k_2}{n}\right) + \frac{(i-j)^2(t-s)^2}{n^2}\left(1 +\frac{4\epsilon k_1 k_2}{n}\right) \nonumber\\
&\leq 1 +\frac{4\epsilon k_1 k_2}{n} + \frac{(i-j)(t-s)}{n}\left(1 +\frac{4\epsilon k_1 k_2}{n}\right) + \frac{3(i-j)^2(t-s)^2}{n^2}. \label{e:exponentbound}
 \end{align}
 
So, from \eqref{e:Z2calc} and \eqref{e:exponentbound} we can conclude that
\begin{align}
\mathbb{E}\left(Z^2_1\right)&\leq \mathbb{E}(Z_{2}) + \sum_{k_1=1}^{\Tilde{k}} \sum_{k_2=1}^{\Tilde{k}} k_1 k_2 \sum_{(i,j)\in \Lambda_{k_1}} \sum_{(s,t) \in \Lambda_{k_2}} \mu_{i,j}\mu_{s,t} \nonumber\\
&\hspace{2cm}\cdot\left(1 +\frac{4\epsilon k_1 k_2}{n} + \frac{(i-j)(t-s)}{n}\left(1 +\frac{4\epsilon k_1 k_2}{n}\right) + \frac{3(i-j)^2(t-s)^2}{n^2}\right). \label{e:fourterms}
\end{align}
We split the sum in \eqref{e:fourterms} into four terms and consider them separately. The first three terms are relatively easy to bound.

Firstly, we have that
\begin{equation}\label{e:firstterm}
\sum_{k_1=1}^{\Tilde{k}} \sum_{k_2=1}^{\Tilde{k}} k_1 k_2 \sum_{(i,j)\in \Lambda_{k_1}} \sum_{(s,t) \in \Lambda_{k_2}} \mu_{i,j}\mu_{s,t} = \mathbb{E}(Z_1)^2.
\end{equation}
The second term can be seen to be
\begin{equation}\label{e:secondterm}
\frac{4\epsilon}{n} \sum_{k_1=1}^{\Tilde{k}} \sum_{k_2=1}^{\Tilde{k}} k_1^{2} k_2^{2} \sum_{(i,j)\in \Lambda_{k_1}} \sum_{(s,t) \in \Lambda_{k_2}} \mu_{i,j}\mu_{s,t} = \frac{4\epsilon}{n} \mathbb{E}(Z_{2})^2.
\end{equation}
Thirdly, since $\mu_{i,j}$ is symmetric in $i$ and $j$ and $\mu_{s,t}$ is symmetric in $s$ and $t$ and $(i-j)$ and $(s-t)$ are antisymmetric, it follows that
\begin{equation}\label{e:thirdterm}
\sum_{k_1=1}^{\Tilde{k}} \sum_{k_2=1}^{\Tilde{k}} k_1 k_2\left(1 +\frac{4\epsilon k_1 k_2}{n}\right) \sum_{(i,j)\in \Lambda_{k_1}} \sum_{(s,t) \in \Lambda_{k_2}} \mu_{i,j}\mu_{s,t}\frac{(i-j)(t-s)}{n} = 0.
\end{equation}

For the fourth term, we have to be a bit more careful. Let us consider 
\begin{align*}
&\sum_{k_1=1}^{\Tilde{k}} \sum_{k_2=1}^{\Tilde{k}} k_1 k_2 \sum_{(i,j)\in \Lambda_{k_1}} \sum_{(s,t) \in \Lambda_{k_2}} \mu_{i,j}\mu_{s,t}\frac{(i-j)^2(t-s)^2}{n^2}\\
 = &  \left(\sum_{k_1=1}^{\Tilde{k}} k_1\sum_{(i,j)\in \Lambda_{k_1}} \mu_{i,j}\frac{(i-j)^2}{n}\right) \left(\sum_{k_2=1}^{\Tilde{k}} k_2 \sum_{(s,t) \in \Lambda_{k_2}} \mu_{s,t}\frac{(t-s)^2}{n}\right) \\
=&\, S^2,
\end{align*}
where
\[
S := \sum_{k=1}^{\tilde{k}} k \sum_{(i,j)\in \Lambda_{k}} \mu_{i,j}\frac{(i-j)^2}{n}.
\]

Using \eqref{e:treeexpectation}, we see that, since $\tilde{k} \leq n^{\frac{2}{3}}$, then
\begin{align}
S&=O\left(\sum_{k=1}^{\tilde{k}} ke^{-\frac{\delta k}{2}} \sum_{(i,j)\in \Lambda_{k}} \frac{(i-j)^2}{(ij)^{\frac{3}{2}}} \left(\frac{i}{j}\right)^{j-i} \exp\left(-\frac{(i-j)^2}{2n}\right) \right) \nonumber= O\left(\sum_{k=1}^{\tilde{k}} ke^{-\frac{\delta k}{2}} \sum_{i+j=k} \frac{(i-j)^2}{(ij)^{\frac{3}{2}}} \left(\frac{i}{j}\right)^{j-i} \exp\left(-\frac{(i-j)^2}{2n}\right) \right) \nonumber\\
&= O\left(\sum_{k=1}^{\tilde{k}} ke^{-\frac{\delta k}{2}} \sum_{d=-k}^k \frac{d^2}{(k^2-d^2)^{\frac{3}{2}}} \left(\frac{k-d}{k+d}\right)^{d} \exp\left(-\frac{d^2}{2n}\right) \right). \label{e:S}
\end{align}

Firstly, we note that for small $k$ the sum is negligible. Indeed, 
\begin{equation}\label{e:smallk}
\sum_{k=1}^{\epsilon^{-\frac{2}{5}}} ke^{-\frac{\delta k}{2}} \sum_{d=-k}^k \frac{d^2}{(k^2-d^2)^{\frac{3}{2}}} \left(\frac{k-d}{k+d}\right)^{d} \exp\left(-\frac{d^2}{2n}\right) \leq \sum_{k=1}^{\epsilon^{-\frac{2}{5}}} k \sum_{d=-k}^k d^2 \leq \epsilon^{-2} = O\left(\sqrt{\frac{n}{\epsilon}}\right).
\end{equation}

For $k \geq \epsilon^{-\frac{2}{5}}$, we split the inner sum up further into two ranges
\[
T_1 := \sum_{|d| \leq k^{\frac{3}{5}}} \frac{d^2}{(k^2-d^2)^{\frac{3}{2}}} \left(\frac{k-d}{k+d}\right)^{d} \exp\left(-\frac{d^2}{2n}\right) \qquad \text{and} \qquad T_2 := \sum_{k \geq |d| \geq k^{\frac{3}{5}}} \frac{d^2}{(k^2-d^2)^{\frac{3}{2}}} \left(\frac{k-d}{k+d}\right)^{d} \exp\left(-\frac{d^2}{2n}\right).
\]

By the same argument as in Lemma \ref{l:difference}, we see that, since $k = \omega(1)$,
\begin{equation}\label{e:T_1}
T_2 \lesssim \int_{k^{\frac{3}{5}}}^\infty y^2 \exp\left(-\frac{y^{\frac{1}{5}}}{2}\right) dy=O\left(e^{-\frac{1}{2}k^\frac{3}{25}}k^\frac{42}{25}\right)= o\left(k^{-\frac{5}{4}}\right).
\end{equation}

Furthermore, we can bound $T_1$ naively, using H\"{o}lder's inequality  and Lemma \ref{l:difference}, to obtain
\begin{align}
T_1 &= \sum_{|d| \leq k^{\frac{3}{5}}} \frac{d^2}{(k^2-d^2)^{\frac{3}{2}}} \left(\frac{k-d}{k+d}\right)^{d} \exp\left(-\frac{d^2}{2n}\right) \nonumber\\
&\leq \sqrt{ \sum_{|d| \leq k^{\frac{3}{5}}} d^4} \sqrt{ \sum_{|d| \leq k^{\frac{3}{5}}} \frac{1}{(k^2-d^2)^{3}} \left(\frac{k-d}{k+d}\right)^{2d} \exp\left(-\frac{d^2}{2n}\right)}\nonumber \\
&\leq  \sqrt{ \sum_{|d| \leq k^{\frac{3}{5}}} d^4} \sqrt{ \sum_{|d| \leq k^{\frac{3}{5}}} \frac{1}{(k^2-d^2)^{3}} \left(\frac{k-d}{k+d}\right)^{d} \exp\left(-\frac{d^2}{2n}\right)}\nonumber\\
&= O\left(\sqrt{k^3} \sqrt{ k^{-\frac{11}{2}}}\right)= O\left(k^{-\frac{5}{4}}\right). \label{e:T_2}
\end{align}

Therefore, by \eqref{e:T_1} and \eqref{e:T_2}, we have 
\begin{align}
   & \sum_{k=\epsilon^{-\frac{2}{5}}}^{\tilde{k}} ke^{-\frac{\delta k}{2}} \sum_{d=-k}^k \frac{d^2}{(k^2-d^2)^{\frac{3}{2}}} \left(\frac{k-d}{k+d}\right)^{d} \exp\left(-\frac{d^2}{2n}\right)=O\left(\sum_{k=\epsilon^{-\frac{2}{5}}}^{\tilde{k}}k^{-\frac{1}{4}}e^{-\frac{\delta k}{2}}\right)=O\left(\int_{y=1}^\infty y^{-\frac{1}{4}}e^{-\frac{\delta y}{2}}dy\right)\nonumber\\&=
   O\left(\epsilon^{-\frac{3}{2}} \int_{x=\frac{\epsilon^2}{4}}^{\infty} x^{-\frac{1}{4}}e^{-x} dx\right)=O\left(\epsilon^{-\frac{3}{2}}\right).\label{e:bigsumofS}
\end{align}

Hence,  by \eqref{e:S}, \eqref{e:smallk}, and \eqref{e:bigsumofS} we see that
\begin{align*}
    S=O\left(\sqrt{\frac{n}{\epsilon}}+\epsilon^{-\frac{3}{2}}\right)=O\left(\sqrt{\frac{n}{\epsilon}}\right),
\end{align*}

and so
\begin{equation}
\sum_{k_1=1}^{\Tilde{k}} \sum_{k_2=1}^{\Tilde{k}} k_1 k_2 \sum_{(i,j)\in \Lambda_{k_1}} \sum_{(s,t) \in \Lambda_{k_2}} \mu_{i,j}\mu_{s,t}\frac{(i-j)^2(t-s)^2}{n^2} = S^2 = O\left(\frac{n}{\epsilon}\right). \label{e:fourthterm}
\end{equation}

Hence, by \eqref{e:fourterms}, \eqref{e:firstterm}, \eqref{e:secondterm}, \eqref{e:thirdterm} and \eqref{e:fourthterm} we can conclude that
\begin{align}\label{e:almostthere}
\text{Var}(Z_1) &\leq \mathbb{E}(Z_{2}) + \frac{4 \epsilon}{n} \mathbb{E}(Z_2)^2 + O\left(\frac{n}{\epsilon}\right).
      \end{align}
      
Using \eqref{e:treeexpectation} and Lemma \ref{l:difference}, we can bound
       \begin{align*}
      \mathbb{E}(Z_2) &\leq  \sum_{k=1}^{n^\frac{2}{3}}k^2 \sum_{(i,j)\in \Lambda_k}\mu_{i,j} = O\left( n \sum_{k=1}^{n^\frac{2}{3}}k^2 e^{-\frac{\delta k}{2}} \sum_{(i,j)\in \Lambda_{k}} \frac{1}{(ij)^{\frac{3}{2}}} \left(\frac{i}{j}\right)^{j-i} \exp\left(-\frac{(i-j)^2}{2n}\right) \right)\\
      &= O\left( n \sum_{k=1}^{n^\frac{2}{3}}\frac{1}{\sqrt{k}}e^{-\frac{\delta k}{2}} \right)= O\left( n \int_{y=1}^{\infty}\frac{1}{\sqrt{y}}e^{-\frac{\delta y}{2}} dy \right)= O\left( \frac{n}{\sqrt{\delta}} \int_{x=\frac{\delta}{2}}^{\infty}\frac{1}{\sqrt{x}}e^{-x} dx \right)= O\left(\frac{n}{\epsilon}\right).
     \end{align*}
     
Finally, putting this together with \eqref{e:almostthere}, we can conclude that
\[
\text{Var}(Z_1) \leq O\left(\frac{n}{\epsilon}\right) + \frac{4 \epsilon}{n} O\left(\frac{n^2}{\epsilon^2}\right) + O\left(\frac{n}{\epsilon}\right) = O\left(\frac{n}{\epsilon}\right).
\]

\section{Proof of Theorem~\ref{t:subcriticalsmallcomponents}}\label{s:Appendix C}
A standard argument tells us that, for a fixed vertex $v$ the order of the component in $G(n,n,p)$ containing $v$ is stochastically dominated by the order of the component of the root in a random subgraph of $T_n$, the infinite $(n+1)$-regular rooted tree, where we include each edge independently with probability $p$.

It is shown in \cite[Corollary 3]{Kohayakawa} that if we let $t(k,n)$ be the number of subtrees of $T_n$ that contain the root and have order $k$ and $k=\omega(1)$, then
\[
t(k,n) \approx \frac{1}{\sqrt{2 \pi} k^{\frac{3}{2}}} n^{k-1} \left( \frac{n}{n-1}\right)^{k(n-1) +2}.
\]
Hence, the probability that the component of the root in a random subgraph of $T_n$ has order $k$ is given by
\begin{align*}
P_k(n,p) &=t(k,n)p^{k-1}(1-p)^{k(n-1)+2}\approx \frac{(pn)^{k-1}}{\sqrt{2 \pi} k^{1.5}}\left( \frac{n(1-p)}{n-1}\right)^{k(n-1) +2}.
\end{align*}
It follows that, if we let $p = \frac{1-\epsilon}{n}$, then
\begin{align*}
P_k(n,p) &\approx \frac{1}{\sqrt{2 \pi} k^{1.5}} (1 - \epsilon)^{k-1}\left( 1 + \frac{\epsilon}{n-1}\right)^{{k(n-1) +2}}\leq  \frac{1}{\sqrt{2 \pi} k^{1.5}} (1 - \epsilon)^{k-1} \text{exp}\left(\epsilon k + O\left(\frac{\epsilon}{n}\right)\right)\lesssim \frac{1}{\sqrt{2 \pi} k^{1.5}}\left( (1 - \epsilon)e^{\epsilon}\right)^{k-1}.
\end{align*}

Furthermore, it is clear by comparison with a branching process that with probability $1$ the component of the root is finite, and hence $\sum_k P_k(n,p) = 1$. It follows that the probability that a vertex in $G(n,n,p)$ belongs to a component of order larger than $k_0 \in \mathbb{N}$ is equal to
\[
\sum_{k \geq k_0} P_k(n,p).
\]
Hence, if we let $Y_{\geq k_0}$ be the number of vertices in $G(n,n,p)$ which belong to a component of order larger than $k_0$, then we have that
\begin{align*}
\mathbb{E}(Y_{\geq k_0}) &= n \sum_{k \geq k_0} P_k(n,p)\lesssim  n \sum_{k \geq k_0}\frac{1}{\sqrt{2 \pi} k^{1.5}}\left( (1 - \epsilon)e^{\epsilon}\right)^{k-1}\lesssim n k_0^{-\frac{3}{2}} \frac{ ((1-\epsilon)e^{\epsilon})^{k_0}}{1- (1-\epsilon)e^{\epsilon}}.
\end{align*}
However, $(1-\epsilon)e^{\epsilon} = 1 - \frac{\epsilon^2}{2} + O(\epsilon^3)$ and so 
\[
\mathbb{E}(Y_{\geq k_0}) \lesssim  n k_0^{-\frac{3}{2}} \frac{4}{\epsilon^2}.
\]

Taking $k_0 = \sqrt{\frac{n}{3 \epsilon}}$, we see that 
\[
\mathbb{E}\left(Y_{\geq \sqrt{\frac{n}{3 \epsilon}}}\right) = O\left( \frac{n^{\frac{1}{4}}}{\epsilon^{\frac{5}{4}}}\right) = o\left( \sqrt{\frac{n}{\epsilon}} \right).
\]	
\end{appendix}
\end{document}